\documentclass[10pt,reqno]{amsart}
\numberwithin{equation}{section}
\usepackage{verbatim}
\usepackage{amssymb}
\usepackage{enumerate, xspace}

\usepackage[charter]{mathdesign}

\usepackage{xcolor}
\usepackage{pstricks,pst-node}
\usepackage{tikz}
\usetikzlibrary{backgrounds}
\usetikzlibrary{patterns,fadings}
\usetikzlibrary{arrows,decorations.pathmorphing}
\usetikzlibrary{shapes}
\usetikzlibrary{calc}

\usepackage{enumerate}

\definecolor{brightcerulean}{rgb}{0.11, 0.67, 0.84}
\definecolor{cerulean}{rgb}{0.0, 0.48, 0.65}
\definecolor{Gray}{rgb}{0.5, 0.5, 0.5}

\newtheorem{thm}{Theorem}[section]
\newtheorem{lem}[thm]{Lemma}
\newtheorem{cor}[thm]{Corollary}

\newtheorem{prop}[thm]{Proposition}

\newtheorem{definition}[thm]{Definition}

\newtheorem{rem}[thm]{Remark}

\newcommand\bZ{{\mathbb Z}}

\newcommand\ve{\varepsilon}

\newcommand\vf{\varphi}

\newcommand\EE{{\mathbb E}}
\newcommand\PP{{\mathbb P}}

\newcommand\RR{{\mathbb R}}

\newcommand\ZZ{{\mathbb Z}}

\newcommand{\mc}[1]{{\mathcal #1}}

\newcommand{\bb}[1]{{\mathbb #1}}
\newcommand{\p}{\prime}

\begin{document}

\title[Symmetric exclusion with long jumps]{Slow to fast infinitely extended reservoirs for the symmetric exclusion process with long jumps}
\author{C\'edric Bernardin}
\address{Universit\'e C\^ote d'Azur, CNRS, LJAD\\
Parc Valrose\\
06108 NICE Cedex 02, France}
\email{{\tt cbernard@unice.fr}}

\author{P.  Gon\c calves}
\address{Center for Mathematical Analysis,  Geometry and Dynamical Systems,
Instituto Superior T\'ecnico, Universidade de Lisboa,
Av. Rovisco Pais, 1049-001 Lisboa, Portugal and   Institut  Henri
Poincar\'e, UMS 839 (CNRS/UPMC), 11 rue Pierre et Marie Curie, 75231 Paris Cedex 05, France.}
\email{{\tt patricia.goncalves@math.tecnico.ulisboa.pt}}

\author{B. Jim\'enez Oviedo}
\address{Universit\'e C\^ote d'Azur, CNRS, LJAD\\
Parc Valrose\\
06108 NICE Cedex 02, France}
\email{{\tt byron@unice.fr}}

\thanks{}

\date{\today.}
\begin{abstract}
We consider an exclusion process with long jumps  in the box $\Lambda_N=\{1, \ldots,N-1\}$, for $N \ge 2$, in contact with infinitely extended reservoirs on its left and on its right. The jump rate is described by a transition probability $p(\cdot)$ which is symmetric, with infinite support but with finite variance.  The reservoirs add or remove particles with rate proportional to $\kappa N^{-\theta}$, where $\kappa>0$ and $\theta \in\bb R$. If $\theta>0$ (resp. $\theta<0$)  the reservoirs add and fastly remove  (resp. slowly remove) particles in the bulk. According to the value of $\theta$ we prove that the time evolution of the spatial density of particles is described by some reaction-diffusion equations with various boundary conditions.

\end{abstract}

\keywords{Hydrodynamic limit, Reaction-diffusion equation, Boundary conditions, Exclusion with long jumps.} 

\maketitle
 
\section{Introduction}
The exclusion process is an interacting particle system introduced in the mathematical literature during the seventies by Frank Spitzer \cite{Spitzer}. Despite the simplicity of its dynamics it captures the main features of more realistic diffusive systems driven out of equilibrium \cite{Li}, \cite{Liggett2}, \cite{Spohn}. It consists in a collection of continuous-time  random walks evolving on the lattice $\ZZ$ whose dynamics  can be described as follows. A particle at the site $x$ waits an exponential time  after which it jumps to a site $x+y$ with probability  $p(y)$.  If, however, if $x+y$ is already occupied, the jump is suppressed and the clock is reset.

Recently a series of work have been devoted to the study of the nearest-neighbor exclusion process whose dynamics is perturbed by the presence of a slow bond \cite{FGN}, a slow site \cite{FGS},  by slow boundary effects \cite{Adriana} and current boundary effects \cite{dmfp, dmptv,dmptv_2,dmptv_3}. The behavior of the system is then strongly affected and new boundary conditions may be derived at the macroscopic level. On the other hand it is known that the presence of long jumps, in particular heavy tailed long jumps, have a drastic effect on the macroscopic behavior and critical exponents of the system \cite{BGS, Jara1, Jara2}. In this work, we propose to mix these two interesting features by considering the symmetric exclusion process with long jumps in contact with extended reservoirs. The coupling with the reservoirs is regulated by a certain power $\theta$ of a scaling parameter which is the inverse of the size system $N \to\infty$. This question has been addressed in a recent paper \cite{Adriana} in the case of the nearest-neighbor exclusion process for a positive power $\theta$ and with finite reservoirs, in fact one at each end point. Here we consider the case where the jumps probability transition $p(z) \sim |z|^{-1-\gamma}$ has an infinite support and the power $\theta$ has an arbitrary sign, so that the boundary effects can be very strong (fast) or very weak (slow). The model of reservoirs chosen is the same as in \cite{SNU} but other choices are possible and we discuss some of them in Section \ref{subsec:complements}. It would be interesting  to consider the boundary dynamics as in \cite{ dmptv,dmptv_2,dmptv_3}, where particles can be injected (resp. removed) at a fixed rate in an interval close to the right (resp. left) boundary. Then, at the macroscopic level the system should exhibit Robin boundary conditions, which, depending on the range of the interval, could be linear or non-linear, has happens in the nearest-neighbor case.  In this paper we will focus only on the case $\gamma>2$, so that $p(\cdot)$ has a finite variance, postponing the study of the case $\gamma \le 2$ for future works \cite{bgjo}. The form of the reservoirs chosen makes the model a case of the general class of superposition of a dynamics of Glauber type with simple exclusion (see the seminal paper \cite{DMFL} and \cite{BL}, \cite{KMS} for more recent studies) but with a possible singular reaction term due to the long jumps.

The problem we address is to characterize the hydrodynamic behavior of the process described above, i.e., to deduce the macroscopic behavior of the system from the microscopic interaction among particles and to analyze the effect of slowing down or fasting up the interaction with the reservoirs, by increasing or decreasing the value of $\theta$, at the level of the macroscopic profiles of the density. Usually the characterization  of the hydrodynamic limit is formulated in terms of a weak solution of some partial differential equation, called the \textit{hydrodynamic equation}. Depending on the intensity of the coupling with the reservoirs we will observe a phase transition for profiles which are solutions of the hydrodynamic equation which consists on reaction-diffusion equations with different types of boundary conditions, depending on the range of the parameter $\theta$.

We extend the results for the nearest neighbor symmetric simple exclusion process with slow boundaries that was studied in \cite{Adriana} by considering long jumps, infinitely extended reservoirs and also fast reservoirs, i.e. $\theta<0$. In the case $\theta \ge 0$ (slow reservoirs) we recover in our model a similar hydrodynamical behavior to the one obtained in \cite{Adriana}, since we imposed that the probability transition rate to be symmetric and with  finite variance. If one of these conditions is violated then the macroscopic  behavior of the system is different. In the case where we drop the hypothesis that $p(\cdot)$ is symmetric, then there is a drift in the microscopic system which appears at the macroscopic level as the heat equation with a transport term and if drop the finite variance condition, then we expect to have the usual laplacian for the case $p(z) \sim |z|^{-1-\gamma}$ with $\gamma=2$ and a fractional operator  when $\gamma\in(1,2)$, see \cite{bjo}. We leave this difficult problem for a future work since it is important to well understand the ``normal" case first. 

When $\theta$ ranges from $-\infty$ to $+\infty$, the model produces five different macroscopic phases, depending on the value of the parameter $\theta.$ If $\theta \in (2-\gamma,1)$, the boundary interactions are not slowed or fasted enough in order to  change the macroscopic  behavior of the system so that we observe exactly the same behavior as in the case $\theta=0 \in (2-\gamma, 1)$. The hydrodynamic equation in this case is the heat equation with Dirichlet boundary conditions. If $\theta = 1$, the reservoirs are slowed enough that we obtain the heat equation but with Robin boundary conditions. For $\theta \in (1,\infty)$, the reservoirs are sufficiently slowed so that we get the heat equation with Neumann boundary conditions. If $\theta= 2 -\gamma$, the reservoirs are fast enough that we obtain the heat equation with a singular reaction term at the boundaries but with Dirichlet boundary conditions. If $\theta<2-\gamma$, the reservoirs are so fasted that the diffusion part of the motion disappears and that only the reaction term survives at the macroscopic level. The two cases $\theta=1$ and $\theta=2-\gamma$ correspond to a critical behavior connecting macroscopically two different regimes (Dirichlet boundary conditions to Neumann boundary conditions for $\theta=1$  and Reaction to Diffusion equation for $\theta=2-\gamma$). Once the form of the hydrodynamic equation is obtained, it is of interest to study its stationary solution which provides the density profile in the stationary state in the thermodynamic limit. In particular for $\theta \le 2- \gamma$ the density profiles are non linear and have nice properties (see Figure \ref{fig:mesh1}). It would be of interest to go further in the study of the non-equilibrium stationary states of this models.

The paper is organized as follows. In Section \ref{sec:model}  we describe precisely the model  and we state the main result. In Section \ref{sec:hyd_eq} we present the hydrodynamic equations and in Section \ref{sec:HL} we state the \textit{Hydrodynamic Limit}. In Section \ref{subsec:complements} we complement our results in the case of other models of reservoirs.  In order to give an intuition for getting the different boundary conditions, we present in Section \ref{sec:CL} the heuristics for obtaining the  weak solutions of the corresponding partial differential equations. This result is rigorously proved in  Section \ref{sec:Characterization of limit points}. We prove tightness in Section \ref{sec:Tightness}. In Section \ref{sec:RL}, we prove some \textit{Replacement Lemmas} and some auxiliary results. In Section \ref{sec:Energy} we establish some energy estimates which are fundamental to establish uniqueness of the hydrodynamic equations. We added the Appendix \ref{sec:app-unique} in which we prove  the uniqueness of weak solutions of the hydrodynamics equations and the Appendix \ref{sec:compgen} which contains computations involving the generator of the dynamics.

\section{Statement of results}
\subsection{The model}
\label{sec:model}
For $N\geq{2}$ let  $\Lambda_N=\{1, \ldots, N-1\}$ be a finite lattice of size $N-1$ called the bulk.  The exclusion process in contact with reservoirs is a Markov process $\{\eta_t:\,t\geq{0}\}$ with state space $\Omega_N:=\{0,1\}^{\Lambda_N}$. The configurations of the state space $\Omega_N$  are denoted by $\eta$, so that for $x\in\Lambda_N$,  $\eta(x)=0$ means that the site $x$ is vacant while $\eta(x)=1$ means that the site $x$ is occupied.    Now, we explain the dynamics of this model and we start by describing the conditions on the  jump rate.  For that purpose,  let $p:\mathbb{Z}\rightarrow{[0,1]}$ be a translation invariant transition probability 
which is symmetric, that is, for any $z\in\mathbb Z$, $p(z)=p(-z)$ and with finite  variance, that is  $\sigma^{2}:=\sum_{z\in \ZZ}z^{2}p(z)<\infty.$ Note that since $p(\cdot)$ is symmetric it is mean zero, that is:
$\sum_{z\in \ZZ}z p(z)=0.$
We denote $m=\sum_{z\ge 1} z p(z)$. As an example we consider $p(\cdot)$ given by $p(0)=0$ and  $p(z) = 
\dfrac{c_{\gamma}}{\vert z\vert^{\gamma+1}},$ for $z\neq 0$, 
where  $c_{\gamma}$ is a normalizing constant and $\gamma> 2$, so that $p(\cdot)$ has finite variance.  

We consider the process in contact with infinitely many stochastic  reservoirs at all the negative integer sites and at all the integer sites $z \geq N$. We fix four parameters $\alpha, \beta\in(0,1)$, $\kappa>0$ and   $\theta \in \mathbb R$. Particles can get into (resp. exit) the bulk of the system from any site at the left of $0$ at rate  $\alpha\kappa/N^\theta p(z)$   (resp. $(1-\alpha)\kappa/N^\theta p(z)$), where $z$ is the jump size (see Figure \ref{fig:epwlj}).The stochastic reservoir at the right acts in the same way as the left reservoir but in the intensity we replace  $\alpha$ by $\beta$.

The dynamics of the process is defined as follows.  We start with the bulk dynamics. Each pair of sites of the bulk $\{x,y\} \subset \Lambda_N$ carries a Poisson process of intensity one. The Poisson processes associated to different bonds are independent. If for the configuration $\eta$, the clock associated to the bound $\{x,y\}$ rings, then we exchange the values $\eta_x$ and $\eta_y$ with rate $p(y-x) /2$.
Now we explain the dynamics at the boundary. Each pair of sites $\{x,y\}$ with $x\in\Lambda_N$ and $y\in\mathbb Z_-$ carries a Poisson process of intensity one all being independent. If for the configuration $\eta$, the clock associated to the bound $\{x,y\}$ rings, then we change the values $\eta_x$ into $1-\eta_x$ with rate $\frac{\kappa}{N^\theta}p(x-y)\, [(1-\alpha) \eta_x  + \alpha (1-\eta_x) ]$. At the right boundary the dynamics is similar but instead of $\alpha$ the intensity is given by $\beta$. Observe that the reservoirs add and remove particles on all the sites of the bulk $\Lambda_N$, and not only at the boundaries, but with rates which decrease as the distance from the corresponding reservoir increases. We can interpret  the dynamics of the reservoirs in  two different ways as follows. In the  first case, we add to the bulk infinitely many reservoirs at all negative sites and at all sites $y\geq N$. Then particles can get into (resp. get out from) the bulk from  the left reservoir at rate $\alpha \kappa/N^\theta p(z)$ (resp. $(1-\alpha) \kappa/N^\theta p(z)$) where $z$ is the size of the jump. The right reservoir acts in the same way, except that we replace $\alpha$ by $\beta$ in the jump rates given above.
In the second case we can consider that particles can be created (resp. annihilated) at all the  sites $x$ in the bulk with one of the rates $r_N^-(x/N)\alpha \kappa/N^\theta$ or $r_N^+(x/N)\beta \kappa/N^\theta$ (resp. $r_N^-(x/N)(1-\alpha) \kappa/N^\theta$ or $r_N^+(x/N)(1-\beta) \kappa/N^\theta$) where $r_N^{\pm}$ are given in \eqref{eq:mcln}.\\

\begin{center}
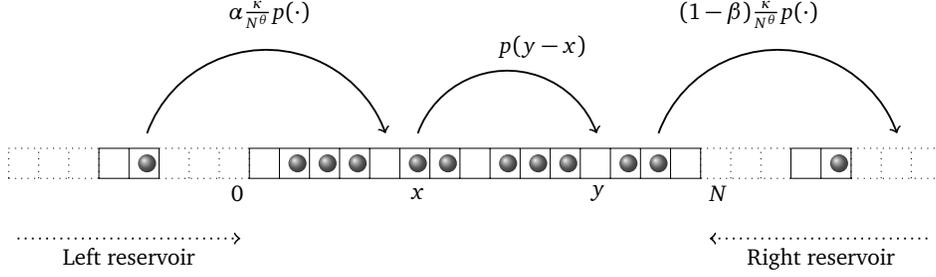
\begin{figure}
\begin{tikzpicture}[scale=0.4]





\draw[step=1cm,black,very thin] (0,0) grid (15,1);

\draw[step=1cm,black, dotted] (15,0) grid (18,1);
\draw[step=1cm,black, very thin] (18,0) grid (20,1);
\draw[step=1cm,black, dotted] (20,0) grid (23,1);

\draw[step=1cm,black, dotted] (-3,0) grid (0,1);
\draw[step=1cm,black, very thin] (-5,0) grid (-3,1);
\draw[step=1cm,black, dotted] (-8,0) grid (-5,1);

\foreach \y in {2,3,4}{
\shade[ball color=gray] (\y-0.4,0.5) circle (.3);
}
\foreach \y in {6,7}{
\shade[ball color=gray] (\y-0.4,0.5) circle (.3);
}
\foreach \y in {9,10,11}{
\shade[ball color=gray] (\y-0.4,0.5) circle (.3);
}
\foreach \y in {13,14}{
\shade[ball color=gray] (\y-0.4,0.5) circle (.3);
}

\shade[ball color=gray] (-3-0.4,0.5) circle (.3);
\shade[ball color=gray] (20-0.4,0.5) circle (.3);

\draw (6-0.4,-0.5) node [color=black] {\small $x$};
\draw (12-0.4,-0.5) node [color=black] {\small{$y$}};
\draw (16-0.4,-0.5) node [color=black] {\small{$N$}};
\draw (0-0.4,-0.5) node [color=black] {\small{$0$}};

\node (LL) at (-8,-2){};
\node (LR) at (0,-2){};
\path[draw,->,color=black, dotted, line width=0.7] (LL) edge node [midway, sloped, below] {\small{Left reservoir}} (LR);

\node (RL) at (15,-2){};
\node (RR) at (23,-2){};
\path[draw,<-,color=black, dotted, line width=0.7] (RL) edge node [midway, sloped, below] {\small{Right reservoir}} (RR);

\draw [->,color=black, line width=0.7] (6-0.4,1.5) arc (160:20:90pt);
\draw (8,3.7) node [color=black, above right] {\small{$p(y-x)$}};

\draw [->,color=black, line width=0.7] (-3-0.4,1.5) arc (160:20:120pt);
\draw (-1,4.7) node [color=black, above right] {\small{$\alpha \tfrac{\kappa}{N^{\theta}} p(\cdot)$}};

\draw [->,color=black, line width=0.7] (14-0.4,1.5) arc (160:20:120pt);
\draw (14,4.7) node [color=black, above right] {\small{$(1-\beta) \tfrac{\kappa}{N^{\theta}} p(\cdot)$}};

\end{tikzpicture}
\caption{Exclusion process with long jumps and infinitely extended reservoirs.}
\label{fig:epwlj}
\end{figure}
\end{center}

The infinitesimal generator of the  process is given by  
\begin{equation}
\label{Generator}
 L_{N} = L_{N}^{0}+ L_{N}^{r}+L_{N}^{\ell},
\end{equation}
where its action on functions  $f:\Omega_N \to \RR$ is
\begin{equation}\label{generators}
\begin{split}
&(L^0_N f)(\eta) =\cfrac{1}{2} \, \sum_{x,y \in \Lambda_N} p(x-y) [ f(\sigma^{x,y}\eta) -f(\eta)],\\
&(L_N^{\ell} f)(\eta) = \frac{\kappa}{N^\theta}\sum_{\substack{x \in \Lambda_N\\ y \le 0}} p(x-y)c_{x}(\eta;\alpha) [f(\sigma^x\eta) - f(\eta)],\\
&(L_N^{r} f)(\eta)= \frac{\kappa}{N^\theta}\sum_{\substack{x \in \Lambda_N \\ y \ge N}} p(x-y) c_{x}(\eta;\beta)  [f(\sigma^x\eta) - f(\eta)],
\end{split}
\end{equation}
 and 
\begin{equation}\label{tranformations}
(\sigma^{x,y}\eta)_z = 
\begin{cases}
\eta_z, \; z \ne x,y,\\
\eta_y, \; z=x,\\
\eta_x, \; z=y
\end{cases}
, \quad (\sigma^x\eta)_z= 
\begin{cases}
\eta_z, \; z \ne x,\\
1-\eta_x, \; z=x.
\end{cases}
\end{equation}
Above, for a function $\vf:[0,1]\rightarrow \RR$, we used the notation 
\begin{equation}\label{rate_c}
c_{x} (\eta;\vf(\cdot)) :=\left[ \eta_x  \left(1-\vf(\tfrac{x}{N}) \right) + (1-\eta_x)\vf(\tfrac{x}{N})\right].
\end{equation}
%

We consider  the Markov process  speeded up in the  time scale $\Theta(N)$ and we use the notation $\eta^{N} (t):= \eta( t \Theta(N))$, so that $(\eta^N (t))_{t\ge 0} $ has  infinitesimal generator $\Theta(N)L_{N}$. Although $\eta ^{N} (t)$ depends on $\alpha$, $\beta$ and $\theta$, we shall omit these index in order to simplify notation.


\subsection{Hydrodynamic equations} \label{sec:hyd_eq}
\label{sec:wsfhe}

From now on up to the rest of this article we fix a finite time horizon $[0,T]$.  To properly  state the hydrodynamic limit, we need to introduce some notations and definitions. We denote by $\langle \cdot,\cdot\rangle _{\mu}$ (resp. $\| \cdot\|_{L^2 (\mu) }$) the inner product (resp. the norm) in $L^{2}([0,1])$ with respect to the measure $\mu$ defined in $[0,1]$ and  when $\mu$ is the Lebesgue measure we simply write $\langle \cdot,\cdot\rangle$ and $\| \cdot\|_{L^2}$ for the corresponding norm. For  an interval $\mc I$ in $\RR$ and integers
$m$ and $n$, we denote by $C^{m,n}([0, T] \times \mc I)$ the set of functions defined on $[0, T] \times \mc I $ that are $m$ times differentiable on the first variable and $n$ times differentiable  on the second variable. An index on a function will always denote a
fixed variable, not a derivative. For example, $G_{s}(q)$ means $G(s, q)$. The derivative of $G \in C^{m,n}([0, T] \times \mc I)$ will be denoted by $\partial _{s}G$ (first variable) and $\partial _{q}G$ (second variable). We shall write $\Delta G$ for $\partial_{q}^{2}G$. We also  consider the set  $C^{m,n}_c ([0,T] \times [0,1])$ of functions $G \in C^{m,n}([0, T] \times[0, 1])$  such that $G_s$ has a compact support included in $(0,1)$ for any time $s$ and, we denote by $C_{c}^{m}(0,1)$ (resp. $C_c^\infty (0,1)$) the set of all $m$ continuously differentiable (resp. smooth) real-valued  functions defined on $(0,1)$ with compact support. The set $C^\infty ([0,1])$ denotes the set of restrictions of smooth functions on $\RR$ to the interval $[0,1]$. The supremum norm is denoted by $\| \cdot \|_{\infty}$.

The semi inner-product $\langle \cdot, \cdot \rangle_{1}$ is defined on the set $C^{\infty} ([0,1])$ by 
\begin{equation}
\langle G, H \rangle_{1} = \int_{0}^1 (\partial_q G)(q) \, (\partial_q H)(q)  \, dq.
\end{equation}  
The corresponding semi-norm is denoted by $\| \cdot \|_{1}$. 

\begin{definition}
\label{Def. Sobolev space}
The Sobolev space $\mathcal{H}^{1}$ on $[0,1]$ is the Hilbert space defined as the completion of $C^\infty ([0,1])$ for the norm 
$$\| \cdot\|_{{\mc H}^1}^2 :=  \| \cdot \|_{L^2}^2  +  \| \cdot \|^2_{1}.$$
Its elements elements coincide a.e. with continuous functions. 
The completion of $C_c^{\infty} (0,1)$ for this norm is denoted by ${\mc H}_0^{1}$. This is a Hilbert space whose elements coincide a.e. with continuous functions vanishing at $0$ and $1$. On ${\mc H}_0^{1}$, the two norms $ \| \cdot \|_{{\mc H}^{1}}$ and  $\| \cdot \|_{1}$ are equivalent.   
The space $L^{2}(0,T;\mathcal{H}^{1})$ is the set of measurable functions $f:[0,T]\rightarrow  \mathcal{H}^{1}$ such that 
$$\int^{T}_{0} \Vert f_{s} \Vert^{2}_{\mathcal{H}^{1}}ds< \infty. $$
The space $L^{2}(0,T;\mathcal{H}_0^{1})$ is defined similarly.
\end{definition}

%

We can now give the definition of the weak solutions of the  hydrodynamic equations that will be derived in this paper.

\begin{definition}
\label{Def. Dirichlet source Condition-g}
Let   $\hat \sigma \ge 0$ and $\hat \kappa \ge 0$ be some parameters. Let $g:[0,1]\rightarrow [0,1]$ be a measurable function. We say that  $\rho:[0,T]\times[0,1] \to [0,1]$ is a weak solution of the reaction-diffusion equation with inhomogeneous Dirichlet boundary conditions
 \begin{equation}\label{eq:Dirichlet source Equation-g}
 \begin{cases}
 &\partial_{t}\rho_{t}(q)=\frac{{\hat \sigma}^2}{2}\Delta\, {\rho} _{t}(q)+ {\hat \kappa} \Big\{ \frac{\alpha-\rho_t(q)}{q^\gamma}+\frac{\beta-\rho_t(q)}{(1-q)^\gamma}\Big\}, \quad (t,q) \in [0,T]\times(0,1),\\
 &{ \rho} _{t}(0)=\alpha, \quad { \rho}_{t}(1)=\beta,\quad t \in [0,T], \\
 &{ \rho}_{0}(\cdot)= g(\cdot),
 \end{cases}
 \end{equation}
if the following three conditions hold:
\begin{enumerate}[1.]
\item $\rho \in L^{2}(0,T;\mathcal{H}^{1})$ if $\hat \sigma >0$ and $\int_0^T \int_0^1 \Big\{ \frac{(\alpha-\rho_t(q))^2}{q^\gamma}+\frac{(\beta-\rho_t(q))^2}{(1-q)^\gamma}\Big\} \, dq\, dt <\infty$ if $\hat \kappa >0$, 
\item $\rho$ satisfies the weak formulation:
\begin{equation}\label{eq:Dirichlet_source_ integral-g}
\begin{split}
&F_{RD}(t, \rho,G,g):=\int_0^1 \rho_{t}(q)  G_{t}(q) \,dq  -\int_0^1 g(q)   G_{0}(q) \,dq \\
&- \int_0^t\int_0^1 \rho_{s}(q)\Big(\dfrac{\hat \sigma^{2}}{2}\Delta + \partial_s\Big) G_{s}(q)  \,ds \, dq\\
&- {\hat \kappa}\int_0^t\int_0^1G_s(q)\left( \frac{\alpha-\rho_s(q)}{q^\gamma}+\frac{\beta-\rho_s(q)}{(1-q)^\gamma}\right)\, ds\,dq=0,
\end{split}   
\end{equation}
for all $t\in [0,T]$ and any function $G \in C_c^{1,2} ([0,T]\times[0,1])$, 

\item if $\hat \sigma >0$ and $\hat\kappa=0$, then $\rho _{t}(0)=\alpha$ and $ {\rho}_{t}(1)=\beta$
for $t$ a.s  in $[0,T]$.
\end{enumerate}
\end{definition}

\begin{rem}
Observe that in the case $\hat \sigma >0$ and $\hat \kappa=0$ we recover the heat equation with Dirichlet inhomogeneous boundary conditions. If $\hat \sigma =0$ the equation does not have a diffusion part and the solution is fully explicit. Despite in the weak formulation we do not require any boundary condition (except the second part of item 1) nor any regularity assumption, it turns out that the (unique) weak solution is smooth and satisfies the boundary conditions of item 3.   
\end{rem}

\begin{rem}\label{use:rem_dir}

Observe that in the case $\hat \sigma >0$ and $\hat \kappa>0$ the item 1 of the previous definition implies that $\rho_t (0)=\alpha $ and $\rho_{t} (1) =\beta$, for almost every $t$ in $[0,T]$. Indeed, first note that by item 1 we know that $\rho_t$  is  $\tfrac{1}{2}$-H\"older for almost every $t$ in $[0,T]$ since a function in $\mc H^1$ is $\tfrac{1}{2}$-H\"older. Now, taking $\ve \in (0,1)$ we note that 
\begin{equation}\label{eq:Rve}
 \int_{0}^{T} \dfrac{(\rho_{t} (0) - \alpha)^{2}}{\gamma-1}dt =\int_{0}^{T}\lim _{\ve\to 0}\ve ^{\gamma-1}\int_{\ve}^{1} \dfrac{(\rho_{t} (0) - \alpha)^{2}}{q^{\gamma}}dq dt. 
\end{equation}
By summing and subtracting $\rho_t (u)$ inside the square in the expression on the right hand side in \eqref{eq:Rve} and using the inequality $(a+b)^2\leq 2a^{2} +2b^{2}$ we get that \eqref{eq:Rve} is bounded from above by
\begin{equation}
\label{eq:Rve1}
\begin{split}
& 2\int_{0}^{T}\lim _{\ve\to 0}\ve ^{\gamma-1}\int_{\ve}^{1} \dfrac{(\rho_{t} (0) - \rho_{t}(q))^{2}}{q^{\gamma}}dq dt 
+2\int_{0}^{T}\lim _{\ve\to 0}\ve ^{\gamma-1}\int_{\ve}^{1} \dfrac{(\rho_{t} (q) - \alpha)^{2}}{q^{\gamma}}dq dt.
\end{split}
\end{equation}
Since $\rho_t$  is  $\tfrac{1}{2}$-H\"older for almost every $t$ in $[0,T]$  the first term  in \eqref{eq:Rve1}  vanishes. Now, the second term in \eqref{eq:Rve1} is bounded from above by
$$2\lim _{\ve\to 0}\ve ^{\gamma-1}\int_{0}^{T}\int_{0}^{1} \dfrac{(\rho_{t} (q) - \alpha)^{2}}{q^{\gamma}}dq dt,$$
which vanishes since we know by the second claim of item 1 that 
$\int_{0}^{T}\int_{0}^{1} \dfrac{(\rho_{t} (q) - \alpha)^{2}}{q^{\gamma}}dq dt <\infty$.
 Thus, we have that 
 $$\int_{0}^{T} \dfrac{(\rho_{t} (0) - \alpha)^{2}}{\gamma-1}dt =0,$$
 whence we get that  $\rho_t (0)=\alpha $ for almost every $t$ in $[0,T]$. Showing that $\rho_t (1)=\beta$ for almost every $t$ in  $[0,T]$ is completely analogous.

\end{rem}

\begin{definition}
\label{Def. Robin Condition-g}
Let $\hat \sigma>0$ and $\hat m\ge 0$ be some parameters.  Let $g:[0,1]\rightarrow [0,1]$ be a measurable function. We say that  $\rho:[0,T]\times[0,1] \to [0,1]$ is a weak solution of the heat equation with Robin boundary conditions 
 \begin{equation}\label{Robin Equation-g}
 \begin{cases}
 &\partial_{t}\rho_{t}(q)= \frac{\hat \sigma^2}{2}\Delta\, {\rho} _{t}(q), \quad (t,q) \in [0,T]\times(0,1),\\
 &\partial_{q}\rho _{t}(0)=\frac{2\hat m}{\hat \sigma^2}(\rho_{t}(0) -\alpha),\quad \partial_{q} \rho_{t}(1)=\tfrac{2\hat m}{\hat \sigma^2}(\beta -\rho_{t}(1)),\quad t \in [0,T] \\
 &{ \rho}_{0}(\cdot)= g(\cdot),
 \end{cases}
 \end{equation}
if the following three conditions hold: 
\begin{enumerate}[1.]
\item $\rho \in L^{2}(0,T;\mathcal{H}^{1})$, 

\item $\rho$ satisfies the weak formulation:
\begin{equation}\label{eq:Robin integral-g}
\begin{split}
&F_{Rob}(t, \rho,G,g):=\int_0^1 \rho_{t}(q)  G_{t}(q) \,dq  -\int_0^1 g(q)   G_{0}(q) \,dq \\
& - \int_0^t\int_0^1 \rho_{s}(q) \Big(\dfrac{\hat \sigma^{2}}{2}\Delta + \partial_s\Big) G_{s}(q)  \,ds\, dq + \dfrac{\hat \sigma^{2}}{2}\int^{t}_{0}   \{\rho_{s}(1) \partial_q G_{s}(1)-\rho_{s}(0)  \partial_q G_{s}(0) \} \, ds\\
& \qquad-\hat m \int^{t}_{0} \{ G_{s}(0)(\alpha -\rho_{s}(0)) +  G_{s}(1)(\beta -\rho_{s}(1)) \}\,  ds=0,
\end{split}   
\end{equation}
for all $t\in [0,T]$, any function $G \in C^{1,2} ([0,T]\times[0,1])$. 
\end{enumerate}
\end{definition}

\begin{rem} \label{neumann_cond_rem}
Observe that in the case $\hat m =0$ the PDE above is the heat equation with Neumann boundary conditions. 
\end{rem}

\subsection{Hydrodynamic Limit}
\label{sec:HL}
%
Let ${\mc M}^+$ be the space of positive measures on $[0,1]$ with total mass bounded by $1$ equipped with the weak topology. For any configuration  $\eta \in \Omega_{N}$ we define the empirical measure $\pi^{N}(\eta,dq)$ on $[0,1]$ by 
\begin{equation}\label{MedEmp}
\pi^{N}(\eta, dq)=\dfrac{1}{N-1}\sum _{x\in \Lambda_{N}}\eta_{x}\delta_{\frac{x}{N}}\left( dq\right),
 \end{equation}
where $\delta_{a}$ is a Dirac mass on $a \in [0,1]$, and
$$\pi^{N}_{t}(\eta, dq):=\pi^{N}(\eta^N(t), dq).$$

Fix $T>0$ and $\theta\in \RR$. We denote by $\PP _{\mu _{N}}$ the probability measure in the Skorohod space $\mathcal D([0,T], \Omega_N)$ induced by the  Markov process $(\eta^N (t))_{t\ge 0}$ and the initial probability measure $\mu_N$ and we denote by $\EE _{\mu _{N}}$ the expectation with respect to $\PP_{\mu _{N}}$.  Let $\lbrace\mathbb{Q}_{N}\rbrace_{N\geq 1}$ be the  sequence of probability measures on $\mathcal D([0,T],\mathcal{M}^{+})$ induced by the  Markov process $\lbrace \pi_{t}^{N}\rbrace_{t\geq 0}$ and by $\mathbb{P}_{\mu_{N}}$.

Let $\rho_0: [0,1]\rightarrow[0,1]$ be a measurable function. We say that a sequence of probability measures $\lbrace\mu_{N}\rbrace_{N\geq 1 }$ in $\Omega_{N}$  is associated to the profile $\rho_{0}(\cdot)$ if for any continuous function $G:[0,1]\rightarrow \mathbb{R}$  and every $\delta > 0$ 
\begin{equation}\label{assoc_mea}
  \lim _{N\to\infty } \mu _{N}\left( \eta \in \Omega_{N} : \left\vert \dfrac{1}{N}\sum_{x \in \Lambda_{N} }G\left(\tfrac{x}{N} \right)\eta_{x} - \int_{0}^1G(q)\rho_{0}(q)dq \right\vert    > \delta \right)= 0.
\end{equation}

The main result of this article is summarized in the following theorem (see Figure \ref{fig:hlf}).

\begin{thm}
\label{th:hl900}
 Let $g:[0,1]\rightarrow[0,1]$ be a measurable function and let $\lbrace\mu _{N}\rbrace_{N\geq 1}$ be a sequence of probability measures in $\Omega_{N}$ associated to $g(\cdot)$. Then, for any $0\leq t \leq T$,
\begin{equation*}\label{limHidreform}
 \lim _{N\to\infty } \PP_{\mu _{N}}\Big( \eta_{\cdot}^{N} \in \mathcal D([0,T], {\Omega_{N}}) : \left\vert \dfrac{1}{N}\sum_{x \in \Lambda_{N} }G\left(\tfrac{x}{N} \right)\eta^N_{x}(t) - \int_{0}^1G(q)\rho_{t}(q)dq \right\vert    > \delta \Big)= 0,
\end{equation*}
where  the time scale is given by
 \begin{equation}\label{time_scales}
\Theta(N)= \begin{cases}
 N^2, &\quad\textrm{if} \,\,\,  \theta\geq 2-\gamma,\\
 N^{\gamma+\theta}, &\quad   \textrm{if} 
\,\,\, \theta<2-\gamma,\\
 \end{cases}
 \end{equation}
and $\rho_{t}(\cdot)$ is the unique weak solution of : 
\begin{itemize}
\item[$\bullet$] \eqref{eq:Dirichlet source Equation-g} with $\hat \sigma=0$ and $\hat \kappa=\kappa c_\gamma \gamma^{-1}$, if $\theta<2-\gamma$;
\item [$\bullet$]  \eqref{eq:Dirichlet source Equation-g} with $\hat \sigma=\sigma$ and $\hat \kappa=\kappa c_\gamma \gamma^{-1}$, if $\theta =2-\gamma$; 
\item[$\bullet$]  \eqref{eq:Dirichlet source Equation-g} with $\hat \sigma=\sigma$ and $\hat \kappa=0$, if $\theta \in (2-\gamma,1)$;
\item[$\bullet$]  (\ref{Robin Equation-g}) with $\hat \sigma=\sigma$ and $\hat m=m\kappa $, if $\theta =1$;
\item[$\bullet$]  (\ref{Robin Equation-g}) with $\hat \sigma=\sigma$ and $\hat m=0$, if $\theta \in (1,\infty)$.
\end{itemize}
\end{thm}

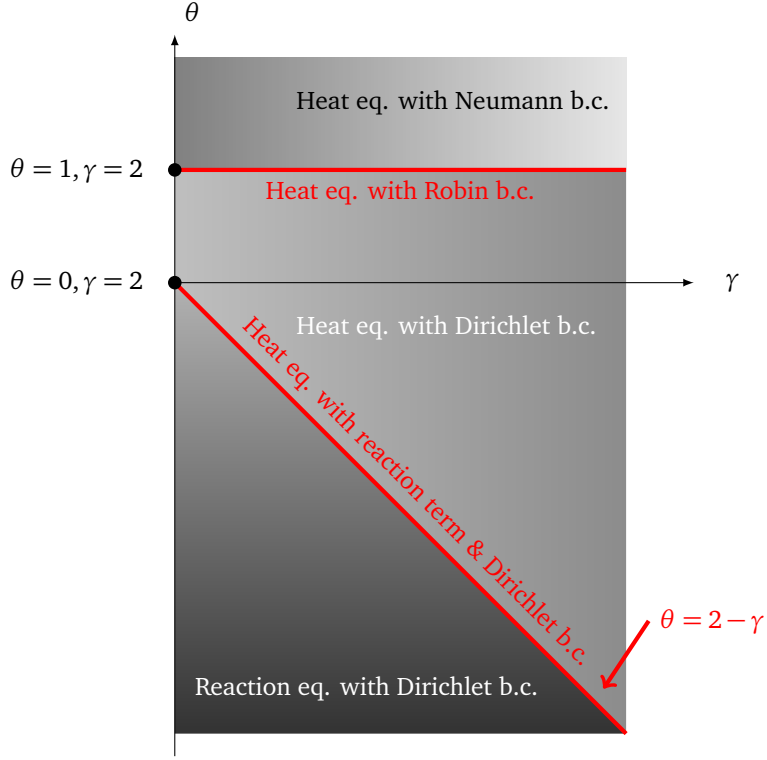
\begin{figure}

\begin{center}
\begin{tikzpicture}[scale=0.3]
\shade[right color=Gray!20] (0,5) -- (0,10) --(20,10)--(20,5)-- cycle;
\shade[right color=Gray!90, left color=Gray!50] (0,5) -- (20,5) --(20,-20)--(0,0)-- cycle;
\shade[top color=black!30, bottom color=black!80] (0,0) -- (20,-20) --(0,-20)-- cycle;
\draw (0,12) node[right]{$\theta$};
\draw (24,0) node[right]{$\gamma$};
\draw (-1,0) node[left]{$\theta=0, \gamma=2$};
\draw (-1,5) node[left]{$\theta=1, \gamma=2$};
\draw[->,>=latex] (0,0) -- (23,0);
\draw[->,>=latex] (0,-21) -- (0,11);
\draw[-,=latex,red,ultra thick] (0,0) -- (20, -20) node[midway, above, sloped] {{Heat eq. with reaction term \& Dirichlet b.c. }};
\draw[-,=latex,red,ultra thick] (0,5) -- (20, 5) node[midway, sloped, below] {{Heat eq. with Robin b.c.}};
\node[right, black] at (5,8) {Heat eq. with Neumann b.c.} ;
\node[right, white] at (5,-2) {Heat eq. with Dirichlet b.c.} ;
\node[right, white] at (0.5,-18) {Reaction eq. with Dirichlet b.c.} ;
\fill[black] (0,0) circle (0.3cm);
\fill[black] (0,5) circle (0.3cm);
\draw[<-,red,ultra thick] (19,-18) -- (21, -15) node[right] {$\theta=2-\gamma$};
\end{tikzpicture}
\end{center}
\caption{The five different hydrodynamic regimes in terms of $\gamma$ and $\theta$.}
\label{fig:hlf}
\end{figure}

It is not always possible to write fully explicit expressions for the solutions of these hydrodynamic equations. The form of the corresponding stationary solutions is of interest since the latter are expected to describe, in general, the mean density profile in the non-equilibrium stationary state of the microscopic system in the thermodynamic limit $N \to \infty$. Observe that this is not a trivial fact since it requires to exchange the limits $t \to \infty$ with $N \to \infty$ (and for $\theta>1$ this is for example false, see below). 

The stationary solutions of the hydrodynamic limits in the $\theta>2-\gamma$ case are standard. On the other hand, the form and properties of the stationary solutions in the $\theta \le 2-\gamma$ case are original and more tricky to obtain in the $\theta= 2-\gamma$ case. This problem is studied in more details in \cite{JO-V}. Here we only present some graphs of the stationary solutions and refer the interested reader to \cite{JO-V} for a complete mathematical treatment. 

\begin{figure}[h]
    \centering
    \includegraphics[width=1\textwidth]{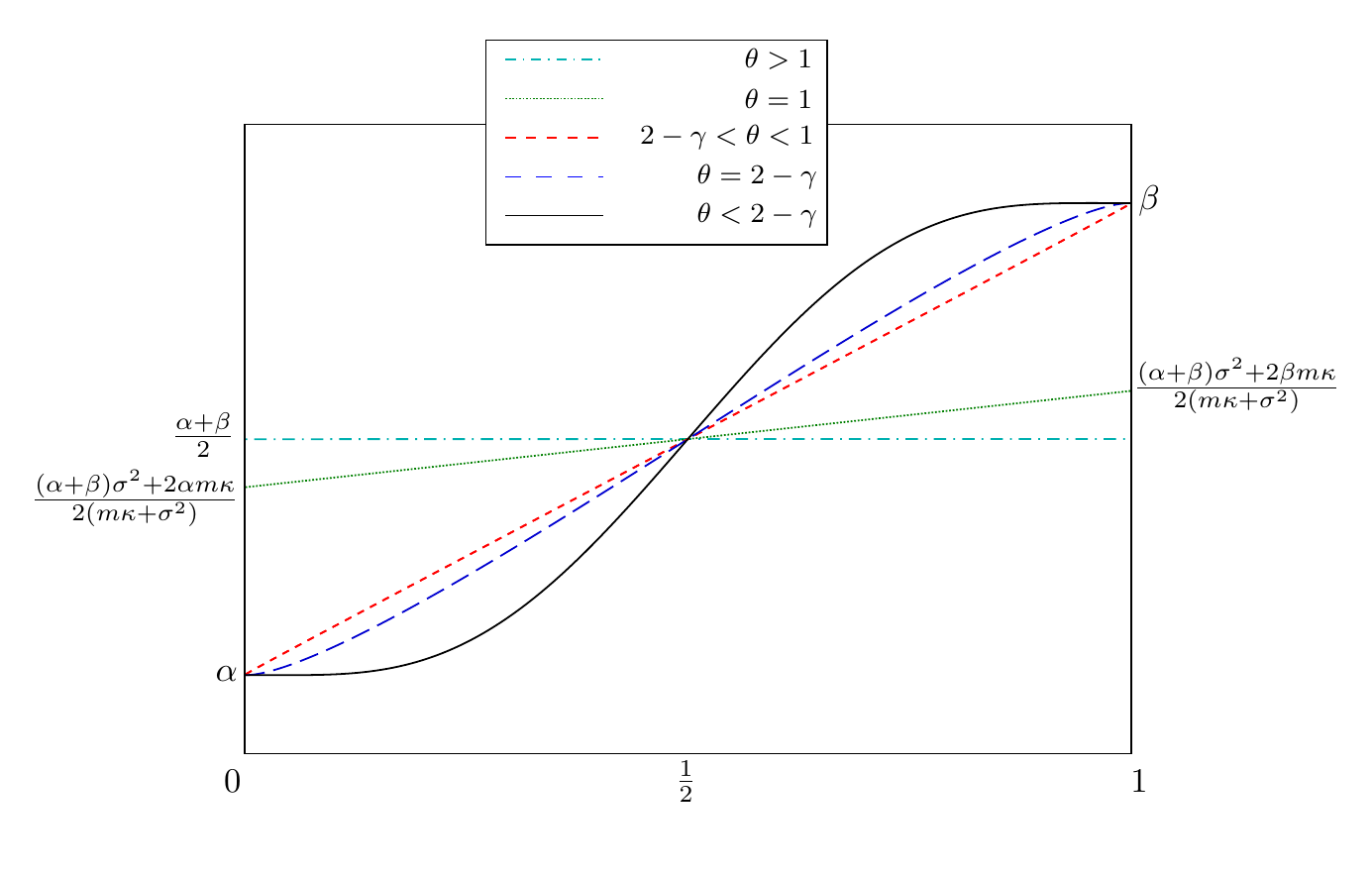}
    \caption{Profiles of the stationary solution of the hydrodynamic equations according to the value of $\theta$.   }
    \label{fig:mesh1}
\end{figure}

For $\theta \in (2-\gamma, 1)$ (heat equation with Dirichlet boundary conditions) the stationary solution is the linear profile connecting $\alpha$ at $0$ to $\beta$ at $1$. For $\theta=1$ (heat equation with Robin boundary conditions) the profile is still linear but the values at the boundaries are different. Observe that if $\kappa \to 0$ these values converge to $\tfrac{\alpha + \beta}{2}$ so that the profile becomes flat equal to $\tfrac{\alpha + \beta}{2}$. For $\theta >1$ (heat equation with Neumann boundary conditions) the stationary solution is constant equal to $\int_0^1 g(q) dq$ where $g(\cdot)$ is the initial condition. In fact, for $\theta>1$, we expect that if we compute directly the stationary profile in the non-equilibrium stationary state of the microscopic system in the thermodynamic limit, the stationary profile will be flat with the value $(\alpha+\beta)/2$. This value is therefore memorized in the form of the hydrodynamic limits for $\theta=1$, despite the fact that it has been forgotten in the hydrodynamic limits for $\theta>1$. In the case $\theta < 2- \gamma$ (reaction equation) the stationary profile is fully explicit and given by $ \tfrac{V_0(q)}{V_1(q)}$  where 
\begin{equation}
\label{V_0_and_V_1}
V_0 (q) = \alpha q^{-\gamma} + \beta (1-q)^{-\gamma}, \quad V_1 (q) =q^{-\gamma} + (1-q)^{-\gamma}.
\end{equation}
Observe that this profile is increasing, non-linear, convex on $(0,1/2)$ and concave on $(1/2, 1)$ and connects $\alpha$ at $0$ to $\beta$ at $1$. At the boundaries the profile is very flat. In \cite{JO-V} it is proved that these properties remain valid for the stationary solution of the hydrodynamic equation in the $\theta=2-\gamma$ case. \\


\subsection{Complementary results}
\label{subsec:complements}

In order to limit the length of the paper we decided to consider in details only one kind of reservoirs. However, since a reservoir model is not universal, other natural models are of interest and in this subsection we explain, without proofs, how our results have to be modified in these contexts. We will discuss three cases:
\begin{enumerate}[{\bf{Case}} 1:\, ]
\item The reservoir consists on the left (resp. on the right) of a single Glauber dynamics whose action of the generator on a function $f:\Omega_N \to \RR$ is
\begin{equation*}
\begin{split}
&(L_N^{\ell} f)(\eta) = \frac{\kappa}{N^\theta} \sum_{x \in \Lambda_N} c_{x}(\eta;\alpha) p(x) [f(\sigma^x\eta) - f(\eta)],\\
&\Big( \quad \text{resp. \;}  (L_N^{r} f)(\eta)= \frac{\kappa}{N^\theta} \sum_{x \in \Lambda_N} c_{x} (\eta;\beta) p(N-x)  [f(\sigma^{N-1}\eta) - f(\eta)] \quad\Big).
\end{split}
\end{equation*}
Thus it creates a particle at the  site $x \in \Lambda_N$ with rate $\tfrac{\kappa}{N^\theta} \alpha p(x) $ (resp.$\tfrac{\kappa}{N^\theta} \beta p(N-x)$) if the site $x$ is empty and it removes a particle at the  site $x$ with rate $\tfrac{\kappa}{N^\theta} (1-\alpha) p(x) $ (resp.$\tfrac{\kappa}{N^\theta}(1-\beta) p(N-x)$) if the site $x$ is occupied. The bulk dynamics is unmodified.
\item The reservoir consists on the left (resp. on the right) of a single Glauber dynamics whose action of the generator on a function $f:\Omega_N \to \RR$ is
\begin{equation*}
\begin{split}
&(L_N^{\ell} f)(\eta) = \frac{\kappa}{N^\theta} c_{1}(\eta;\alpha) [f(\sigma^1\eta) - f(\eta)],\\
&\Big( \quad \text{resp. \;}  (L_N^{r} f)(\eta)= \frac{\kappa}{N^\theta} c_{N-1}(\eta;\beta)  [f(\sigma^{N-1}\eta) - f(\eta)] \quad\Big).
\end{split}
\end{equation*}
Thus it creates a particle at the site $1$ with rate $\tfrac{\kappa}{N^\theta} \alpha$ (resp.$\tfrac{\kappa}{N^\theta} \beta$) if the site $1$ (resp. $N-1$) is empty and it removes a particle at the  site $1$ with rate $\tfrac{\kappa}{N^\theta} (1-\alpha)$ (resp.$\tfrac{\kappa}{N^\theta}(1-\beta)$) if the site $1$ (resp. $N-1$) is occupied. The bulk dynamics is unmodified.
\item The reservoir consists on the left (resp. on the right) of an infinite number of Glauber dynamics whose action of the generator on a local function $f:\{0,1\}^\ZZ \to \RR$ is
\begin{equation*}
\begin{split}
&(L_N^{\ell} f)(\eta) = \frac{\kappa}{N^\theta}\sum_{x \le 0} c_{x}(\eta;\alpha) [f(\sigma^x\eta) - f(\eta)],\\
&\Big( \quad \text{resp. \;}  (L_N^{r} f)(\eta)= \frac{\kappa}{N^\theta}\sum_{x\ge N}  c_{x}(\eta;\beta)  [f(\sigma^x\eta) - f(\eta)] \quad\Big).
\end{split}
\end{equation*}
Thus it creates a particle at the site $x\le 0$ (resp. $x\ge N$) with rate $\tfrac{\kappa}{N^\theta} \alpha$ (resp.$\tfrac{\kappa}{N^\theta} \beta$) if the site $x$ is empty and it removes a particle at the  site $x \le 0$ (resp. $x\ge N$) with rate $\tfrac{\kappa}{N^\theta} (1-\alpha)$ (resp.$\tfrac{\kappa}{N^\theta}(1-\beta)$) if the site $x$ is occupied. 
Moreover in this case we assume that the long jumps are not restricted to sites $x,y \in \Lambda_N$ but may occur in all the lattice $\ZZ$, i.e. the action of the bulk dynamics generator on a local function $f:\{0,1\}^\ZZ \to \RR$ is now described by 
\begin{equation*}
(L^0 f)(\eta) =\cfrac{1}{2} \, \sum_{x,y \in \ZZ} p(x-y) [ f(\sigma^{x,y}\eta) -f(\eta)].
\end{equation*}
\end{enumerate} 

In the two first cases, the density profile will be described by a function $\rho_t (q)$ where $q\in [0,1]$ while in the third case it will be described by a function $\rho_t (q)$, $q \in \RR$, since the system evolves on $\ZZ$. \\

\begin{enumerate}[{\bf{Case}} 1:\, ]
\item We have still five different regimes. The changes with respect to our results are:
\begin{enumerate}[a)]
\item the value of $\theta$ for which we obtain the reaction-diffusion equation (now is $\theta=\gamma-1$ instead of $\theta=2-\gamma$) and the reaction equation (now is for $\theta<\gamma-1$ instead of $\theta<2-\gamma$); 
\item the functions $V_1$ and $V_0$ are the same as before but the exponent in this case is $1+\gamma$ instead of $\gamma$. We note that all the other regimes are not affected.
\end{enumerate}
\item We have now only three different regimes which occur all in the diffusive time scale. If $\theta>1$ the macroscopic behavior is described by the heat equation with Neumann boundary conditions; if $\theta=1$, it is described by the heat equation with Robin boundary condition; if $\theta<1$ (positive or negative) it is described by the heat equation with Dirichlet boundary conditions.  
\item We have now only three different regimes. 
\begin{enumerate}[a)]
\item If $\theta>2$ the reservoirs are too weak and the density profile evolves in the diffusive scaling according to the heat equation on $\RR$ 
$$\partial \rho_t (q)  = \tfrac{\sigma^2}{2} \Delta \rho_t (q) ,$$
without any boundary conditions.
\item If $\theta=2$, the density profile evolves in the diffusive scaling according to the reaction-diffusion equation on $\RR$
$$\partial \rho_t (q) = \tfrac{\sigma^2}{2} \Delta \rho_t (q) - \kappa {\bf 1}_{q\le 0} (\rho_t (q) -\alpha) -\kappa {\bf 1}_{q\ge 1} (\rho_t (q) -\beta).$$
\item If $\theta>2$, the reservoirs are so fast that in the diffusive time scale they fix the density profile to be $\alpha$ at the left of $0$ and $\beta$ at the right of $1$. In the bulk $(0,1)$, the density profile evolves according to the heat equation restricted to $(0,1)$ with these inhomogeneous Dirichlet boundary conditions. 
\end{enumerate}
\end{enumerate}

\vspace{1cm}

{\bf Notations:} We write $f(x) \lesssim g(x)$ if there exists a constant $C$ independent of $x$ such that $f(x) \le C g(x)$ for every $x$. We will also write $f(x) = {\mc O} (g(x) )$ if the condition $|f (x) | \lesssim |g(x) |$ is satisfied. Sometimes, in order to stress the dependence of a constant $C$ on some parameter $a$, we write $C(a)$.

\section{Heuristics for the Hydrodynamic equations} \label{sec:CL}

In this section we give the main ideas which are behind the identification of limit points as weak solutions of the partial differential equations given in  Section \ref{sec:hyd_eq}. In Section \ref{sec:Tightness}, we show that the sequence $\{ \mathbb Q_N \}_{N\ge 1}$ is tight and in Section \ref{sec:Characterization of limit points} we prove that all limiting points of  the sequence $\lbrace\mathbb{Q}_{N}\rbrace_{N\geq 1}$ are concentrated on trajectories of  measures that are  absolutely continuous with respect to the Lebesgue measure, that is $\pi_t(dq)=\rho_t(q)dq$. Now we argue that the density $\rho_t(q)$ is a weak solution of the corresponding hydrodynamic equation for each regime of $\theta$. The precise proof of this result is given ahead in Proposition  \ref{prop:weak_sol_car}.

The identification of the density $\rho_t(q)$ as a weak solution of the hydrodynamic equation is obtained by using auxiliary martingales. For that purpose, and to make the exposition simpler, we  fix a function $G:[0,1]\to\bb R$ which does not depend on time and  is two times continuously differentiable. If $\theta <1$ we will assume further that it has a compact support included in $(0,1)$ and for $\theta\geq 1$ we assume that it has a compact support but not necessarily  contained in $[0,1]$ so that $G$ has a good decay at infinity. In the last case  observe that $G$ can take non-zero values at $0$ and $1$.
We know by Dynkin's formula  that
\begin{equation}\label{Dynkin'sFormula}
M_{t}^{N}(G)= \langle \pi_{t}^{N},G\rangle -\langle \pi_{0}^{N},G\rangle-\int_{0}^{t}\Theta(N)L_{N}\langle \pi_{s}^{N},G\rangle \, ds,
\end{equation}
is a martingale with respect to the natural filtration  $\{\mathcal{F}_{t}\}_{ t\ge 0}$, where for each $t\ge 0$, $\mathcal{F}_t:=\sigma(\eta(s): s < t)$. Above  the notation $\left\langle \pi_{s}^{N},G\right\rangle$ represents  the integral of $G$ with respect the measure $ \pi_{s}^{N}$. This notation should not  be mistaken  with the notation used for the inner product in $L^{2}([0,1])$.
A simple computation, based on \eqref{T6} and the discussion after this equation, shows that $\mathbb{E}_{\mu_N} \Big[ \big(M_{t}^{N}(G)\big)^2\Big]$ vanishes as $N\to\infty$. Now we look at the integral term in \eqref{Dynkin'sFormula}. A simple computation shows that 
\begin{equation}\label{genaction}
\begin{split}
\int_0^t \Theta(N) L_N ( \langle \pi^N_{s}, G \rangle ) \, ds =& \cfrac{\Theta(N)}{N-1} \int_0^t  \sum_{x\in \Lambda_N}  \mc L_NG(\tfrac{x}{N}) \eta^N_x(s) \, ds \\+& \cfrac{ \kappa \Theta(N)}{(N-1)N^\theta} \int_0^t \sum_{x \in \Lambda_N}  (G r_{N}^{-})(\tfrac{x}{N}) (\alpha-\eta^N_x(s)) \, ds \\+& \cfrac{\kappa \Theta(N)}{(N-1)N^\theta} \int_0^t  \sum_{x \in \Lambda_N}  (G r_{N}^{+})(\tfrac{x}{N}){(\beta-\eta^N_x(s)} \, ds,
\end{split}
\end{equation}
where for all $x\in \Lambda_N$
\begin{equation}
\label{eq:mcln}
\begin{split}
 ({\mc L}_N G) (\tfrac{x}{N}) = \sum_{y \in \Lambda_N} p(y-x) \left[ G(\tfrac{y}{N}) -G(\tfrac{x}{N})\right],\\
r_N^- (\tfrac{x}{N})= \sum_{y \ge x} p(y), \quad r_N^+ (\tfrac{x}{N})= \sum_{y \le x-N} p(y).
\end{split}
\end{equation}   

Now, we want to extend the first sum in \eqref{genaction} to all the integers. For that purpose we extend the function $G$ to $\RR$ in such a way that it remains two times continuously differentiable. By the definition of $\mc L_{N}$, we get that
\begin{equation}\label{bulkaction}
\begin{split}
\cfrac{\Theta(N)}{N-1} \int_0^t \sum_{x\in \Lambda_N}  \mc L_NG(\tfrac{x}{N}) \eta_x^N (s) \, ds
 =&\cfrac{\Theta(N)}{N-1} \int_0^t \sum_{x\in \Lambda_N} (K_N G) (\tfrac{x}{N})\eta_x^N (s) \, ds  \\
 -& \cfrac{\Theta(N)}{N-1}\int_0^t  \sum_{x\in \Lambda_N}\sum_{y\leq 0}\left[G(\tfrac{y}{N})- G(\tfrac{x}{N})\right] p(x-y)\eta_x^N (s) \, ds \\
-&\cfrac{\Theta(N)}{N-1} \int_0^t  \sum_{x\in \Lambda_N}\sum_{y\geq N}\left[G(\tfrac{y}{N})- G(\tfrac{x}{N})\right] p(x-y)\eta_x^N (s) \, ds,
\end{split}
\end{equation}
where 
\begin{equation}
\label{eq:kappan}
 ({K}_N G) (\tfrac{x}{N}) = \sum_{y \in\mathbb Z} p(y-x) \left[ G(\tfrac{y}{N}) -G(\tfrac{x}{N})\right].
\end{equation}

Now, we are going to analyze all the  terms in \eqref{bulkaction} and the boundary terms in \eqref{genaction} for the different regimes of $\theta.$ Thus, we will be able to see how the different boundary conditions appear on the hydrodynamic equations given in Section \ref{sec:hyd_eq} from the underlying particle system.

Let us first observe that, for any $a \in (0,1)$, uniformly in $u \in (a,1-a)$, we have that
\begin{equation}
\label{eq:rs-r+-}
N^\gamma r_N^- ( [uN]) \to c_\gamma \gamma^{-1} u^{-\gamma}:=r^- (u), \quad N^\gamma r_N^+ ( [uN]) \to c_\gamma \gamma^{-1} (1-u)^{-\gamma}:=r^+ (u)
\end{equation}
as $N\to \infty$.

\subsection{The case $\theta <2-\gamma$} \label{sec:CLPDC0}

In this regime we take initially a function $G: (0,1) \to\bb R$ two times continuously differentiable  and with compact support  in $(0,1)$ (so that we can choose an extension by $0$ outside of $(0,1)$). 

Now we start by analyzing the first term on the right hand side of \eqref{bulkaction}. Since $\Theta(N)=N^{\gamma+\theta}$, a simple computation, shows that the first term on the right hand side of \eqref{bulkaction}  vanishes for $\theta<2-\gamma$. Indeed, by a Taylor expansion on $G$ and the fact that $p(\cdot)$ is mean zero, we have that
\begin{eqnarray*}
&& N^{\gamma+\theta}\sum_{y\in \ZZ}(G(\tfrac{y+x}{N})-G(\tfrac{x}{N}))p(y)
\end{eqnarray*}
is of same order as
$$N^{\gamma+\theta-2}G''(\tfrac{x}{N})\sum_{y\in \ZZ}y^2p(y)$$
and since $\theta<2-\gamma$ last expression vanishes as $N\to\infty. $

Moreover, a simple computation shows that  the second and third terms on the right hand side of \eqref{bulkaction} vanish as $N\to\infty$, since $\Theta(N)=N^{\gamma+\theta}$ and $\theta<2-\gamma$. Indeed we can bound from above, for example the second  term in \eqref{bulkaction} by $t N^\theta$ times 
$$\cfrac{1}{N-1} \sum_{x \in \Lambda_N} N^\gamma r_N^- (\tfrac{x}{N})\, |G(\tfrac{x}{N})|$$ 
because $G$ vanishes outside $(0,1)$ and $|\eta_x^N (s)| \le 1$ for all $s>0$. Since $\theta<0$ and that the previous sum converges to the (finite) integral of $|G| r^-$ on $(0,1)$, by (\ref{eq:rs-r+-}), the previous display vanishes as $N\to\infty$. Now we look at the boundary terms in \eqref{genaction}. The second term on the right hand side of \eqref{genaction} can be written, for the choice of $\Theta(N)=N^{\gamma+\theta}$, as:
\begin{equation*}
\frac{\kappa N^\gamma}{N-1} \int_0^t \sum_{x\in\Lambda_N}G \big(\tfrac{x}{N} \big ) r_N^-(\tfrac{x}{N})(\alpha-\eta^N_x(s)) \, ds
\end{equation*}
which can be replaced, thanks to (\ref{eq:rs-r+-}) and the fact that $G$ has compact support, by
\begin{equation*}
\kappa \, \int_0^t \langle \alpha-\pi_s^N, G r^- \rangle \, ds  \to \kappa \int_0^t \int_0^1 G(q) r^- (q) (\alpha- \rho_s (q) {)}  dq \, ds
\end{equation*}
as $N \to \infty$. The last convergence holds because $G$ has a compact support included in $(0,1)$ so that $G r^-$ is a continuous function. For the remaining term we can perform exactly the same analysis.

\subsection{The case $\theta =2-\gamma$} \label{sec:CLPDC0} 

In this case, and as above, we take initially a function $G:(0,1) \to\bb R$ two times continuously differentiable and with compact support  in $(0,1)$ (so that we can choose a two times continuously differentiable extension which is $0$ outside of $(0,1)$). 
In this case, since $\Theta(N)=N^2$, by Lemma \ref{convergence laplacian},
 which we prove below, the first term on the right hand side of \eqref{bulkaction} can be replaced, for $N$ sufficiently big, by 
 \begin{eqnarray*}
\cfrac{1}{N-1} \int_0^t \sum_{x\in \Lambda_N}\cfrac{\sigma^2}{2} \Delta G(\tfrac{x}{N}) \, \eta^N_x(s) \, ds. 
\end{eqnarray*}
Moreover, a  computation similar to the one above shows that  the second and third terms on the right hand side of \eqref{bulkaction} vanish as $N\to\infty$ (recall that $\Theta(N)=N^{2}$ and $\gamma>2$). Finally, the first term on the right hand side of \eqref{genaction} can be rewritten as 
\begin{equation*}\label{gen_action}
\cfrac{ \kappa N^{\gamma}}{(N-1)} \int_0^t \sum_{x \in \Lambda_N}  (G r_{N}^{-})(\tfrac{x}{N})\,  (\alpha-\eta^N_x(s)) \, ds
\end{equation*}
which can be replaced, thanks to (\ref{eq:rs-r+-}) and the fact that $G$ has compact support, by
\begin{equation*}
\kappa \, \int_0^t \langle {\alpha-}  \pi_s^N, G r^- \rangle \, ds \to \kappa \int_0^t \int_0^1 G(q) r^- (q) {(\alpha-} \rho_s (q){)}  dq \, ds
\end{equation*}
as $N \to \infty$ because $G r^-$ is a continuous function. The same computation can be done for the remaining term. 


\subsection{The case $\theta \in (2-\gamma,1)$} \label{sec:CLPDC} 

In this case we take again  a function $G: (0,1) \to\bb R$ two times continuously differentiable and with compact support  in $(0,1)$ and extend it by $0$ outside of $(0,1)$. As above, we can easily show that  the last two terms on the right hand side of \eqref{genaction} vanish as $N\to\infty$, since we can transform each one of them it into $N^{2+\gamma-\theta}$ times a converging integral, which  vanishes since $\theta>2-\gamma$. Analogously, the second and third terms on the right hand side of \eqref{bulkaction} also vanish because, for example, the second term on the right hand side of \eqref{bulkaction}
 \begin{equation*}
 \cfrac{N^{2}}{N-1} \int_0^t \sum_{x\in \Lambda_N}G(\tfrac{x}{N}) r_N^-(\tfrac{x}{N})\eta^N_x (s) \, ds
\end{equation*}
can be bounded from above by a constant times $t N^{2-\gamma}$ times a sum converging to the integral of $| G | r^-$ on $(0,1)$.  The estimate of the third term is analogous. Therefore  since $\gamma>2$, both vanish as $N\to\infty$. \\

\begin{rem}
Observe that in the three previous cases, we imposed to  $G$ to have a compact support included in $(0,1)$. This was used in order to extend smoothly the function $G$ by $0$ outside of $(0,1)$ (the condition $G(0)=G(1)=0$ would not have been sufficient) and this was fundamental to ensure that the functions $G r^-$, $G r^+$ do not have singularities at the boundaries. On the other hand, in the two next cases, it will be fundamental to consider test functions $G:[0,1] \to \RR$ which are not necessarily $0$ at the boundaries in order to ``see" the boundaries in the weak formulation. 
\end{rem}

\subsection{The case $\theta=1$}

In this case we consider an arbitrary function $G:[0,1] \to \RR$ which is two times continuously differentiable and we extend it on $\RR$ in a two times continuously differentiable function with compact support. Its support strictly (a priori) contains $[0,1]$ since $G$ can take non-zero values at $0$ and $1$. We start by looking  at the terms coming from the boundary, namely the two  last terms on the right hand side of \eqref{genaction}. Then, in the second term on the right hand side of \eqref{genaction} (resp. the third term) we perform at first a Taylor expansion on $G$ and then we  replace $\eta_x$ by the average $\overrightarrow{\eta}_{0}^{\ve N}$ (resp.  $\eta_x$ by $\overleftarrow{\eta}_{N}^{\ve N}$) defined in \eqref{boxes}, which can be done as a consequence of Lemma \ref{Rep-Neumann} as pointed out in Remark \ref{sec_rep_robin}. Moreover, note that 
\begin{equation}\label{mean}
\begin{split}
\sum_{x \in \Lambda_N}   r_{N}^{-}(\tfrac{x}{N})\xrightarrow[N\uparrow \infty]{} \sum_{y\geq 1}yp(y)=m,\quad \quad \quad 
\sum_{x \in \Lambda_N}   r_{N}^{+}(\tfrac{x}{N})\xrightarrow[N\uparrow \infty]{} \sum_{y\geq 1}yp(y)=m.
\end{split}
\end{equation}
Therefore, we can write the last two terms in \eqref{genaction} as
\begin{equation*}
m\kappa \int_0^t \{ (\alpha-\overleftarrow{\eta}_0^{\ve N}(sN^2)) G(0) + {(\beta-\overrightarrow{\eta}_{N}^{\ve N}(s N^2))} G(1) \} \, ds,
\end{equation*}
plus lower-orders terms (with respect to $N$). Since  (in some sense that we will see in the proof of Proposition \ref{prop:weak_sol_car} in Section \ref{sec:Characterization of limit points}), 
$$\overrightarrow{\eta}_{0}^{\ve N} (sN^2) \xrightarrow[N\uparrow \infty]{}\rho_s (0), \quad \overleftarrow{\eta}_{N}^{\ve N} (sN^2) \xrightarrow[N\uparrow \infty]{} \rho_s (1)$$ last term writes as 
\begin{equation}\label{rob_cond1}
m\kappa  \int_0^t \{ (\alpha-\rho_s (0))G(0) + {(\beta-\rho_s (1))} G(1) \} \, ds.
\end{equation}
Now we look at the remaining terms, namely, the two last terms on the right hand side of \eqref{bulkaction}.  Recall that the function $G$ has been extended into a two times continuously differentiable function on $\bb R$.
By a Taylor expansion on $G$ we can write those terms as 
 \begin{equation}\label{eq:1_0}
   \cfrac{N}{N-1} \sum_{x\in \Lambda_N}G'(\tfrac{x}{N})\Theta^-_x\eta_x(sN^2)-\cfrac{N}{N-1} \sum_{x\in \Lambda_N}G'(\tfrac{x}{N})\Theta^+_x\eta_x(s N^2)
\end{equation}
plus lower-order terms (with respect to $N$). 
Above for $x\in\Lambda_N$, 
 \begin{eqnarray*}
  \Theta^-_x=\sum_{y\leq 0} (x-y)p(x-y)\quad \textrm{and}\quad \Theta^+_x=\sum_{y\geq N} (y-x)p(x-y).
\end{eqnarray*}
Note that
\begin{equation}
\label{eq:thetaminus}
\sum_{x\in \Lambda_N}\Theta^-_x \; \lesssim \; 1\quad \textrm{and}\quad \frac{1}{N}\sum_{x\in \Lambda_N}x\Theta^-_x  \xrightarrow[N\to \infty]{} 0.
\end{equation}
Moreover, note that
 \begin{equation}
 \label{eq:sum236}
 \begin{split}
 &\sum_{x\in \Lambda_N}\Theta^-_x= \sum_{x\in \Lambda_N}\sum_{y\geq x}yp(y)\xrightarrow[N\uparrow \infty]{} \tfrac{\sigma^2}{2},
\\
&\sum_{x\in \Lambda_N}\Theta^+_x= \sum_{x\in \Lambda_N}\sum_{y\geq N-x}yp(y)\xrightarrow[N\uparrow \infty]{} \tfrac{\sigma^2}{2}.
 \end{split}
 \end{equation}
In order to prove the convergence of  $\,\sum_{x\in \Lambda_N}\Theta^-_x$ (or of $\sum_{x\in \Lambda_N}\Theta^+_x$ in (\ref{eq:sum236})) we use Fubini's theorem  to get that
\begin{equation}
\begin{split}
\sum_{x\in \Lambda_N}\Theta^-_x &= \sum_{y\in\Lambda_{N}}\sum_{x=1}^{y}yp(y)+\sum_{y\geq N}\sum_{x\in \Lambda_{N}}yp(y)\\\nonumber
&= \sum_{y\in\Lambda_{N}}y^{2}p(y)+(N-1)\sum_{y\geq N}yp(y),
\end{split}
\end{equation}
 and since  $\gamma>2$ the result follows. 
By another Taylor expansion on $G$ we can write \eqref{eq:1_0}  as 
 \begin{eqnarray}\label{eq:rob_1}
  \cfrac{N}{N-1} G'(0)\sum_{x\in \Lambda_N}\Theta^-_x\eta_x(s N^2)-\cfrac{N}{N-1}G'(1) \sum_{x\in \Lambda_N}\Theta^+_x\eta_x(s N^2)
\end{eqnarray}
plus lower-order terms (with respect to $N$). 
Thanks to Lemma \ref{Rep-Neumann} we can replace in the  term on the left (resp. right) hand side of last expression $\eta_x(sN^2)$ by $\overrightarrow{\eta}_0^{\ve N}(sN^2)$ (resp. $\overleftarrow{\eta}_{N}^{\ve N}(sN^2)$). Therefore, \eqref{eq:rob_1} can be replaced, for $N$ sufficiently big and then $\varepsilon$ sufficiently small, by
  \begin{eqnarray*}
  G'(0)\tfrac{\sigma^2}{2}\overrightarrow{\eta}_0^{\ve N}(sN^2)-G'(1)\tfrac{\sigma^2}{2}\overleftarrow{\eta}_{N}^{\ve N}(sN^2).
\end{eqnarray*}
Since (in some sense that we will see in the proof of Proposition \ref{prop:weak_sol_car} in Section \ref{sec:Characterization of limit points}), we have that $\overrightarrow{\eta}_{0}^{\ve N} (s N^2) \xrightarrow[N\to \infty]{} \rho_s (0)$ and $\overleftarrow{\eta}_{N}^{\ve N} (sN^2) \xrightarrow[N\to \infty]{} \rho_s (1)$, last term tends to 
  \begin{eqnarray}
  \label{rob_cond2}
 G'(0)\tfrac{\sigma^2}{2}\rho_s (0)-G'(1)\tfrac{\sigma^2}{2}\rho_s (1).
\end{eqnarray}

Putting together \eqref{rob_cond1} and \eqref{rob_cond2} we see the boundary terms  that appear at the right hand side of  \eqref{eq:Robin integral-g}.

 \subsection{The case $\theta \in (1,\infty)$}
 
 \label{subsec:thetage1}
 
In this case we consider an arbitrary function $G:[0,1] \to \RR$ which is two times continuously differentiable and we extend it on $\RR$ in a two times continuously differentiable function with compact support. Its support may strictly contain $[0,1]$ since $G$ can take non-zero values at $0$ and $1$. The last two terms on the right hand side of \eqref{genaction} vanish, as $N\to\infty$ since, we can bound, for example, the first term on the right hand side of  \eqref{genaction}  by a constant times
 \begin{equation*}
N^{1-\theta} \sum_{x \in \Lambda_N}   r_{N}^{-}(\tfrac{x}{N}).
\end{equation*}
Since $\gamma>2$ last expression vanishes if $\theta>1$. Thus, we only need to look at the expression (\ref{bulkaction}). Therefore, in order to see the boundaries terms that appear in \eqref{eq:Robin integral-g}, we can use exactly the computations already done in the case $\theta =1$ from  which we obtain (\ref{rob_cond2}).

Now we prove the convergence to the Laplacian which was required above. 
\begin{lem}\label{convergence laplacian}
Let $G:\bb R\to \bb R$ be a two times continuously differentiable function with compact support. We have
\begin{eqnarray*}
&&\limsup_{N\rightarrow\infty} \sup_{x \in \Lambda_N} \left\vert N^2(K_{N}G)(\tfrac{x}{N}) - \frac{\sigma^{2}}{2}\Delta G(\tfrac{x}{N})\right\vert= 0.
\end{eqnarray*}
\end{lem}

\begin{proof} 
Let $\varepsilon>0$ be fixed. We have that $N^2(K_{N}G)(\tfrac{x}{N})$ is equal to
\begin{eqnarray}\label{CL1}
N^2\sum_{\vert y \vert \geq \varepsilon N}(G(\tfrac{x+y}{N})-G(\tfrac{x}{N}))p(y) + N^2\sum_{\vert y \vert < \varepsilon N}(G(\tfrac{x+y}{N})-G(\tfrac{x}{N}))p(y).
\end{eqnarray}
The first term in (\ref{CL1}) goes to zero with $N$, since  we have that 
$$\left\vert N^2\sum_{\vert y \vert \geq \varepsilon N}(G(\tfrac{x+y}{N})-G(\tfrac{x}{N}))p(y) \right\vert\lesssim \dfrac{\Vert G\Vert_{\infty}N^2}{(\varepsilon N)^{\gamma}}.$$
On the second term of  (\ref{CL1}) we perform  a Taylor expansion of $G$ and  we have that
\begin{eqnarray*}
& &N^2\sum_{\vert y \vert < \varepsilon N}[G(\tfrac{x+y}{N})-G(\tfrac{x}{N})]p(y)\\
&=&N^2\sum_{\vert y \vert < \varepsilon N} \left[  G^{\p}(\tfrac{x}{N})\tfrac{y}{N}+\dfrac{1}{2}G^{\p\p}(\tfrac{x}{N})(\tfrac{y}{N})^{2}\right]p(y),
\end{eqnarray*} 
plus lower-order terms (with respect to $N$).
Now, we use the fact that $p(\cdot)$ is symmetric to see that $\sum_{\vert y \vert < \varepsilon N}  y p(y) =0$. Since $p(\cdot)$ has finite second moment, $\sum_{|y| < \ve N} y^2 p(y) \to \sigma^2$  so that the proof ends.
\end{proof}

\section{Tightness}
\label{sec:Tightness}

In this section we prove that the sequence $\lbrace  \mathbb {Q}_{N} \rbrace_{N \geq 1} $, defined in Section \ref{sec:HL},  is tight. 
\begin{prop}\label{Tightness}
The sequence of measures $\lbrace\mathbb{Q}_{N}\rbrace_{N\geq 1}$ is tight with respect to the Skorohod topology of $\mathcal D([0, T],{\mathcal{M^{+}}})$.
\begin{proof}
In order to prove the assertion see, for example, Proposition 1.6 of Chapter 4 in \cite{KL}, it is enough to show that, for all $\varepsilon > 0$
\begin{equation}
\label{T1}
\displaystyle \lim _{\delta \rightarrow 0} \limsup_{N\rightarrow\infty} \sup_{\tau  \in \mathcal{T}_{T},\bar\tau \leq \delta} {\mathbb{P}}_{\mu _{N}}\Big(\eta_{\cdot}^{N}\in {\mathcal D} ( [0,T], \Omega_{N}) :\left\vert \langle\pi^{N}_{\tau+ \bar\tau},G\rangle-\langle\pi^{N}_{\tau},G\rangle\right\vert > \ve \Big) =0, 
\end{equation}
holds for any function $G$ belonging to $C([0,1])$. Here $\mathcal{T}_{T}$ is the set of stopping times bounded by $T$ and we implicitly assume that all the stopping times are bounded by $T$, thus, $\tau+ \bar\tau$ should be read as $ (\tau+ \bar\tau) \wedge T$. In fact it is enough to prove the assertion for functions $G$ in a dense subset of $C([0,1])$, with respect to the uniform topology. 

We split the proof according to two different regimes of $\theta$, namely $\theta\geq 1$ and $\theta<1$. When $\theta\geq 1$ we  prove (\ref{T1}) directly for functions  $G \in C^{2}([0,1])$ and we conclude that the sequence is tight. When $\theta<1$, we prove (\ref{T1}) first for functions $G \in C^{2}_{c}(0,1)$ and then we extend it, by a $L^1$ approximation procedure which is explained below, to functions $G\in C^1([0,1])$, the latter space being dense in $C([0,1])$ for the uniform topology.

Recall from (\ref{Dynkin'sFormula})  that $M_{t}^{N}(G)$ is a martingale with respect to the natural filtration  $\{\mathcal{F}_{t}\}_{t\geq 0}$. In order to prove (\ref{T1}) it is enough to show that
\begin{eqnarray} 
\label{TC1}
\displaystyle \lim _{\delta\rightarrow 0} \limsup_{N\rightarrow\infty} \sup_{\tau  \in \mathcal{T}_{T},\bar\tau \leq \delta}\mathcal{\mathbb{E}}_{\mu _{N}}\Big[ \Big| \int_{\tau}^{\tau+ \bar\tau}\Theta(N) L_{N}\langle \pi_{s}^{N},G\rangle ds \Big|\Big] = 0
\end{eqnarray}
and
 \begin{equation} 
 \label{TC2}
\displaystyle \lim _{\delta\rightarrow 0} \limsup_{N\rightarrow\infty} \sup_{\tau  \in \mathcal{T}_{T},\bar\tau \leq \delta}\mathcal{\mathbb{E}}_{\mu _{N}}\left[\left( M_{\tau}^{N}(G)- M_{\tau+ \bar\tau}^{N}(G) \right)^{2}  \right]=0.
\end{equation}

\vspace{0.5cm}
\noindent 
\underline{Proof of \eqref{TC1}:} Given a function $G$, we claim that we can find a positive constant $C:=C(G,\alpha,\beta, \gamma, \kappa)$ such that 
\begin{equation}
\label{eq:claim}
\vert \Theta(N) L_N ( \langle \pi^N_{s}, G \rangle )\vert \le  C
\end{equation}
for any $s\le T$, which trivially implies \eqref{TC1}. To prove it, we recall \eqref{gen_action} and start to prove that the last two terms of \eqref{gen_action} are bounded. For example, the absolute value of the second term at the right hand side of \eqref{gen_action} is bounded from above by 
\begin{equation}\label{T2}
\int_0^t\Big| \frac{ \Theta(N)\kappa }{(N-1)N^{\theta}} \sum_{x \in \Lambda_N}  (G r_N^-)(\tfrac{x}{N})(\alpha-\eta^N_{x}(s))\Big| ds. 
\end{equation}
Now, for $\theta < 1$, we use the fact that $G \in C_{c}^{2}(0,1)$ and that $|\eta_x^N (s)| \le 1$ is bounded, and we bound from above this last term by a constant times $\Theta(N)N^{-\theta-\gamma}$.  Using the definition of $\Theta(N)$ it is easy to see, for $\theta < 2-\gamma$ and for $2-\gamma\leq\theta <1$,   that  (\ref{T2}) is bounded from above by a constant. This proves \eqref{eq:claim} in the case $\theta<1$. In the case   $\theta\geq 1$, we use the fact that the sum in (\ref{T2}) is uniformly bounded in $N$  to conclude that (\ref{T2}) is bounded from above even if $G$ does not have a compact support included in $(0,1)$ . A similar argument can be done for the last term at the right hand side of \eqref{gen_action}.

Now we need to bound the first term at the right hand side of  (\ref{gen_action}). For $\theta<1$ we use the fact that $G\in C^2_c(0,1)$ so that 
 $\Big|\dfrac{\Theta(N)}{N-1} \langle \pi^N_{s}, \mc L_N G \rangle\Big| $ is less or equal  than
 \begin{equation}\label{T41}
\begin{split}
\cfrac{\Theta(N)}{N-1} \sum_{x\in \Lambda_N} \vert K_NG (\tfrac{x}{N})\vert  
 + \cfrac{\Theta(N)}{N-1} \sum_{x\in \Lambda_N}|G(\tfrac{x}{N})| r_N^-\Big(\tfrac{x}{N}\Big)
+\cfrac{\Theta(N)}{N-1} \sum_{x\in \Lambda_N}|G(\tfrac{x}{N})| r_N^+\Big(\tfrac{x}{N}\Big).
\end{split}
\end{equation}
The two terms at the right hand side of the previous expression can be bounded from above by a constant times $\Theta(N)N^{-\gamma}$.  It is clearly bounded in the case $\theta \ge 2-\gamma$ since then $\Theta(N)=N^2$ (recall $\gamma>2$). In the case $\theta<2-\gamma$, $\Theta(N)=N^{\theta+\gamma}$ and thus $\Theta(N)N^{-\gamma}$ is bounded. This together with 
 Lemma \ref{convergence laplacian} shows that $$\Big|\dfrac{\Theta(N)}{N-1} \langle \pi^N_{s}, \mc L_N G \rangle\Big| \le C, $$
 which proves the claim (\ref{eq:claim}) in the case $\theta<1$. Now, in the case $\theta\geq 1$, since $\Theta(N)=N^2$, we have that  the first term at the right hand side of  (\ref{gen_action}) is bounded from above by a constant times
  \begin{equation}\label{T4}
\begin{split}
&\cfrac{N^{2}}{N-1} \sum_{x\in \Lambda_N} \vert K_NG (\tfrac{x}{N})\vert  
 + \cfrac{N^{2}}{N-1} \sum_{x\in \Lambda_N}\sum_{y\leq 0}\left\vert G(\tfrac{y}{N})- G(\tfrac{x}{N})\right\vert p(x-y)\\
+&\cfrac{N^{2}}{N-1} \sum_{x\in \Lambda_N}\sum_{y\geq N}\left\vert G(\tfrac{y}{N})- G(\tfrac{x}{N})\right\vert p(x-y).
\end{split}
\end{equation}
 By the Mean Value Theorem, the two terms at the right hand side of the previous expression can be bounded from above by   
   \begin{equation}\label{T42}
\|G'\|_{\infty} \sum_{x\in \Lambda_N}\sum_{y\leq 0} |y-x|p(x-y)\lesssim \sum_{x\in \Lambda_N}\frac{1}{x^{\gamma-1}}
\end{equation}
which is finite since $\gamma>2$.
This together with Lemma \ref{convergence laplacian} proves   (\ref{eq:claim}) in the case $\theta\geq 1$.

\vspace{0.5cm}
\noindent 
\underline{Proof of \eqref{TC2}:} We know by Dynkin's formula that  $$\displaystyle\left( M^{N}_{t}(G)\right)^{2}-\int^{t}_{0} \Theta(N)\left[ L_{N} \langle\pi^{N}_{s},G \rangle^{2}- 2\langle\pi^{N}_{s},G \rangle L_{N} \langle\pi^{N}_{s},G \rangle\right]ds,$$ is a martingale with respect to the natural filtration  $\{\mathcal{F}_{t}\}_{t\ge 0}$.
From the computations of Appendix \ref{sec:gen_comp} we get that the term inside the time integral in the previous display is equal to 
\begin{eqnarray*}\label{T5}
&&\dfrac{\Theta(N)}{(N-1)^{2}} \sum_{x<y\in\Lambda_{N}}\left( G\left(\tfrac{x}{N} \right) -G\left(\tfrac{y}{N} \right)\right)^{2}p(x-y)(\eta^N_{y}(s)-\eta^N_{x}(s))^{2}\nonumber\\
&+&\dfrac{\Theta(N)\kappa}{N^{\theta}(N-1)^{2}}\sum_{x\in\Lambda_{N}} G^{2}\left(\tfrac{x}{N}\right) r_N^-(\tfrac{x}{N})(\alpha-\eta^N_{x}(s))(1-2\eta^N_{x}(s)) \nonumber\\
&+& \dfrac{\Theta(N)\kappa}{N^{\theta}(N-1)^{2}}\sum_{x\in\Lambda_{N}} G^{2}\left(\tfrac{x}{N}\right) r_N^+(\tfrac{x}{N})(\beta-\eta^N_{x}(s))(1-2\eta^N_{x}(s)). \nonumber\\
\end{eqnarray*}
Since $\Theta (N) \le N^2$ and $G'$ is bounded it is easy to see that the absolute value of the previous display is bounded from above by a constant times
\begin{eqnarray}
\label{T6}
&&\dfrac{1}{(N-1)^{2}} \sum_{x,y\in\Lambda_{N}} (x-y)^{2}p(x-y)\nonumber
+\dfrac{\Theta(N)}{N^{\theta}(N-1)^{2}}\sum_{ x\in\Lambda_{N}} G^{2}\left(\tfrac{x}{N}\right)\Big( r_N^-(\tfrac{x}{N})+r_N^+(\tfrac{x}{N})\Big)\nonumber\\
\end{eqnarray}
Since $\sum_{x,y\in\Lambda_{N}} (x-y)^{2}p(x-y) = \mathcal{O}(N)$ the first term in (\ref{T6}) is $\mathcal{O}(N^{-1})$. For the second term at the right hand side of \eqref{T6}, we split the argument according to the cases $\theta \ge 1$ and $\theta<1$. First when $\theta\geq1$, by using  the fact that $\gamma>2$ and $G$ is bounded so that the sum in that term is finite, and since $\Theta(N)=N^2$, we conclude that  the term is $\mathcal{O}(N^{-\theta})\leq \mathcal {O}(N^{-1})$. From this we obtain \eqref{TC2}. Now if $\theta<1$, recall that $G$ has compact support and (\ref{eq:rs-r+-}). We then write
$$\dfrac{\Theta(N)}{N^{\theta}(N-1)^{2}}\sum_{ x\in\Lambda_{N}} G^{2}\left(\tfrac{x}{N}\right)\Big( r_N^-(\tfrac{x}{N})+r_N^+(\tfrac{x}{N})\Big) = \dfrac{\Theta(N)}{N^{\theta+\gamma} (N-1)}\;  I_N (G)$$
where $I_N (G)$ is a Riemann sum converging to $\int_0^1 G^2 (q) \big[ r^- (q) + r^+ (q) \big] dq <\infty.$ Therefore the second term in (\ref{T6}) is of order $\Theta (N) N^{-1-\theta-\gamma}= {\mc O} (N^{-1})$ by (\ref{time_scales}). 
\\

This ends the proof of tightness in the case $\theta\geq 1$, since  $C^{2}([0,1])$ is a dense subset of $C([0,1])$ with respect to the uniform topology. Nevertheless, for $\theta < 1$, we have proved (\ref{TC1}) and (\ref{TC2}), and thus (\ref{T1}), only for functions $G\in C^{2}_{c}(0,1)$ and we need to extend this result to functions in $C^1([0,1])$. To accomplish that,  we take a function $G \in C^1([0,1])\subset L^{1}([0,1])$,  and we take  a sequence of functions $\{G_{k}\}_{k \geq 0} \in C^{2}_{c}(0,1)$ converging to $G$ with respect to the $L^{1}$-norm as $k \to \infty$. Now, since the probability in (\ref{T1}) is less or equal than 
\begin{eqnarray*}\label{T7} \nonumber
 &&{\mathbb{P}}_{\mu _{N}}\Big(\eta_{\cdot}^{N}\in {\mathcal D} ( [0,T], \Omega_{N}) :\left\vert \langle\pi^{N}_{\tau+ \bar\tau},G_{k}\rangle-\langle\pi^{N}_{\tau},G_{k}\rangle\right\vert > \dfrac{\ve}{2} \Big)  \\\nonumber
&+&{\mathbb{P}}_{\mu _{N}}\Big(\eta_{\cdot}^{N}\in {\mathcal D} ( [0,T], \Omega_{N}) :\left\vert \langle\pi^{N}_{\tau+ \bar\tau},G-G_{k}\rangle-\langle\pi^{N}_{\tau},G-G_{k}\rangle\right\vert > \dfrac{\ve}{2} \Big) \\
\end{eqnarray*}
and  since $G_{k}$ has compact support, from the computation above, it remains only to check that the last probability vanishes as $N \to \infty$ and then $k \to \infty$. 
For that purpose, we use the fact that 
\begin{equation}\label{T8}
\left\vert \langle\pi^{N}_{\tau+ \bar\tau},G-G_{k}\rangle-\langle\pi^{N}_{\tau},G-G_{k}\rangle\right\vert\leq \dfrac{2}{N}\sum _{x\in \Lambda_{N}}\left\vert(G-G_{k})(\tfrac{x}{N})\right\vert, 
\end{equation}
and we use the estimate
\begin{equation*}
\begin{split}
\dfrac{1}{N}\sum _{x\in \Lambda_{N}}\left\vert(G-G_{k})(\tfrac{x}{N})\right\vert & \le \sum _{x\in \Lambda_{N}}  \int_{x/N}^{(x+1)/N} \left\vert(G-G_{k})(\tfrac{x}{N}) -(G-G_k) (q) \right\vert dq \\
&+\,  \int_{0}^{1}\vert (G-G_{k})(q)\vert dq\\
& \le \cfrac{1}{N} \| (G- G_k)' \|_{\infty} + \int_{0}^{1}\vert (G-G_{k})(q)\vert dq.
\end{split}
\end{equation*}
We conclude the result by taking first the limsup in $N\to\infty$ and then in $k \to \infty$.
\end{proof}
\end{prop}

\section{Replacement lemmas and auxiliary results}\label{sec:RL}
In this section we establish some technical results needed in the proof of the hydrodynamic limit. In what follows, we will suppose without loss of generality that $\alpha \leq \beta$. Let  $h : [0, 1] \rightarrow [0, 1]$ be a Lipschitz function such that  $\alpha \leq h(q)\leq\beta$, for all $q\in [0,1]$. Let $\nu_{h(\cdot)}^{N}$ be the Bernoulli product measure on 
$\Omega_{N}$ with marginals given by
$$\nu_{h(\cdot)}^{N}\lbrace \eta: \eta_{x} = 1 \rbrace = h\left( \tfrac{x}{N}\right).$$

Given two functions $f,g:\Omega_{N} \to \RR$ and a probability measure $\mu$ on $\Omega_{N}$, we denote here by $\langle f, g\rangle_{\mu }$ the scalar product between $f$ and $g$ in $L^2 (\Omega_N, \mu)$, that is,
$$ \langle f, g\rangle_{\mu } = \int _{\Omega_{N}} f(\eta)g(\eta) \, d\mu .$$
The notation above should note be mistaken to the notation that we introduced in Section \ref{sec:hyd_eq}.
We denote  by $H_{N}(\mu\vert \nu_{h(\cdot)}^{N})$ the relative entropy of a probability measure $\mu$ on 
$\Omega_{N}$ with respect to the probability measure $\nu_{h(\cdot)}^{N} $ on $\Omega_{N}$. It is easy to prove the existence of a constant $C_0 :=C_0(\alpha,\beta)$, such that
\begin{equation}\label{H}
H_{N}(\mu\vert \nu_{h(\cdot)}^{N})\leq NC_0.
\end{equation}
In fact, using the explicit formula for the  entropy and the definition of the product measure $\nu_{h(\cdot)}^{N}$, we get that
\begin{eqnarray*}
H(\mu|\nu_{h(\cdot)}^{N}) &=& \sum_{\eta\in \Omega_{N}}\mu(\eta)\log\left(\dfrac{\mu(\eta)}{\nu_{h(\cdot)}^{N}(\eta)}\right)
\leq  \sum_{\eta\in \Omega_{N}}\mu(\eta)\log\left(\dfrac{1}{\nu_{h(\cdot)}^{N}(\eta)}\right)\\
&\leq & \log\left(\left[\dfrac{1}{\alpha \wedge (1-\beta)}\right]^{N} \right)\sum_{\eta\in \Omega_{N}}\mu(\eta)
\leq  N\log\left(\dfrac{1}{\alpha \wedge (1-\beta)}\right)\leq NC_0.
\end{eqnarray*}

\subsection{Estimates on Dirichlet forms}
For a probability measure $\mu$ on $\Omega_N$, $x,y \in \Lambda_N$ and a density function $f:\Omega_N \to [0,\infty)$ with respect to $\mu$ we introduce 
\begin{eqnarray*}
I_{x,y}(\sqrt f,\mu)&:=& \int_{\Omega_N} \left(\sqrt {f(\sigma^{x,y}\eta)}-\sqrt {f(\eta)}\right)^{2} d\mu,\\
I_{x}^{\alpha}(\sqrt f,\mu)&:=& \int_{\Omega_N}  c_{x}(\eta;\alpha)\left(\sqrt {f(\sigma^{x}\eta)}-\sqrt {f(\eta)}\right)^{2} d\mu.
\end{eqnarray*}
Then we define
\begin{eqnarray*}
D_{N}(\sqrt{f},\mu )&:=& (D_{N}^{0}+D_{N}^{\ell}+D_{N}^{r})(\sqrt{f},\mu) \end{eqnarray*}
where 
\begin{eqnarray} 
\label{left_rig_form}
D_{N}^{0}(\sqrt{f},\mu):=\cfrac{1}{2}\sum_{x,y\in\Lambda_N}p(y-x)\, I_{x,y}(\sqrt{f},\mu), 
\end{eqnarray}
\begin{eqnarray}
\label{left_dir_form}
D_{N}^{\ell}(\sqrt{f},\mu):=\frac{\kappa}{N^\theta}\sum_{x\in\Lambda_N}\sum_{y\leq 0}p(y-x)\, I^\alpha_{x}(\sqrt{f},\mu)=\frac{\kappa}{N^\theta}\sum_{x\in\Lambda_N}r_N^-(\tfrac{x}{N})I^\alpha_{x}\, (\sqrt{f},\mu) 
\end{eqnarray} 
and $D_{N}^{r}(\sqrt{f},\mu)$ is the same as $D_{N}^{\ell}(\sqrt{f},\mu)$ but in $I^\alpha_{x}(\sqrt{f},\mu)$ the parameter $\alpha$ is replaced by $\beta$ and $r_N^-(\cdot)$ is replaced by $r_N^+(\cdot)$.

Our first goal is to express, for the measure $\nu_{h(\cdot)}^{N}$,  a relation between the Dirichlet form defined by $\langle L_N\sqrt{f},\sqrt{f} \rangle_{\nu_{h(\cdot)}^{N}}$ and $
D_{N}(\sqrt{f},\nu_{h(\cdot)}^{N} )$.
More precisely, we claim that for any positive constant $B$, there exists a constant $C>0$ such that
\begin{equation}\label{dir_est}
\begin{split}
\frac{1}{BN}\langle L_{N}\sqrt{f},\sqrt{f} \rangle_{\nu_{h(\cdot)}^N} &\leq -\dfrac{1}{4BN}D_{N}(\sqrt{f},\nu_{h(\cdot)}^N) + \frac{C}{BN}\sum_{x,y\in\Lambda_N}p(y-x)\Big(h(\tfrac xN)-h(\tfrac yN)\Big)^2\\ &+ \frac{C\kappa}{BN^{1+\theta}} \sum_{x\in\Lambda_N} \left\{ \Big(h(\tfrac xN)-\alpha)^2r_N^{-}(\tfrac xN)+\Big(h(\tfrac xN)-\beta\Big)^2r_N^+(\tfrac xN)\right\}.
\end{split}
\end{equation}
Our aim is then to choose $h(\cdot)$ in order to minimize the error term, i.e. the two last terms at  the right hand side of the previous inequality. 

If $h(\cdot)$ is such that $h(0)=\alpha$ and $h(1) =\beta$, since it is assumed to be Lipschitz, we get the estimate
\begin{equation}
\label{dir_est_lip}
\begin{split}
\frac{N}{B}\langle L_{N}\sqrt{f},\sqrt{f} \rangle_{\nu_{h(\cdot)}^N} &\leq -\dfrac{N}{4B}D_{N}(\sqrt{f},\nu_{h(\cdot)}^N) + \frac{C}{B}\sigma^2\\& + \frac{C\kappa}{BN^{1+\theta}}\sum_{x\in\Lambda_N}\Big\{x^2r_N^-(\tfrac{x}{N})+\big( x-N\big)^2r^+_N(\tfrac{x}{N})\Big\}.
\end{split}
\end{equation}
Moreover, if the function $h(\cdot)$ is such that $h(0)=\alpha$ and $h(1) =\beta$, H\"older of parameter $\gamma/2$ at the boundaries and Lipschitz inside, then we have 
\begin{equation}
\label{dir_est_holder}
\begin{split}
\frac{N}{B}\langle L_{N}\sqrt{f},\sqrt{f} \rangle_{\nu_{h(\cdot)}^N} \leq& -\dfrac{N}{4B}D_{N}(\sqrt{f},\nu_{h(\cdot)}^N) + \frac{C}{B} \sigma^2 + \frac{C\kappa}{BN^{\gamma+\theta-2}}.
\end{split}
\end{equation}
On the other hand if the function $h(\cdot)$ is constant, equal to $\alpha$ or to $\beta$, then we have 
\begin{equation}\label{dir_est_const}
\begin{split}
\frac{N}{B}\langle L_{N}\sqrt{f},\sqrt{f} \rangle_{\nu_{\alpha}} &\leq  -\dfrac{N}{4B}D_{N}(\sqrt{f},\nu_{\alpha}) + \frac{C\kappa}{B} N^{1-\theta}.
\end{split}
\end{equation}

In order to prove \eqref{dir_est} we need some intermediate results. In what follows $C$ is a constant depending on $\alpha$ and $\beta$ whose value can change from line to line. 

\begin{lem}
\label{lemmaleft0}
Let $T : \eta \in  \Omega_N \to T(\eta) \in \Omega_N$ be a transformation and $c: \eta \to c(\eta)$ be a positive local function. Let $f$ be a density with respect to a probability  measure $\mu$ on $\Omega_N$. Then, we have that
\begin{eqnarray}
\label{use_comp}
\begin{split}
&\left\langle c (\eta) [ \sqrt{f(T (\eta))} -\sqrt{f(\eta)}]\; ,\;\sqrt{f (\eta)} \right\rangle_{\mu}  \\
&\le  -\dfrac{1}{4}\int c (\eta)\left(\left[ \sqrt{f(T (\eta))}\right]-\left[ \sqrt{f(\eta)}\right]\right)^{2}  d\mu\\
& + \dfrac{1}{16}\int \dfrac{1}{c (\eta)}\left[c (\eta)-c (T(\eta)) \dfrac{\mu (T(\eta))}{\mu(\eta)} \right]^{2}\left(\left[ \sqrt{f(T(\eta))}\right]+\left[ \sqrt{f(\eta)}\right]\right)^{2}  d\mu.
\end{split}
\end{eqnarray}
\end{lem}

\begin{proof}
{By writing the term at the left hand side of \eqref{use_comp} as its half plus its half and summing and subtracting the term needed to complete the square as written in the first term at the right hand side of \eqref{use_comp}}, we have that
\begin{equation*}
\begin{split}
\int c (\eta)& \left[ \sqrt{f( T(\eta))}-\sqrt{f(\eta)}\right]\sqrt{f(\eta)} \; d\mu\\
&=-\dfrac{1}{2}\int c (\eta) \left[ \sqrt{f(T(\eta))}-\sqrt{f(\eta)}\right]^{2} \; d\mu\\
&+\dfrac{1}{2}\int \left[ \sqrt{f(T(\eta))}\right]^{2} \left[c (\eta)-c (T(\eta)){\dfrac{\mu (T(\eta))}{\mu(\eta)}}\right]\; d\mu.
\end{split}
\end{equation*}
Repeating again the same argument,  the second term at the right hand side of last expression can be written as 
\begin{eqnarray*}
\dfrac{1}{4}\int \left(\left[ \sqrt{f( T(\eta))}\right]^{2}-\left[ \sqrt{f(\eta)}\right]^{2}\right) \left[c (\eta)-c ( T(\eta)) {\dfrac{\mu (T(\eta))}{\mu(\eta)}}\right] d\mu.\\
\end{eqnarray*}
By Young's inequality and the elementary equality $a^2-b^2=(a-b)(a+b)$, last expression  is bounded from above by
\begin{eqnarray*}
&& \dfrac{1}{4}\int c (\eta)\left(\left[ \sqrt{f(T(\eta))}\right]-\left[ \sqrt{f(\eta)}\right]\right)^{2}  d\mu\\
&+& \dfrac{1}{16}\int \dfrac{1}{c (\eta)}\Big[c (\eta)-c (T(\eta)) {\dfrac{\mu (T(\eta))}{\mu(\eta)}}\Big]^{2}\left(\left[ \sqrt{f( T(\eta))}\right]+\left[ \sqrt{f(\eta)}\right]\right)^{2}  d\mu,\\
\end{eqnarray*}
which finishes the proof.
 \end{proof}
 
\begin{lem}
\label{lem:densaf}
There exists a constant $C:=C(h)$ such that for any $N \ge 1$ and density $f$ be a density with respect to $\nu_{h(\cdot)}^N$
\begin{equation*}
\sup_{x\ne y \in \Lambda_N} \int_{\Omega_N} f (\sigma^{x,y} \eta) \; d\nu_{h(\cdot)}^N (\eta) \; \le \; C, \quad \quad  \sup_{x \in \Lambda_N} \int_{\Omega_N} f (\sigma^{x} \eta) \; d\nu_{h(\cdot)}^N (\eta) \; \le \; C.
\end{equation*} 
\end{lem} 
 
 \begin{proof}
Let us prove only the first bound since the proof of the second one is similar. We perform in the first integral above  the change of variable $\omega=\sigma^{x,y} \eta$ and we use that uniformly in $x,y \in \Lambda_N$ and $\omega$ we have 
$$\theta^{x,y} (\omega) =\cfrac{\nu_{h(\cdot)}^N (\sigma^{x,y} \omega) }{\nu_{h(\cdot)}^N (\omega)} =1 +\mathcal{O}(\tfrac{1}{N}).$$
By using the fact that $f$ is a density it is easy to conclude.
 \end{proof}
 
Now, let us look at some consequences of these lemmas. We start with the bulk generator $L_N^0$ given in \eqref{generators}.
\begin{cor}
\label{lemmabulk}
There exists a constant $C>0$ (independent of $f$ and $N$) such that
\begin{eqnarray*}
\left\langle L_{N}^0\sqrt{f},\sqrt{f}\right \rangle_{\nu_{h(\cdot)}^N}  &\le&  -\dfrac{1}{4}D_{N}^{0}(\sqrt{f},\nu_{h(\cdot)}^N) + C\sum_{x,y\in\Lambda_N}p(y-x)\Big(h(\tfrac xN)-h(\tfrac yN)\Big)^2
\end{eqnarray*}
for any density $f$ with respect to $\nu_{h(\cdot)}^N$.
\end{cor}
\begin{proof}
To prove this we note that 
\begin{equation*}
\left\langle L_{N}^0\sqrt{f},\sqrt{f} \right\rangle_{\nu_{h(\cdot)}^N} =\frac{1}{2} \sum_{x,y\in\Lambda_N}p(y-x)\; \left\langle \Big[ \sqrt{f(\sigma^{x,y}\eta)} -\sqrt{f(\eta)}\Big]\; ,\;\sqrt{f (\eta)} \right\rangle_{\nu_{h(\cdot)}^N}.
\end{equation*}
Now, by Lemma \ref{lemmaleft0} with $c \equiv1$, $T=\sigma^{x,y}$,  and Lemma \ref{lem:densaf} last expression is bounded from above by
\begin{eqnarray*}
 -\dfrac{1}{4}D_{N}^{0}(\sqrt{f},\nu_{h(\cdot)}^N) + C\sum_{x,y\in\Lambda_N}p(y-x)\Big(h\big(\tfrac xN \big)-h\big(\tfrac yN\big)\Big)^2,
\end{eqnarray*}
because $| \theta^{x,y} (\eta) -1 |^2 \lesssim (h(x/N) -h (y/N))^2$.
\end{proof}
Now we look at the generators of the reservoirs given in \eqref{generators}.
  
\begin{cor}
\label{lemmaleft1} 
Let $\theta \in \RR$ be fixed. There exists a constant $C>0$ (independent of $f$ and $N$) such that
\begin{equation}
\begin{split}
& \langle L_{N}^\ell\sqrt{f},\sqrt{f} \rangle_{\nu_{h(\cdot)}^N} \le  -\dfrac{1}{4}D_{N}^{\ell}(\sqrt{f},\nu_{h(\cdot)}^N) + \frac{C\kappa }{N^{\theta}}\sum_{x\in\Lambda_N}r_N^{-}(\tfrac{x}{N})\Big(h(\tfrac xN)-\alpha\Big)^2, \\
 & \langle L_{N}^r\sqrt{f},\sqrt{f} \rangle_{\nu_{h(\cdot)}^N} \le  -\dfrac{1}{4}D_{N}^{r}(\sqrt{f},\nu_{h(\cdot)}^N) + \frac{C\kappa}{N^{\theta}}\sum_{x\in\Lambda_N}r_N^{+}(\tfrac{x}{N})\Big(h(\tfrac xN)-\beta\Big)^2
\end{split}
\end{equation}
for any density $f$ with respect to $\nu_{h(\cdot)}^N$.
\end{cor}
\begin{proof}
We present the proof for the first inequality but we note that the proof of the second one is analogous. 
First observe that 
\begin{equation*}
\left\langle L_{N}^\ell\sqrt{f},\sqrt{f} \right\rangle_{\nu_{h(\cdot)}^N} =\frac{\kappa}{N^\theta}\sum_{x\in\Lambda_N}\sum_{y\leq0}p(y-x)\left\langle c_{x}(\eta;\alpha) \Big[\sqrt{f(\sigma^{x}\eta)} -\sqrt{f(\eta)}\Big]\; , \;\sqrt{f (\eta)} \right\rangle_{\nu_{h(\cdot)}^N}.
\end{equation*}  
Now, by using Lemma \ref{lemmaleft0} with $c (\eta) = c_{x}(\eta;\alpha)$, $T = \sigma^{x}$ and Lemma \ref{lem:densaf}, last expression is bounded from above by
\begin{eqnarray*}
 -\dfrac{1}{4}D_{N}^{\ell}(\sqrt{f},\nu_{h(\cdot)}^N) + \frac{C\kappa}{N^\theta}\sum_{x\in\Lambda_N}\sum_{y\leq0}p(y-x)\Big(h(\tfrac xN)-\alpha\Big)^2.
\end{eqnarray*}
 \end{proof}
From the two previous corollaries the claim (\ref{dir_est}) follows.

\subsection{Replacement Lemmas}
\begin{lem}
 \label{bound}
For any density $f$ with respect to $\nu_{h(\cdot)}^N$,  any $x\in \Lambda_{N}$ and any positive constant $A_x$, we have that
$$\left\vert \left\langle t_{x}^{\alpha},f\right\rangle_{\nu_{h(\cdot)}^{N}} \right\vert \; \lesssim\; \dfrac{1}{A_{x}}  I_{x}^{\alpha}(\sqrt{f},\nu_{h(\cdot)}^{N})+ A_{x}+[h(\tfrac{x}{N})- \alpha],$$
where $t_x^{\alpha}(\eta)=\eta_x-\alpha$. The same result holds if $\alpha$ is replaced by $\beta$.
\end{lem}
\begin{proof}
By a simple computation  we have that:
\begin{eqnarray}
\label{Rep.sec.1}
\nonumber
\displaystyle \left\vert \left\langle t_{x}^{\alpha},f\right\rangle_{\nu_{h(\cdot)}^{N}} \right\vert
&\leq & \dfrac{1}{2}\left\vert \int t_x^\alpha (\eta)(f(\eta)-f(\sigma^x\eta)) \; d\nu_{h(\cdot)}^{N}\right\vert \\ 
&+& \dfrac{1}{2} \left\vert \int  [f(\sigma^{x}\eta)+f(\eta)]t^\alpha_x(\eta) \; d\nu_{h(\cdot)}^{N} \right\vert,
\end{eqnarray}
where $\sigma^{x}$ is the flip given in  \eqref{tranformations}.
By Young's inequality, using the fact that $(a-b)=(\sqrt a -\sqrt b)(\sqrt a+\sqrt b)$ for all $a,b\ge 0$ and Lemma \ref{lem:densaf}, the first term at the right side of (\ref{Rep.sec.1}) is bounded from above, for any positive constant $A_{x}$, by
\begin{equation*}
\begin{split}
& \dfrac{A_{x}}{4}\int \dfrac{(t^\alpha_{x}(\eta))^2}{c_{x}(\eta;\alpha)}\left(\left[ \sqrt{f(\sigma^{x}\eta)}\right]+\left[ \sqrt{f(\eta)}\right]\right)^{2}  d\nu_{h(\cdot)}^{N}+\dfrac{I_{x}^{\alpha}(\sqrt{f},\nu_{h(\cdot)}^{N})}{4A_{x}}\\
&\lesssim \; A_x +  \dfrac{I_{x}^{\alpha}(\sqrt{f},\nu_{h(\cdot)}^{N})}{A_{x}}.
\end{split}
\end{equation*}
Now, we look at the second term at the right hand side of (\ref{Rep.sec.1}). By using the fact that $\nu^N_{h(\cdot)}$ is product and denoting by $\bar \eta$ the configuration $\eta$ removing its value at $x$ so that $(\eta_x,\bar\eta)=\eta$ , we have that the second term at the right side of (\ref{Rep.sec.1}) is equal to
\begin{equation*}
\begin{split}
&\frac{1}{2}\big\vert \sum_{\bar{\eta} }\left( (1-\alpha)(f(1,\bar\eta)+f(0,\bar\eta))\nu_{h(\cdot)}^{N}(\eta_{x}=1)\right.\\
& -\left.\alpha(f(0,\bar\eta)+f(1,\bar\eta))\nu_{h(\cdot)}^{N}(\eta(x)=0)\right)\nu_{h(\cdot)}^{N}(\bar\eta)\big\vert\\
=&\frac{1}{2}\Big\vert\sum_{\bar{\eta}}\Big(h(\tfrac xN)-\alpha\Big)(f(0,\bar\eta)+f(1,\bar\eta))\nu_{h(\cdot)}^{N}(\bar\eta)\Big\vert\\
\lesssim &\, \Big(h(\tfrac xN)-\alpha\Big)\sum_{\bar\eta}h(\tfrac xN)f(1,\bar\eta)\nu_{h(\cdot)}^{N}(\bar\eta)+\Big(1-h(\tfrac xN)\Big)f(0,\bar\eta)\nu_{h(\cdot)}^{N}(\bar\eta)\\
 =&\Big(h(\tfrac xN)-\alpha\Big)\sum_{\eta \in \Omega_N} f(\eta) \nu_{h(\cdot)}^N (\eta) = \Big(h(\tfrac xN)-\alpha\Big)
\end{split}
\end{equation*}
because $\max_{x\in \Lambda_N} \Big\{\frac{1}{2h\left(\tfrac{x}{N}\right)},\frac{1}{2\left(1-h\left(\tfrac{x}{N}\right)\right)}\Big\}$ is bounded from above by a constant depending only on $\alpha$ and $\beta$. Above $f(1,\bar\eta)$ (resp. $f(0,\bar \eta)$) means that we are computing $f(\eta)$ with  $\eta_x=1$ (resp. $\eta_x=0$).
 \end{proof}
 \begin{lem} 
\label{Rep-Dirichlet1}
Let $\theta>1$. For any $t>0$, we have that 
\begin{equation}
\label{eq:A1234}
\begin{split}
{\limsup_{N
\to\infty}}\,\bb E_{\mu_N}\left[\Big|\int_0^t N^{1-\theta} \sum _{x \in \Lambda_{N}}Gr^{-}_{N}(\tfrac{x}{N})(\eta_x(sN^2)-\alpha)\, ds\Big|\right]=0,\\
{\limsup_{N
\to\infty}}\,
\bb E_{\mu_N}\left[\Big|\int_0^t N^{1-\theta} \sum _{x \in \Lambda_{N}}Gr^{+}_{N}(\tfrac{x}{N})(\eta_x(sN^2)-\beta)\, ds\Big|\right]=0,
\end{split}
\end{equation} 
 {for any bounded $G:\RR\rightarrow \RR$.}
\end{lem}
\begin{proof}
We present the proof for the first term, but we note that  the proof for the second term is completely analogous. 

We start by  fixing a Lipschitz profile $h(\cdot)$ such that $h(0)=\alpha$ and $h(1)=\beta$. By the entropy and Jensen's inequalities, for any $B>0$,  the first expectation of (\ref{eq:A1234}) is bounded from above by
\begin{equation}\label{eq:Ent+FK}
\begin{split}
\dfrac{H(\mu _{N}\vert \nu_{h(\cdot)}^{N}) }{BN}
+ \dfrac{1}{BN}\log \mathbb{E}_{\nu_{h(\cdot)}^{N}}\left[ e^{BN\vert \int _{0}^{t}N^{1-\theta} \sum_{x\in \Lambda_{N}} Gr^{-}_{N}\big(\tfrac{x}{N}\big)(\eta_x(sN^2)-\alpha) ds\vert}\right].\\
\end{split}
\end{equation}
We can remove the absolute value inside the exponential since $e^{\vert x\vert} \leq e^{x}+e^{-x}$ and  \begin{equation}\label{Log bounded}
\limsup_{N\rightarrow\infty} N^{-1}\log(a_{N}+b_{N}){=} \max \left\lbrace \limsup_{N\rightarrow\infty} N^{-1}\log(a_{N}), \limsup_{N\rightarrow\infty} N^{-1}\log(b_{N}) \right\rbrace.
\end{equation}
By \eqref{H} and Feynman-Kac's formula,  \eqref{eq:Ent+FK}  is bounded from above by 
\begin{eqnarray}\label{FK 1}\nonumber
\frac{C_0}{B}+ t \sup _{f}\Big\{ N^{1-\theta}\sum _{x \in \Lambda_{N}}\left\vert G r^{-}_{N}(\tfrac{x}{N})\langle t_{x}^{\alpha},f\rangle_{ \nu_{h(\cdot)}^{N}}\right\vert  + \dfrac{N}{B} \left\langle L_{N}\sqrt{f},\sqrt{f}\right\rangle_{ \nu_{h(\cdot)}^{N}}\Big\},
\end{eqnarray}
where the supremum is carried over all the densities $f$ with respect to $\nu_{h(\cdot)}^N$. We recall that $t_{x}^{\alpha}(\eta) = \eta_{x} -\alpha$. 
From  Lemma \ref{bound}  we have that there exists a constant $C:=C(\alpha,\beta, \gamma)>0$ such that
\begin{equation}
\label{eq:intermezzo}
\begin{split}
& N^{1-\theta}\sum _{x \in \Lambda_{N}}\left\vert (Gr^{-}_{N})(\tfrac{x}{N})\langle t_{x}^{\alpha},f\rangle_{ \nu_{h(\cdot)}^{N}}\right\vert \\
& \leq  C N^{1-\theta} \sum_{x\in\Lambda_{N}}\vert (Gr_N^-)(\tfrac{x}{N})\vert  \left[A_{x}+ \tfrac{I_{x}^{\alpha}(\sqrt{f},\nu_{h(\cdot)}^{N})}{A_{x}} + \tfrac{x}{N}\right]\\
&\leq   4C^2 \kappa^{-1} B N^{1-\theta} \sum_{x\in\Lambda_N}G^2(\tfrac{x}{N})r^{-}_{N}(\tfrac{x}{N})+\dfrac{N}{4B} D_{N}^{\ell}(\sqrt{f},\nu_{h(\cdot)}^{N})+ CN^{-\theta}\sum_{x\in\Lambda_N}\vert G(\tfrac{x}{N})\vert r^{-}_{N}(\tfrac{x}{N})x.
\end{split}
\end{equation}
The last inequality is obtained by choosing $A_{x}=4 \kappa^{-1} C\vert G(\tfrac{x}{N})\vert B$. Recall  \eqref{dir_est_lip}. 


{Since $\theta >1$ and the function $G$ is  bounded, we use (\ref{eq:intermezzo}) and \eqref{dir_est_lip} and we estimate from above (\ref{eq:Ent+FK})  by a constant times 
\begin{equation}
 \begin{split}
\frac{1}{B} + \frac{1}{BN^{1+\theta}}\sum_{x\in\Lambda_N}\Big\{x^2r_N^-(\tfrac{x}{N})+\big(x-N \big)^2r^+_N(\tfrac{x}{N})\Big\}
+ B N^{1-\theta}\sum_{x\in\Lambda_N}r^{-}_{N}(\tfrac{x}{N})+ N^{-\theta} \sum_{x\in\Lambda_N}r^{-}_{N}(\tfrac{x}{N})x,
\end{split}
\end{equation}
which, by 
\begin{equation}\label{estimates_on_p}
\sum_{x\in\Lambda_N}xr^{-}_{N}(\tfrac{x}{N}) 
\lesssim \; 
\begin{cases}
N^{3-\gamma}, \quad \gamma \in (2,3),\\
\log N, \quad \gamma=3,\\
1, \quad \gamma >3,
\end{cases}
\end{equation}
 and (\ref{mean}),  goes to zero as $N \to \infty$ and then $B\to\infty$.}
\end{proof}

Let us define  for $\ell\in\mathbb N$ the following empirical densities 
\begin{equation} \label{boxes}
\overrightarrow{\eta}_{0}^{\ell}:=\dfrac{1}{\ell}\sum_{y=1}^{\ell}\eta_{y} \qquad\text{and} \qquad  \overleftarrow{\eta}_{N}^{\ell}:=\dfrac{1}{\ell}\sum_{y=N-1 -\ell}^{N-1}\eta_{y}.
\end{equation}
 
 \begin{lem} \label{Rep-Neumann}
For any $t>0$ and any   $ \theta \geq 1$ we have that
\begin{equation*}
\begin{split}
&{\limsup_{\ve \to 0}}\,{\limsup _{N\rightarrow \infty}}\,\bb E_{\mu_N}\left[\Big|\int_0^t\sum_{x\in\Lambda_N}\Theta_x^- \, (\eta_x(sN^2)-\overrightarrow{\eta}_{0}^{\ve N}(sN^2))\, ds\Big|\right] =0,\\
&{\limsup_{\ve \to 0}}\,{\limsup _{N\rightarrow \infty}}\,\bb E_{\mu_N}\left[\Big|\int_0^t\sum_{x\in\Lambda_N}\Theta_x^+ \, (\eta_{x}(sN^2)-\overleftarrow{\eta}_{N}^{\ve N}(sN^2))\, ds\Big|\right] =0.
\end{split}
\end{equation*}
\end{lem}
\begin{proof}
We present the proof for the first term, but we note that the proof for the second one it is analogous. Here we take as reference measure the Bernoulli product measure with constant parameter (for example $\alpha$) and we recall  \eqref{dir_est_const}. By the entropy and Jensen's inequalities  the expectation in the statement of the lemma is bounded from above, for any $B>0$, by
\begin{equation*}
\begin{split}
\dfrac{H(\mu _{N}\vert \nu^{N}_{\alpha})}{BN}
+ \dfrac{1}{BN}\log \mathbb{E}_{\nu^{N}_{\alpha}}\left[ e^{BN\vert \int _{0}^{t}\sum _{x \in \Lambda_{N}}\Theta_x^-(\eta_x(sN^2)-\overrightarrow{\eta}_{0}^{\ve N}(sN^2))\, ds\vert}\right]\\
\end{split}.
\end{equation*}
As in the previous proof, we can remove the absolute value inside the exponential, so that by \eqref{H} and by Feynman-Kac's formula last expression can be estimated from above by
\begin{eqnarray} 
\label{repla_term}
&&\frac {C_0} {B}+t \sup _{f}\Big\{\sum _{x \in \Lambda_{N}}\Theta_x^-\langle \tau_{x}^{\ve N},f\rangle_{\nu^{N}_\alpha}  + \dfrac{N}{B}\left \langle L_{N}\sqrt{f},\sqrt{f}\right\rangle_{\nu^{N}_\alpha}\Big\},\\ \nonumber
\end{eqnarray}
where the supremum is carried over all the densities $f$ with respect to $\nu_{\alpha}^N$. Here $\tau_{x}^{\ve N}(\eta) = \eta_{x} - \overrightarrow{\eta}_{0}^{\ve N}$.

Now we have to split the sum in $x$, depending on wether  $N-1 \ge x \ge \ve N$ or $x\leq \ve N-1$. We start by the first case and we have 
\begin{eqnarray*}
\langle \tau_{x}^{\ve N},f\rangle_{\nu^{N}_\alpha} &=& \dfrac{1}{\ve N}\sum _{y=1}^{\ve N} \int (\eta_{x}-\eta_{y})f(\eta)\; d\nu^{N}_\alpha\\
&=& \dfrac{1}{\ve N}\sum _{y=1}^{\ve N}\sum_{z=y}^{x-1}\int (\eta_{z+1}-\eta_{z})f(\eta)\; d\nu^{N}_\alpha.
\end{eqnarray*}
By writing the previous term as its half plus its half and by performing in one of the terms the change of variables 
$\eta$ into $\sigma^{z,z+1}\eta$, for which the measure $\nu^N_\alpha$ is invariant, we write it as 
$$\dfrac{1}{2\ve N}\sum _{y=1}^{\ve N}\sum_{z=y}^{x-1}\int (f(\eta)-f(\sigma^{z,z+1}\eta))
(\eta_{z+1}-\eta_{z})\, d{\nu_\alpha^N}.$$
By using the fact that $(a-b)=(\sqrt a -\sqrt b)(\sqrt a+\sqrt b)$ for any $a,b \ge 0$ and since $ab \leq \dfrac{Aa^{2}}{2}+\dfrac{b^{2}}{2A}$ for all $A>0$, we have that 
\begin{equation}
\label{est_repla_term_1}
\begin{split}
\sum _{x=\ve N}^{N-1} \Theta_x^-\langle \tau_{x}^{\ve N},f\rangle_{\nu^{N}_\alpha} &\le \cfrac{A}{2} \, \sum _{x=\ve N}^{N-1} \Theta_x^-\dfrac{1}{2\ve N}\sum _{y=1}^{\ve N}\sum_{z=y}^{x-1}\int (\sqrt {f(\eta)}-\sqrt{f(\sigma^{z,z+1}\eta)})^2 d{\nu^{N}_\alpha}\\
&+ \cfrac{1}{2A} \, \sum _{x=\ve N}^{N-1} \Theta_x^-\dfrac{1}{2\ve N}\sum _{y=1}^{\ve N}\sum_{z=y}^{x-1}\int (\sqrt {f(\eta)}+\sqrt{f(\sigma^{z,z+1}\eta)})^2 (\eta_{z+1}-\eta_{z})^2 d{\nu^{N}_\alpha}.
\end{split}
\end{equation}
By neglecting the jumps of size bigger than one, we see that
$$D^{NN}(\sqrt f,\nu_{\alpha}^N)=\sum_{z\in\Lambda_N}\int \Big(\sqrt {f(\eta)}-\sqrt{f(\sigma^{z,z+1} \eta)}\Big)^2\; d{\nu_{\alpha}^N\; }\lesssim \; D_N^0(\sqrt f,\nu_{\alpha}^N).$$ 
Therefore, by using also (\ref{eq:thetaminus}), the first term at the right hand side of (\ref{est_repla_term_1}) can be bounded from above by 
\begin{eqnarray}
\label{eq:ouf}
\cfrac{A}{4}  \sum _{x =\ve N}^{N-1} \Theta_x^-\; D^{NN}(\sqrt f, \nu^{N}_\alpha)\; \lesssim \; AD^{NN} (\sqrt f, \nu^{N}_\alpha) \; \lesssim \; A D_N^0(\sqrt f,\nu^{N}_\alpha).
\end{eqnarray}
Recall (\ref{dir_est_const}) and observe that $D_N (\sqrt f,\nu_\alpha^N) \ge D_N^0 (\sqrt f,\nu_\alpha^N)$. Then we choose the constant $A$ in the form $A=C N/B$ for some suitable $C$ in order that one half of the term $-\tfrac{N}{4B} D_N (\sqrt f,\nu_\alpha)$ appearing in (\ref{dir_est_const}) counterbalances negatively the term at the right hand side of (\ref{eq:ouf}). Moreover we can bound from above the last term at the right hand side of \eqref{est_repla_term_1} by (use Lemma \ref{lem:densaf}) 
\begin{equation}
\begin{split}
&\frac{B}{N}\; \sum _{x=\ve N}^{N-1} \Theta_x^-\; \dfrac{1}{2\ve N}\sum _{y=1}^{\ve N}\sum_{z=y}^{x-1}\int (\sqrt {f(\eta)}+\sqrt{f(\sigma^{z,z+1}\eta)})^2
(\eta_{z+1}-\eta_{z})^2 d{\nu^{N}_\alpha}\; \lesssim \;  \frac{B}{N}\sum _{x \in \Lambda_{N}}x\Theta_x^-
\end{split}
\end{equation}
which vanishes as $N \to \infty$ by (\ref{dir_est_const}). Therefore we proved that uniformly in $\ve$
\begin{equation*}
\limsup_{B \to \infty} \limsup_{N\to \infty} \; \sup _{f}\Big\{\sum _{x=\ve N}^{N-1} \Theta_x^-\langle \tau_{x}^{\ve N},f\rangle_{\nu^{N}_\alpha}  + \dfrac{N}{2B}\left \langle L_{N}\sqrt{f},\sqrt{f}\right\rangle_{\nu^{N}_\alpha}\Big\} = 0.
\end{equation*}
It remains to prove that
\begin{equation*}
\limsup_{B\to \infty} \limsup_{\ve \to 0} \limsup_{N\to \infty} \; \sup _{f}\Big\{\sum _{x=1}^{\ve N -1} \Theta_x^-\langle \tau_{x}^{\ve N},f\rangle_{\nu^{N}_\alpha}  + \dfrac{N}{2B}\left \langle L_{N}\sqrt{f},\sqrt{f}\right\rangle_{\nu^{N}_\alpha}\Big\} = 0.
\end{equation*}

If $x\le \ve N -1$, we write
\begin{equation*}
\begin{split}
&\langle \tau_{x}^{\ve N},f\rangle_{\nu^{N}_\alpha} =\dfrac{1}{\ve N}\sum _{y=1}^{\ve N} \int (\eta_{x}-\eta_{y})f(\eta)\; d\nu^{N}_\alpha\\
&= \dfrac{1}{\ve N}\sum _{y=1}^{x-1}\sum_{z=y}^{x-1}\int (\eta_{z+1}-\eta_{z})f(\eta)\; d\nu^{N}_\alpha -  \dfrac{1}{\ve N}\sum _{y=x+1}^{\ve N}\sum_{z=x}^{y-1}\int (\eta_{z+1}-\eta_{z})f(\eta)\; d\nu^{N}_\alpha .
\end{split}
\end{equation*}
and the same estimates as before give there exists a constant $C>0$ such that for any $A>0$, 
\begin{equation*}
\begin{split}
\sum _{x=1}^{\ve N-1} \Theta_x^-\langle \tau_{x}^{\ve N},f\rangle_{\nu^{N}_\alpha} &\le C \left[ A D_N (\sqrt f ,\nu_\alpha^N) + \cfrac{\ve N}{A} \sum_{x=1}^{\ve N -1} \Theta_x^-  \right].
\end{split}
\end{equation*}
Recall (\ref{dir_est_const}) and (\ref{eq:thetaminus}). Then, we choose $A=N/ \, 8CB$ and we get that
\begin{equation*}
\limsup_{B \to \infty} \limsup_{\ve \to 0} \limsup_{N\to \infty} \; \sup _{f}\Big\{\sum _{x=1}^{\ve N -1} \Theta_x^-\langle \tau_{x}^{\ve N},f\rangle_{\nu^{N}_\alpha}  + \dfrac{N}{2B}\left \langle L_{N}\sqrt{f}\sqrt{f}\right\rangle_{\nu^{N}_\alpha}\Big\} = 0.
\end{equation*}
 This finishes the proof. 
\end{proof}

\begin{rem}\label{sec_rep_robin}
We note that above, if we change in the statement of the lemma $\Theta_x^-$ by $r_N^-$, then the same result holds by performing exactly the same estimates as above, because what we need is that
\begin{equation}
\sum _{x \in \Lambda_{N}}\Theta_x^- \lesssim 1 \quad \textrm {and}\quad \frac{1}{N} \sum _{x \in \Lambda_{N}}x\Theta_x^-\to0
\end{equation}
which also holds for $r_N^- $ instead of $\Theta_x^-$ since $\gamma>2.$
\end{rem}

\subsection{Fixing the profile at the boundary}
Let $\bb Q$ be a limit point of the sequence $\lbrace\bb Q_{N}\rbrace_{N \geq 1}$, whose existence follows from Proposition \ref{Tightness}  and assume, without lost of generality, that $\lbrace\bb Q_{N}\rbrace_{ N \ge 1}$ converges to $\bb Q$. We note that since our model is an exclusion process, it is standard (\cite{KL}) to show that $\mathbb Q$ almost surely the trajectories of measures are absolutely continuous with respect to the Lebesgue measure, that is: $\pi_t (dq)=\rho_{t}(q)dq$ for any $t \in [0,T]$. In Section \ref{sec:Energy} we  prove that  the density $\rho_t (q)$ belongs to $L^2 (0,T; {\mc H}^1)$ if $\theta \ge 2-\gamma$. In particular, for almost every $t$, $\rho_t$ can be identified with a continuous function on $[0,1]$.\\

{
In this section we prove 3. of Definition \ref{Def. Dirichlet source Condition-g}, that is,  for $\theta \in[2-\gamma,1)$ we show that the profile satisfies $\rho_t(0)=\alpha$ and $\rho_t(1)=\beta$ for $t \in [0,T]$ a.s.}\\

Recall \eqref{boxes}. Observe that 
$$\bb E_{\mu_N}\left[\Big|\int_0^t(\overrightarrow{\eta}_0^{\epsilon N}(sN^2)-\alpha)\, ds\Big|\right] = {\bb E}_{\bb Q_N}  \left[\Big|\int_0^t( \langle \pi_s, \iota^0_\ve \rangle -\alpha)\, ds\Big|\right]$$
where $\iota_\ve^0 (\cdot)  =\ve^{-1} \, {\bf 1}_{(0,\ve)} (\cdot)$. Therefore we have that for any $\delta>0$, 
$${\bb Q_N}  \left[\Big|\int_0^t( \langle \pi_s, \iota^0_\ve \rangle -\alpha)\, ds\Big| >\delta \right] \; \le \; \delta^{-1} \, \bb E_{\mu_N}\left[\Big|\int_0^t(\overrightarrow{\eta}_0^{\epsilon N}(sN^2)-\alpha)\, ds\Big|\right].$$ 
Py Portemanteau's Theorem {\footnote{In fact, since $\iota_{\ve}^0$ is not a continuous function it is not given for free that the set $\Big\{ \pi \, ;\, \Big|\int_0^t( \langle \pi_s, \iota^0_\ve \rangle -\alpha)\, ds\Big| >\delta \Big\}$ is an open set in the Skorohod topology. A simple argument based on a $L^1$-approximation of $\iota^{0}_{\ve}$ by continuous functions permits to bypass this difficulty.}} we conclude that
$${\bb Q}  \left[\Big|\int_0^t( \langle \pi_s, \iota^0_\ve \rangle -\alpha)\, ds\Big| >\delta \right] \; \le \; \delta^{-1} \, \liminf_{N \to \infty} \, \bb E_{\mu_N}\left[\Big|\int_0^t(\overrightarrow{\eta}_0^{\epsilon N}(sN^2)-\alpha)\, ds\Big|\right].$$
Now, if we are able to prove that the right hand side of the previous inequality is zero,  since we have that $\mathbb Q$ a.s. $\pi_s (dq) = \rho_s (q) dq$ with $\rho_s$ a continuous function in $0$ for a.e. $s$, by taking the limit $\ve \to 0$, we can deduce that $\mathbb Q$ a.s. $\rho_s (0) =\alpha$ for a.e. $s \in [0,T]$. A similar argument applies for the right boundary. Therefore it is sufficient to prove the following lemma.

\begin{lem} \label{fix_prof}
Let $\theta<1$. For any $t\in[0,T]$ we have that 
\begin{equation*}
\begin{split}
&{\limsup_{\ve \to 0}}\,{\limsup _{N\rightarrow \infty}}\,\bb E_{\mu_N}\left[\Big|\int_0^t(\overrightarrow{\eta}_0^{\epsilon N}(sN^2)-\alpha)\, ds\Big|\right] =0,\\
&{\limsup_{\ve \to 0}}\,{\limsup _{N\rightarrow \infty}}\,\bb E_{\mu_N}\left[\Big|\int_0^t(\overleftarrow{\eta}_{N}^{\epsilon N}(sN^2)-\beta)\, ds\Big|\right] =0.
\end{split}
\end{equation*}
\end{lem}
Last lemma is a consequence of the next  two results. 
\begin{lem} \label{Rep-Dirichlet2}
Let $\theta<1$. For any $t\in[0,T]$ we have that 
\begin{equation*}
\begin{split}
&{\limsup _{N\rightarrow \infty}}\,\bb E_{\mu_N}\left[\Big|\int_0^t(\eta_1(sN^2)-\alpha)\, ds\Big|\right] =0,\\
& {\limsup _{N\rightarrow \infty}}\,\bb E_{\mu_N}\left[\Big|\int_0^t(\eta_{N-1}(sN^2)-\beta)\, ds\Big|\right] =0.
\end{split}
\end{equation*}
\end{lem}

\begin{proof} 
We give the proof for the first display, but  we note that for the other one it is similar. Fix a Lipschitz profile $h(\cdot)$ such that $\alpha \le h(\cdot) \le \beta$ and $h(0)=\alpha$, $h(1)=\beta$ and $h(\cdot)$ is $\gamma/2$-H\"older at the boundaries. By the entropy  and Jensen's inequalities, for any $B>0$, the previous expectation is bounded from above by
\begin{equation*} \label{use1}
\begin{split}
\dfrac{H(\mu _{N}\vert {\nu_{h(\cdot)}^{N}}) }{BN}
+ \dfrac{1}{BN}\log \mathbb{E}_{{\nu_{h(\cdot)}^{N}}}\left[ e^{BN\vert \int _{0}^{t}(\eta_1(sN^2)-\alpha)\, ds\vert}\right].
\end{split}
\end{equation*}
By \eqref{H}, Feynman-Kac's formula and noting, as we did in the proof of Lemma \ref{Rep-Dirichlet1}, that we can remove the absolute value inside the exponential, last display can be estimated from above by 
\begin{equation} 
\label{F-K 333}
\frac {C_0} {B}+ t \sup _{f}\left\{\left\langle t_{1}^{\alpha}, f\right\rangle_{{\nu_{h(\cdot)}^{N}}}  + \dfrac{N}{B} \left\langle L_{N}\sqrt{f},\sqrt{f}\right\rangle_{{\nu_{h(\cdot)}^{N}}}\right\},
\end{equation}
where the supremum is carried over all the densities $f$ with respect to $\nu_{h(\cdot)}^N$. Here we recall that $t_{1}^{\alpha}(\eta) = \eta_{1} -\alpha$. By Lemma \ref{bound}, since $h$ is Lipschitz, for any $A>0$, the first term in the supremum in  (\ref{F-K 333}) is bounded from above by 
\begin{equation*}
C \left[ \dfrac{1}{A}  \; I_{1}^{\alpha}(\sqrt{f},\nu_{h(\cdot)}^{N})+ {A}+\frac{1}{N} \right]
\end{equation*}
for some constant $C>0$ independent of $f$ and $A$. Moreover from \eqref{dir_est_holder}, since 
$$D_N (\sqrt{f}, \nu_{h(\cdot)}^N) \ge D_N^{\ell} (\sqrt{f}, \nu_{h(\cdot)}^N)$$
and $\gamma+\theta -2 \ge 0$, we know that there exists a constant $C' >0$ such that  
\begin{eqnarray*}
\dfrac{N}{B} \langle L_{N}\sqrt{f},\sqrt{f}\rangle_{\nu^{N}_{h(\cdot)}} &\leq& -\dfrac{N^{1-\theta}}{4B} \sum _{x \in \Lambda_{N}}I_{x}^{\alpha}(\sqrt{f},\nu_{h(\cdot)}^{N})r_{N}^{-}(\tfrac{x}{N}) + \dfrac{C'}{B}.
\end{eqnarray*}
To get an upper bound, at the right hand side of the previous inequality, we only keep the term coming from $x=1$ in the sum. By  
choosing $A =4C  (r^-_N(\tfrac{1}{N}))^{-1}B N^{\theta -1}$, we get then that the expression inside the brackets in (\ref{F-K 333}) is bounded by 
\begin{eqnarray*}
4 C^2  \dfrac{BN^{\theta-1}}{r_N^-(\tfrac{1}{N})}+\cfrac{C}{N}+\dfrac{C'}{B}.
\end{eqnarray*}
Now since $r_N^-(\tfrac{1}{N})$ is bounded from below by a constant independent of $N$ and $\theta<1$, the proof follows by sending first $N\to \infty$  and then $B\to \infty$.
\end{proof}

\begin{lem}
\label{repla_novo}
Let $\theta\in\mathbb R$. For any $t>0$ we have that 
\begin{equation}
\begin{split}
&{\limsup_{\ve \to 0}}\,{\limsup _{N\rightarrow \infty}}\, \bb E_{\mu_N}\Big[\Big|\int_0^t\overrightarrow{\eta}_0^{\epsilon N}(sN^2)-\eta_{1}(sN^2))\, ds\Big|\Big] =0,\\
&{\limsup_{\ve \to 0}}\,{\limsup _{N\rightarrow \infty}}\, \bb E_{\mu_N}\Big[\Big|\int_0^t\overleftarrow{\eta}_{N}^{\epsilon N}(sN^2)-\eta_{N-1}(sN^2))\, ds\Big|\Big] =0.
\end{split}
\end{equation}
\end{lem}
\begin{proof}
We present the proof of the first item, but we note that for the second it is exactly the same. Fix a Lipcshitz profile $h(\cdot)$ such that $\alpha \le h(\cdot) \le \beta$, $h(0)=\alpha$, $h(1)=\beta$ and $h(\cdot)$ is $\gamma/2$-H\"older at the boundaries. By the entropy and Jensen's inequalities, for any $B>0$,  the previous expectation is bounded from above by
\begin{equation*}
\begin{split}
\dfrac{H(\mu _{N}\vert \nu_{h(\cdot)}^{N}) }{BN}
+ \dfrac{1}{BN}\log \mathbb{E}_{\nu_{h(\cdot)}^{N}}\left[ e^{BN\vert\int_0^t\overrightarrow{\eta}_0^{\epsilon N}(sN^2)-\eta_{1}(sN^2)\, ds\vert}\right].\\
\end{split}
\end{equation*}
By \eqref{H}, Feynman-Kac's formula, and using the same argument as in the proof of the previous  lemma, the estimate of the previous expression can be reduced to bound
\begin{eqnarray} 
\label{E1} 
\nonumber
&&\frac {C_0} {B}+ t \sup _{f}\Big\{\frac{1}{\ell}\sum _{y=2}^{\ell+1} \vert \langle v_{y}^{1},f\rangle_{\nu_{h(\cdot)}^{N}}\vert + \dfrac{N}{B} \left\langle L_{N}\sqrt{f},\sqrt{f}\right\rangle_{\nu_{h(\cdot)}^{N}}\Big\},\\ 
\end{eqnarray}
where $\ell=\epsilon N $ and $v_{y}^{1} (\eta) = \eta_y-\eta_{1}$. Here the supremum is carried over all the densities $f$ with respect to $\nu_{h(\cdot)}^N$. 
Note that since $y \in \Lambda_{N}$ we know that $v_{y}^{1}(\eta) = \sum _{z=1}^{y-1}(\eta_{z+1}-\eta_{z}).$
Observe now that 
\begin{eqnarray*} 
\label{eq:termo} \nonumber
\sum _{z=1}^{y-1}\int(\eta_{z+1}-\eta_{z}) f(\eta)d\nu_{h(\cdot)}^N &=& \dfrac{1}{2}\sum _{z=1}^{y-1}\int(\eta_{z+1}-\eta_{z})(f(\eta)-f(\sigma^{z, z+1} \eta))d\nu_{h(\cdot)}^{N}\\\nonumber
&&+\dfrac{1}{2}\sum _{z=1}^{y-1}\int(\eta_{z+1}-\eta_{z})(f(\eta)+f(\sigma^{z, z+1}\eta))d\nu_{h(\cdot)}^{N}.\\
\end{eqnarray*}
By using the fact that for any $a,b \ge 0$, $(a-b)=(\sqrt a -\sqrt b)(\sqrt a +\sqrt b)$ and Young's inequality, we have, for any positive constant $A$, that 
\begin{equation}\label{eq:term1}
\begin{split}  
\frac{1}{\ell}\sum _{y=2}^{\ell+1} \vert \langle v_{y}^{1},f\rangle_{\nu_{h(\cdot)}^{N}}\vert  &\;  \le \;   \dfrac{1}{2A\ell} \sum _{y=2}^{\ell+1} \sum _{z=1}^{y-1}\int (\eta_{z+1}-\eta_{z})^2\; \Big(\sqrt{f(\eta)}+\sqrt{f(\sigma^{z, z+1} \eta)}\Big)^2d\nu_{h(\cdot)}^{N}\\
&+\dfrac{A}{2\ell}\sum _{y=2}^{\ell+1}  \sum _{z=1}^{y-1}\int \Big(\sqrt{f(\eta)}-\sqrt{f(\sigma^{z, z+1} \eta)}\Big)^2 d\nu_{h(\cdot)}^{N}\\
&+\dfrac{1}{2\ell} \sum _{y=2}^{\ell+1} \left| \sum _{z=1}^{y-1}\int \big( \eta_{z+1}-\eta_{z} \big) \; \Big(f(\eta)+f(\sigma^{z, z+1}\eta)\Big) d\nu_{h(\cdot)}^{N} \right|.
\end{split}
\end{equation}
By neglecting the jumps of size bigger than one, we see that 
$$D^{NN}(\sqrt f,\nu_{h(\cdot)}^N)=\sum_{z\in\Lambda_N}\int \Big(\sqrt {f(\eta)}-\sqrt{f(\sigma^{z,z+1} \eta)}\Big)^2\; d{\nu_{h(\cdot)}^N}\; \lesssim \; D_N^0(\sqrt f,\nu_{h(\cdot)}^N).$$
Then, the second term on the right hand side of (\ref{eq:term1}) is bounded from above by 
\begin{eqnarray*}\nonumber
 \frac{A}{2\ell} \sum_{y=2}^{\ell+1}  D^{NN}(\sqrt f, \nu^N_{h(\cdot)})\leq A\; D^{NN}(\sqrt f, \nu_{h(\cdot)}^N) \le C A \; D_N^0(\sqrt f,\nu_{h(\cdot)}^N)\le  C A \; D_N (\sqrt f,\nu_{h(\cdot)}^N)
\end{eqnarray*}
where $C$ is a positive constant independent of $A,\ell,f$. Then, for the choice $A=N(4BC)^{-1}$ and from \eqref{dir_est_holder}, since $\gamma +\theta -2 \ge 0$, we can bound from above \eqref{E1} by 
\begin{equation}
\label{eq:E1234567890}
\begin{split}
&\dfrac{2BC}{N \ell} \sum _{y=2}^{\ell+1} \sum _{z=1}^{y-1}\int (\eta_{z+1}-\eta_{z})^2\; \Big(\sqrt{f(\eta)}+\sqrt{f(\sigma^{z, z+1} \eta)}\Big)^2d\nu_{h(\cdot)}^{N}\\
&+\dfrac{1}{2\ell} \sum _{y=2}^{\ell+1} \left| \sum _{z=1}^{y-1}\int \big( \eta_{z+1}-\eta_{z} \big) \; \Big(f(\eta)+f(\sigma^{z, z+1}\eta)\Big) d\nu_{h(\cdot)}^{N} \right| + \frac{C' }{B}\\
&\lesssim  \frac{B\ell}{N} + \frac{1}{B} +\dfrac{1}{2\ell} \sum _{y=2}^{\ell+1} \left| \sum _{z=1}^{y-1}\int \big( \eta_{z+1}-\eta_{z} \big) \; \Big(f(\eta)+f(\sigma^{z, z+1}\eta)\Big) d\nu_{h(\cdot)}^{N} \right|
\end{split}
\end{equation}
for some constant $C'>0$. For the last inequality we used Lemma \ref{lem:densaf}. Observe that $B\ell/N = B\varepsilon$ vanishes as $\epsilon \to 0$. It remains to estimate the third term on the right hand side of the last inequality. For that purpose we make a similar computation to the one of Lemma \ref{bound}. Let $C_z=\max\left\{\frac{1}{h\left(\tfrac{z}{N}\right)\left(1-h\left(\tfrac{z+1}{N}\right)\right)},\frac{1}{h\left(\tfrac{z+1}{N}\right)\left(1-h\left(\tfrac{z}{N}\right)\right)}\right\}$ which is bounded above by a constant depending only on $\alpha$ and $\beta$. By using the fact that $\nu^N_{h(\cdot)}$ is product and denoting by $\tilde \eta$ the configuration $\eta$ removing its value at $z$ and $z+1$ so that $(\eta_z,\eta_{z+1},\tilde\eta)=\eta$, we have that
\begin{equation*}
\begin{split}
&\sum_{z=1}^{y-1}
\left\vert \int(\eta_{z+1}-\eta_{z})(f(\eta)+f(\sigma^{z, z+1}\eta)) d\nu_{h(\cdot)}^{N}\right\vert \\
=&\sum_{z=1}^{y-1} \left\vert \sum_{\tilde{\eta}}(f(0,1,\tilde\eta)+f(1,0, \tilde\eta))h(\tfrac{z+1}{N})(1-h(\tfrac{z}{N}))\; \nu_{h(\cdot)}^{N}(\tilde\eta) \right.
\\&-\left. \sum_{\tilde{\eta}}(f(1,0,\tilde\eta)+f(0,1,\tilde\eta))h(\tfrac{z}{N})(1-h(\tfrac{z+1}{N}))\;\nu_{h(\cdot)}^{N}(\bar\eta)\right\vert\\
=&\sum_{z=1}^{y-1}\Big\vert\sum_{\tilde{\eta}}\Big(h\Big(\frac{z+1}{N}\Big)-h\Big(\frac{z}{N}\Big)\Big)(f(0,1,\tilde\eta)+f(1,0,\tilde\eta))\; \nu_{h(\cdot)}^{N}(\tilde\eta)\Big\vert\\
\leq & \frac{1}{2}\, \sum_{z=1}^{y-1}C_z \Big|h\Big(\tfrac{z+1}{N}\Big)-h\Big(\tfrac{z}{N}\Big)\Big)\Big|\; \sum_{\tilde\eta}\Big\{f(1,0,\tilde\eta)\; h\Big(\tfrac{z}{N}\Big)\Big(1-h\Big(\tfrac{z+1}{N}\Big)\Big)\; \nu_{h(\cdot)}^{N}(\tilde\eta)\\&\quad \quad \quad \quad \quad \quad \quad + f(0,1,\bar\eta)\; \Big(1-h\Big(\tfrac{z}{N}\Big)\Big)h\Big(\tfrac{z+1}{N}\Big)\; \nu_{h(\cdot)}^{N}(\tilde\eta)\Big\}\\
 &\lesssim \; \sum_{z=1}^{y-1}\Big|h\Big(\frac{z+1}{N}\Big)-h\Big(\frac{z}{N}\Big)\Big)\Big|.
\end{split}
\end{equation*}
Above, for example, $f(1,0,\tilde\eta)$ (resp. $f(0,1,\tilde\eta)$)  means that we are computing $f(\eta)$ with $\eta$ such that $\eta_z=1$ and $\eta_{z+1}=0$ (resp. $\eta_z=0$ and $\eta_{z+1}=1$).
Since $h(\cdot)$ is Lipschitz, by (\ref{eq:E1234567890}), this estimate provides an upper bound for  \eqref{E1} which is in the form of a constant times 
$$  \frac{B\ell}{N} + \frac{1}{B} +\frac{1}{N\ell}\sum _{y=2}^{\ell+1} y \; \lesssim \; B\ve + B^{-1} + \ve$$
 which vanishes, as $\varepsilon \to 0$ and then $B \to \infty$. This ends the proof.
 \end{proof}

\section{Energy Estimates}\label{sec:Energy}

Let $\bb Q$ be a limit point of the sequence $\lbrace\bb Q_{N}\rbrace_{N \geq 1}$, whose existence follows from Proposition \ref{Tightness}  and assume, without lost of generality, that $\lbrace\bb Q_{N}\rbrace_{ N \ge 1}$ converges to $\bb Q$. We note that since our model is an exclusion process, it is standard (\cite{KL}) to show that $\mathbb Q$ almost surely the trajectories of measures are absolutely continuous with respect to the Lebesgue measure, that is: $\pi_t(dq)=\rho_{t}(q)dq$ for any $t \in [0,T]$.

\subsection{The case $\theta\geq 2-\gamma$}
Recall that in this case the system is speeded up in the diffusive time scale so that $\Theta(N)=N^2$. In this section we prove that the density $\rho_{t}(q)$ belongs to $L^{2}(0,T;\mathcal{H}^{1}(0,1))$, see Definition \ref{Def. Sobolev space}. For that purpose, we define the linear functional $\ell_{\rho}$ on $C^{0,1}_{c}([0,T]\times (0,1))$ by 
$$ \ell_{\rho}(G) = \int^{T}_{0}\int^{1}_{0}\partial_{q}G_{s}(q)\rho_{s}(q) \, dq ds = \int^{T}_{0}\int^{1}_{0}\partial_{q}G_{s}(q)\, d \pi_{s}(q) ds.$$
By Proposition \ref{EE} below we have that $\ell_{\rho}$ is $\bb Q$ almost surely continuous, thus we can extend this linear functional to $L^{2}([0,T]\times (0,1))$. Moreover, by the Riesz's Representation Theorem  we find $\zeta \in L^{2}([0,T]\times (0,1))$ such that
$$\ell _{\rho}(G) = -\int^{T}_{0}\int^{1}_{0} G_{s}(q)\zeta_{s}(q)dq ds , $$ for all $G \in C^{0,1}_{c}([0,T]\times (0,1))$, which implies that $\rho \in L^{2}(0,T;\mathcal{H}^{1}(0,1))$.

\begin{prop} 
\label{EE}
For all $\theta \geq 2-\gamma$. There exist positive constants $C$ and $c$ such that 
$$\mathbb{E}\left[ \sup _{G}\lbrace \ell_{\rho}(G) - c \Vert G \Vert_{2}^{2}\rbrace \right]\leq C < \infty,$$
where the supremum above is taken on the set $C^{0,1}_{c}([0,T]\times (0,1))$. Here we denote by $\Vert G \Vert_{2}$ the norm of a function $G \in L^{2}([0,T]\times (0,1)).$
\end{prop}

\begin{proof}
By density it is enough to prove Proposition \ref{EE} for a countable dense subset $\lbrace  G ^{m}\rbrace_{m \in \mathbb{N}}$ on $C_{c}^{0,2}([0,T]\times (0,1))$ and by Monotone Convergence Theorem it is enough to prove that
$$\mathbb{E}\left[ \sup _{k \leq m}\lbrace \ell_{\rho}(G^{k}) - c  \Vert G^{k} \Vert_{2}^{2}\rbrace \right]\leq K_{0},$$
for any $m$ and for $K_{0}$ independent of $m$. 
Now, we define $\Phi: \mc D([0,T], \mc M^+)\rightarrow \RR$ by
$$\Phi(\pi_{\cdot}) = \max_{k\leq m}\left\{  \int^{T}_{0}\int^{1}_{0}\partial_{q}G^k_{s}(q)\, d \pi_{s}(q) ds -c\Vert G^{k} \Vert_{2}^{2}\right\},$$ which is a  continuous and bounded  function for the Skorohod topology of $ \mc D([0,T], \mc M^+)$. Thus we have that 
\begin{eqnarray*}
\mathbb{E}[\Phi] = \lim _{N\rightarrow \infty} \EE _{\mu _{N}}\left[ \max_{k\leq m}\left\lbrace \int_{0}^{T}\dfrac{1}{N-1} 
\sum_{x=1}^{N-1} \partial _{q}G_{s}^{k}(\tfrac{x}{N})\eta_{x}(s)ds -c\Vert G^{k} \Vert_{2}^{2} \right\rbrace \right].
\end{eqnarray*}
By the  entropy inequality,  Jensen's inequality and the fact that $e^{\max_{k\leq m} a_{k}}\leq \sum_{k=1}^{m}e^{a_{k}}$ the previous display is bounded from above by 
\begin{eqnarray*}\nonumber
&&C_{0}
+ \dfrac{1}{N}\log \mathbb{E}_{\nu_{h(\cdot)}^{N}}\left[ \sum _{k =1}^{m}e^{\int _{0}^{T}\sum _{x \in \Lambda_{N}}\partial_{q}G_{s}^{k}(\tfrac{x}{N})\eta_x(s) ds - cN\Vert G^{k} \Vert_{2}^{2} } \right],\\
\end{eqnarray*}
where $\nu^N_{h(\cdot)}$ is the Bernoulli product measure corresponding to a profile $h(\cdot)$ which is Lipschitz such that $\alpha \le h(\cdot)\le \beta$, $h(0)=\alpha$, $h(1)=\beta$ and $h$ is  $\gamma/2$-H\"older at the boundaries.  In order to deal with the second term in the previous display we use (\ref{Log bounded}) and it is enough  to bound 
$$\limsup_{N \to \infty} \;  \dfrac{1}{N}\log \mathbb{E}_{{\nu_{h(\cdot)}^{N}}}\left[ e^{\int _{0}^{T}\sum _{x \in \Lambda_{N}}\partial_{q}G_{s}(\tfrac{x}{N})\eta_x(s) ds - cN\Vert G \Vert_{2}^{2} } \right],$$
for a fixed function $G\in C_{c}^{0,2}([0,T]\times (0,1))$, by a constant independent of $G$.
By Feynman-Kac's formula, the expression inside the limsup is bounded from above by 
\begin{equation*} 
\label{EE3}
\int _{0}^{T}\sup _{f}\Big\{\dfrac{1}{N}\int_{\Omega_{N}}\sum _{x \in \Lambda_{N}}\partial_{q}G_{s}(\tfrac{x}{N})\eta_x f(\eta)d{\nu_{h(\cdot)}^{N}} (\eta) - c\Vert G \Vert_{2}^{2}  + \frac{\Theta(N)}{N} \langle L_{N}\sqrt{f},\sqrt{f}\rangle_{{\nu_{h(\cdot)}^{N}}}\Big\}\, ds
\end{equation*}
where the supremum is carried over all the densities $f$ with respect to $\nu_{h(\cdot)}^N$. Let us now focus on the first term inside braces in the previous expression.  
Observe first that the space derivative of $G_s$ can be replaced by the discrete gradient $\nabla_{N} G_{s}(\tfrac{x-1}{N})= N \big[ G_s (\tfrac{x}{N}) -G_s (\tfrac{x-1}{N}) \big]$ of $G_s$ with an error $R_N(G)$ satisfying uniformly in $N$ the bound $| R_N (G) |  \lesssim 1/N$ since $G \in C_{c}^{0,2}([0,T], (0,1))$. By summing and subtracting the term $\nabla_{N} G_{s}(\tfrac{x-1}{N})$ inside the sum, and doing a summation by parts, we can write
\begin{equation*}
\dfrac{1}{N}\int_{\Omega_{N}}\sum _{x \in \Lambda_{N}}\partial_{q}G_{s}(\tfrac{x}{N})\eta_x f(\eta)d{\nu_{h(\cdot)}^{N}} (\eta) = \int_{\Omega_{N}} \sum_{x=1}^{N-2} G_{s}(\tfrac{x}{N})(\eta_{x}-\eta_{x+1})f(\eta)d\nu_{h(\cdot)}^{N}(\eta)+ R_N(G).
\end{equation*}
A simple computation shows that we can write the first term at the right hand side of the previous display as 
\begin{equation}
\begin{split}
\label{EE5}
&\dfrac{1}{2}\int_{\Omega_{N}} \sum_{x=1}^{N-2} G_{s}(\tfrac{x}{N})(\eta_{x}-\eta_{x+1})(f(\eta)-f(\sigma^{x,x+1}\eta))d{\nu_{h(\cdot)}^{N}} \\
+&  \dfrac{1}{2}\int_{\Omega_{N}} \sum_{x=1}^{N-2} G_{s}(\tfrac{x}{N})(\eta_{x}-\eta_{x+1})f(\sigma^{x,x+1}\eta)(1-\theta^{x,x+1}(\eta))d{\nu_{h(\cdot)}^{N}}.
\end{split}
\end{equation}
Recall that for $u,v \ge 0$, $u- v = (\sqrt u -\sqrt v) (\sqrt u + \sqrt v)$ and the inequality $ab \leq \dfrac{B a^{2}}{2} +\dfrac{b^{2}}{2B}$ valid for any $B>0$. Taking $B=\dfrac{N}{\Theta(N)}$ and using  Lemma \ref{lem:densaf} we bound the first term in (\ref{EE5})  by 
\begin{eqnarray}\label{EE6}\nonumber
&&\dfrac{N}{4\Theta(N)}\int_{\Omega_{N}} \sum_{x=1}^{N-2} (G_{s}(\tfrac{x}{N}))^{2}(\sqrt{f(\eta)}+ \sqrt{ f(\sigma^{x,x+1}\eta)})^{2}d\nu_{h(\cdot)}^{N}(\eta) \\\nonumber
&+ & \dfrac{\Theta(N)}{4N}\int_{\Omega_{N}} \sum_{x=1}^{N-2} (\sqrt{f(\eta)}- \sqrt{ f(\sigma^{x,x+1}\eta)})^{2}d\nu_{h(\cdot)}^{N}(\eta) \\\nonumber
&\leq & \dfrac{\Theta(N)}{4N} D^{0}_{N}(\sqrt{f},\nu_{h(\cdot)}^{N}) + \dfrac{C N}{\Theta(N)}\sum_{x\in \Lambda_{N}} (G_{s}(\tfrac{x}{N}))^{2}
\end{eqnarray}
for some $C>0$. 
Similarly we can estimate the second term in (\ref{EE5}) from above by 
\begin{eqnarray*}
& &\dfrac{1}{4N}\int_{\Omega_{N}} \sum_{x=1}^{N-2} (G_{s}(\tfrac{x}{N}))^{2} (\eta_x -\eta_{x+1})^2  f(\sigma^{x,x+1}\eta)d\nu_{h(\cdot)}^{N}(\eta) \\\nonumber
& +& \dfrac{N}{4}\int_{\Omega_{N}} \sum_{x=1}^{N-2}  f(\sigma^{x,x+1}\eta)(\theta^{x,x+1}(\eta)-1)^{2}d\nu_{h(\cdot)}^{N}(\eta) \\\nonumber
&\lesssim & \dfrac{1}{N} \sum_{x\in \Lambda_{N}} (G_{s}(\tfrac{x}{N}))^{2} + 1.
\end{eqnarray*}
We use now (\ref{dir_est_holder}) with $B=1$ there and observe that last two terms at the right hand side of  \eqref{dir_est_holder} are  bounded from above by  a constant since $\gamma+ \theta -2 \ge 0$.  Observe also that $D_N^0 (\sqrt{f}, \nu_{h(\cdot)}^N) \le D_N (\sqrt{f}, \nu_{h(\cdot)}^N)$. Recalling that $\Theta (N) =N^2$ we get then that (\ref{EE3}) is bounded from above by 
\begin{equation*}
C \int_0^T \Big[ 1 +  \dfrac{1}{N} \sum_{x\in \Lambda_{N}} (G_{s}(\tfrac{x}{N}))^{2}\Big] \, ds  \; -\;  c \| G|_2^2 \; +\; R_N (G)
\end{equation*}
where $C$ is a positive constant independent of $G$. We then choose $c>C$ in order to conclude that 
$$\limsup_{N\to \infty} \; \Big\{ C \int_0^T \Big[ 1 +  \dfrac{1}{N} \sum_{x\in \Lambda_{N}} (G_{s}(\tfrac{x}{N}))^{2}\Big] \, ds  \; -\;  c \| G|_2^2 \; +\; R_N (G) \Big\} \; \lesssim \; 1.$$
This achieves the proof.
\end{proof}

\subsection{The case $\theta\leq 2-\gamma$}
In this section we prove that the function  $(t,q) \to \rho_{t}(q) -\alpha$ belongs to $L^{2}([0,T]\times (0,1), dt \otimes d\mu)$, where $\mu$ is the measure that has the density with respect to the Lebesgue measure given by
$$u\in[0,1]\rightarrow\frac{1}{u^{\gamma}}.$$
A similar proof would show that the function  $(t,q) \to \rho_{t}(q) -\beta$ belongs to $L^{2}([0,T]\times (0,1), dt \otimes d\mu')$, where $\mu'$ is the measure that has the density with respect to the Lebesgue measure given by
$$u\in[0,1]\rightarrow\frac{1}{(1-u)^{\gamma}}.$$

Let $\nu^{N}_{h(\cdot)}$ be as above, 
where $h : [0, 1] \rightarrow [0, 1]$ is a profile such that $\alpha \leq h(q)\leq\beta$, for all $q\in [0,1]$,  $h(0) = \alpha$ and $h(1)=\beta$, H\"older of parameter $\gamma/2$ at the boundaries and Lipschitz inside.
 Let $G \in C_c^{1,\infty} ([0,T]\times [0,1])$.  By the  entropy inequality and the Feynmann-Kac's formula, we have that
\begin{equation}\label{eq:varfor}
\begin{split}
&{\mathbb E}_{\mu_N} \left( \int_0^T dt \;  N^{\gamma -1} \sum_{x\in \Lambda_N}  G(t, \tfrac{x}{N}) r_N^-\Big(\tfrac{x}{N}\Big) (\eta_x(tN^{\theta+\gamma})-\alpha)\right)\\
&\le C_0 +\int_0^T \sup_{f} \left\{ N^{\gamma-1}  \sum_{x \in \Lambda_N}  G(t, \tfrac{x}{N}) r_N^-\Big(\tfrac{x}{N}\Big) \langle t_x^\alpha,f\rangle_{\nu^{N}_{h(\cdot)}}+  \dfrac{\Theta(N)}{N} \left\langle  L_N   {\sqrt f} , {\sqrt f} \right\rangle_{\nu^{N}_{h(\cdot)}}  \right\}dt
\end{split}
\end{equation}
where the supremum is taken over all the densities $f$ on $\Omega_N$ with respect to $\nu_{h(\cdot)}^N$. Below $C$ is a constant that may change from line to line. Since the profile is H\"older of parameter $\gamma/2$ at the boundaries and Lipschitz inside, and from \eqref{dir_est_holder} the term at the right hand side of last expression is bounded from above by $$-\frac{\Theta(N)}{4N}D_N(\sqrt f,\nu_{h(\cdot)}^N)+\dfrac{\Theta(N)}{N^{2}}C+\dfrac{\Theta(N)}{N^{\gamma+\theta}}C.$$ Repeating the proof of Lemma \ref{fix_prof} last expression is bounded from above by 
\begin{eqnarray*}
CN^{\gamma-1}\sum_{x\in\Lambda_N}r_N^-\Big(\tfrac{x}{N}\Big)G^2\Big(t,\tfrac{x}{N}\Big)+C+ \dfrac{\Theta(N)}{N^{2}}C +\dfrac{\Theta(N)}{N^{\gamma+\theta}}C.
\end{eqnarray*}
We take the limit $N \to \infty$. We conclude that there exist constants $C',C''>0$ independent of $G$ such that
\begin{equation}
{\bb E} \left[  \int_0^T \int_0^1\cfrac{(\rho_t(u) -\alpha) G (t,u) }{|u|^{\gamma}}\; du dt \; -\;   C \int_0^T\int_0^1\cfrac{  G^2(t,u) }{|u|^{\gamma}} \; du dt  \right] \le C'.
\end{equation}
By using a similar method as in the proof of the previous lemma we see that the supremum over $G$ can be inserted in the expectation so that
\begin{equation}
{\bb E} \left[ \sup_G \left\{  \int_0^T \int_0^1\cfrac{(\rho_t(u) -\alpha) G (t,u) }{|u|^{\gamma}}\; du dt \; -\;   C \int_0^T\int_0^1\cfrac{G^2(t,u) }{|u|^{\gamma}} \; du dt\right\}  \right] \le C'.
\end{equation}
The previous formula implies that 
$$\EE \left[ \int_0^T\int_0^1 \frac{ (\rho_t(u)-\alpha)^2}{|u|^\gamma}\, du dt \right] \le C''.$$
which  proves the claim.

\section{Characterization of limit points} \label{sec:Characterization of limit points}
We prove in this section that for each range of $\theta$, all limit points $\bb Q$ of the sequence $\{\mathbb Q_N\}_{N\in\mathbb{N}}$ are concentrated on trajectories of measures absolutely continuous with respect to the Lebesgue measure whose density $\rho_{t}(q)$ is a weak solution of the corresponding hydrodynamic equation. Let $\bb Q$ be a limit point of the sequence $\lbrace\bb Q_{N}\rbrace_{N \geq 1}$, whose existence follows from Proposition \ref{Tightness}  and assume, without lost of generality, that $\lbrace\bb Q_{N}\rbrace_{ N \ge 1}$ converges to $\bb Q$. As mentioned above, since there is at most one particle per site, it is easy to show that $\bb Q$ is concentrated on trajectories $\pi_{t}(dq)$ which are absolutely continuous with respect to the Lebesgue measure, that is, $\pi_{t}(dq)=\rho_{t}(q)dq$ (for more details see \cite{KL}).  Below, we prove, for each range of $\theta$, that the density $\rho_{t}(q)$ is a weak solution of the corresponding hydrodynamic equation. 

\begin{prop}
\label{prop:weak_sol_car}
If $\bb Q$ is a limit point of $ \{\bb Q_{N}\}_{N\in\mathbb N}$  then 
\begin{enumerate}[1.]
\item if $\theta < 1$:
\begin{eqnarray*} 
\bb Q\left( \pi _{\cdot}\in \mc D([0,T], \mathcal{M^{+}}): F_{RD}(t, \rho,G,g)= 0,\forall t\in [0,T],\, \forall G \in C_c^{1,2} ([0,T]\times[0,1])\,\right)=1.
\end{eqnarray*}
\item if $\theta\in{[1,+\infty)}$:
\begin{eqnarray*} 
\bb Q\left( \pi _{\cdot}\in \mc D([0,T], \mathcal{M^{+}}): F_{Rob}(t, \rho,G,g)= 0,\forall t\in [0,T],\, \forall G \in C^{1,2} ([0,T]\times[0,1])\,\right)=1.
\end{eqnarray*}
\end{enumerate}
\end{prop}

\begin{rem} In this proposition, the constants $\hat \kappa, \hat \sigma, \hat m$ appearing in $F_{RD}$ and $F_{Rob}$ are fixed by Theorem \ref{th:hl900}.
\end{rem}

\begin{proof}
Note that in order to prove the proposition, it is enough to verify, for $\delta > 0$ and $ G$ in the corresponding space of test functions,  that  
\begin{eqnarray} \nonumber
&&\bb Q\left(\pi _{\cdot}\in \mc D([0,T], \mathcal{M^{+}}): \sup_{0\le t \le T} \left\vert F_{\bullet}(t,\rho,G,g) \right\vert>\delta\right)=0,
\end{eqnarray}
for each $\theta$, where $F_\bullet$ stands for $F_{RD}$ if $\theta<1$ and $F_{Rob}$ if $\theta\ge 1$.  From here on, in order to simplify notation, we will erase $\pi_\cdot$ from the sets that we have to look at. 

\vspace{0.5cm}
\noindent $\bullet$ We start with the case $\theta \in [1,\infty)$.
Recall $F_{Rob}(t,\rho,G,g)$ from Definition \ref{eq:Robin integral-g}. Observe  that, due to the boundary terms that involve $\rho_s(1)$ and $\rho_s(0)$, the set inside last probability  is not an open set in the Skorohod space, therefore we cannot use directly Portmanteau's Theorem as we would like to. In order to avoid this problem, we fix $\ve>0$ and we consider two approximations of the identity given by $\iota^0_\ve(q)=\frac{1}{\ve}\textrm{1}_{(0,\ve)}(q)$ and $\iota^1_\ve(q)=\frac{1}{\ve}\textrm{1}_{(1-\ve,1)}(q)$ and we sum and subtract  to $\rho_{s}(0)$ (resp.  $\rho_{s}(1)$) the mean  $< \pi_s, \iota^0_\ve>= \tfrac{1}{\ve}\int_{0}^{\ve}\rho_{s}(q)dq$ (resp. $<\pi_s, \iota^1_\ve>=\tfrac{1}{\ve}\int_{1-\ve}^{\ve}\rho_{s}(q)dq$). 
Thus,  we bound last probability  from above by the sum of the following four terms
\begin{equation}\label{RC6}
\begin{split}
&\bb Q\left(  \sup_{0\le t \le T} \Big|\int_0^1 \rho_{t}(q)  G_{t}(q) \,dq  -\int_0^1  \rho_{0}(q) G_{0}(q) \,dq 
-\int_0^t\int_0^1 \rho_{s}(q)  \Big(\tfrac{\hat \sigma^{2}}{2}\Delta+\partial_s\Big) G_{s}(q)\,dqds \right.\\
 & -  \int^{t}_{0}< \pi_s, \iota^0_\ve>\Big(\tfrac{\hat \sigma^{2}}{2}\partial_{q}G_{s}(0) - \hat m G_{s}(0\Big)\, ds + \int^{t}_{0}  <\pi_s, \iota^1_\ve>\Big(\tfrac{\hat \sigma^{2}}{2} \partial_{q}G_{s}(1)+ {\hat m}  G_{s}(1)\Big)\,ds \\
 &\left.- {\hat m}\int^{t}_{0}  G_{s}(0)\alpha+ G_{s}(1)\beta \,ds  \Big|>\dfrac{\delta}{4}\right),
\end{split}
\end{equation}
\begin{equation}
\label{RC71}
\bb Q \left(  \Big| \int^{1}_{0} (\rho_{0}(q)-g(q)) G_{0}(q)\, dq \Big|>\dfrac{\delta}{4}\right),
\end{equation}
\begin{equation}
\label{RC7}
\bb Q\left(  \sup_{0\le t \le T}  \Big| \int^{t}_{0} \left[ \rho_{s}(0)-< \pi_s, \iota^0_\ve> \right]\left[\hat m G_{s}(0)-\dfrac{ \hat \sigma^{2}}{2} \partial_{q}G_{s}(0)  \right]ds  \Big|>\dfrac{\delta}{4}\right),
\end{equation}
and
\begin{equation}
\label{RC8}
\bb Q\left(  \sup_{0\le t \le T}  \Big|\int^{t}_{0} \left[\rho_{s}(1)- < \pi_s, \iota^1_\ve>\right](\hat m G_{s}(1) +\dfrac{\hat \sigma^{2}}{2} \partial_{q}G_{s}(1)  )ds  \Big|>\dfrac{\delta}{4}\right). 
\end{equation}
We note that the terms  (\ref{RC7}) and (\ref{RC8}) converge to $0$ as $\ve \to 0$  since we are comparing $\rho_s(0)$ (resp. $\rho_s(1)$) with the corresponding  average around the boundary points $0$ (resp. $1$) and (\ref{RC71}) is equal to zero since $\mathbb Q$ is a limit point of $\{\mathbb Q_N\}_{N\in\mathbb N}$ and $\mathbb Q_N$ is induced by $\mu_N$ which satisfies  \eqref{assoc_mea}. Therefore it remains only to consider (\ref{RC6}).  We still cannot use Portmanteau's Theorem, since the functions $\iota^0_\ve$ and $\iota_\ve^1$ are not continuous. Nevertheless,  we can approximate each one of these functions by continuous functions  in such a way that the error vanishes as $\ve \to 0$. Then, from Proposition A.3 of \cite{FGN} we can use Portmanteau's Theorem and   bound (\ref{RC6}) from above by
\begin{equation}
\label{RC1111}
\begin{split}
&\liminf_{N\to\infty}\,\bb Q_{N}\left( \sup_{0\le t \le T} \Big| \int_0^1 \rho_{t}(q)  G_{t}(q) \,dq  -\int_0^1  \rho_{0}(q)  G_{0}(q) \,dq \right.\\
&-\int_0^t\int_0^1 \rho_{s}(q)  \Big(\dfrac{\hat \sigma^{2}}{2}\Delta+\partial_s\Big) G_{s}(q)\,dqds \\
&  -\int^{t}_{0}< \pi_s, \iota^0_\ve>\Big(\dfrac{\hat \sigma^{2}}{2}\partial_{q}G_{s}(0) - \hat m G_{s}(0\Big)\, ds + \int^{t}_{0}  <\pi_s, \iota^1_\ve>\Big(\dfrac{\hat \sigma^{2}}{2} \partial_{q}G_{s}(1)+ {\hat m}  G_{s}(1)\Big)\,ds \\
& \left. - {\hat m}\int^{t}_{0}  G_{s}(0)\alpha+ G_{s}(1)\beta \,ds  \Big|>\dfrac{\delta}{2^{4}}\right).
\end{split}
\end{equation} 
Summing and subtracting $\displaystyle\int_{0}^{t} N^{2}L_{N}\langle \pi_{s}^{N},G_{s}\rangle ds$ to the term inside the supremum in  (\ref{RC1111}), recalling \eqref{Dynkin'sFormula} and \eqref{boxes}, the definition of $\mathbb Q_N$,
we bound \eqref{RC1111}  from above by the sum of the next two terms
\begin{equation}\label{RC12}
\liminf_{N\to\infty}\,\bb P_{\mu_N}\left( \sup_{0\le t \le T} \left\vert M_{t}^{N}(G) \right\vert>\dfrac{\delta}{2^{5}}\right),
\end{equation} 
and
\begin{equation}
\label{RC13}
\begin{split}
&\liminf_{N\to\infty}\,\bb P_{\mu_N}\left(  \sup_{0\le t \le T} \Big|\int_0^t N^{2}L_{N}\langle \pi_{s}^{N},G_{s}\rangle\,ds-\dfrac{\hat \sigma^{2}}{2}\int_0^t\int_0^1 \rho_{s}(q) \Delta G_{s}(q)\,dqds \right.\\
 & - \int^{t}_{0}\overrightarrow{\eta}_0^{\ve N}(s)\Big(\dfrac{\hat \sigma^{2}}{2}\partial_{q}G_{s}(0) - \hat m G_{s}(0\Big)\, ds + \int^{t}_{0}  \overleftarrow{\eta}_{N-1}^{\ve N}(s)\Big(\dfrac{\hat \sigma^{2}}{2} \partial_{q}G_{s}(1)+ {\hat m}  G_{s}(1)\Big)\,ds\\
 &\left.  -{\hat m}\int^{t}_{0}  G_{s}(0)\alpha+ G_{s}(1)\beta\,ds  \Big|>\dfrac{\delta}{2^{5}}\right).
\end{split}
\end{equation} 
From Doob's inequality together with \eqref{T6}, (\ref{RC12}) goes to $0$ as $N\to\infty$. Finally, (\ref{RC13}) can be rewritten as  
 \begin{equation}
 \label{RC14}
 \begin{split}
&\liminf_{N\to\infty}\,\bb P _{\mu _{N}}\left( \sup_{0\le t \le T}  \Big| \int_0^t N^{2}L_{N}\langle \pi_{s}^{N},G_{s}\rangle\,ds-\dfrac{\hat \sigma^{2}}{2}\int_0^t \langle \pi_{s}^{N} ,\Delta G_{s}\rangle\,ds \right. \\
& -\int^{t}_{0}\overrightarrow{\eta}_0^{\ve N}(s)\Big(\dfrac{\hat \sigma^{2}}{2}\partial_{q}G_{s}(0) - \hat m G_{s}(0\Big)\, ds + \int^{t}_{0}  \overleftarrow{\eta}_{N-1}^{\ve N}(s)\Big(\dfrac{\hat \sigma^{2}}{2} \partial_{q}G_{s}(1)+ {\hat m}  G_{s}(1)\Big)\,ds \\
&\left. -{\hat m}\int^{t}_{0}  G_{s}(0)\alpha+ G_{s}(1)\beta\,ds   \Big|>\dfrac{\delta}{2^{5}}\right).
\end{split}
\end{equation} 
Now, from (\ref{gen_action}) and (\ref{bulkaction}) we can bound from above the probability in (\ref{RC14}) by the sum of the five following terms
\begin{equation}
\label{RC15}
\bb P_{\mu_{N}}  \left(\sup_{0\le t \le T} \Big| \cfrac{N^{2}}{N-1}\int_{0}^{t} \sum_{x\in \Lambda_N} K_NG_s(\tfrac{x}{N})\eta_x(sN^2) ds  - \dfrac{\hat \sigma^{2}}{2} \int_{0}^{t}\left\langle \pi_{s}^{N},\Delta G_{s} \right\rangle   \, ds \Big|>\dfrac{\delta}{2^{6}}\right),
\end{equation}
\begin{equation}
\label{RC16}
\begin{split}
&\bb P_{\mu_{N}}  \left( \sup_{0\le t \le T}  \Big| \cfrac{N^{2}}{N-1}\int_{0}^{t} \sum_{x\in \Lambda_N}\sum_{y\leq 0}\left[G_{s}(\tfrac{y}{N})- G_{s}(\tfrac{x}{N})\right] p(x-y)\eta_x(sN^2) ds \right. \\
&\qquad \qquad \qquad \qquad\qquad \left. +\dfrac{\hat \sigma^{2}}{2} \int_{0}^{t} \overrightarrow{\eta}_0^{\ve N}(sN^{2}) \partial_{q}G_{s}(0)   \, ds  \Big|>\dfrac{\delta}{2^{6}}\right),
\end{split}
\end{equation}
and
\begin{equation}
\label{RC18}
\begin{split}
 &\bb P_{\mu_{N}}\left(\sup_{0\le t \le T} \Big|  \int_{0}^{t}\cfrac{N\kappa}{N-1}\sum_{x \in \Lambda_N}  ( G_{s} r_{N}^{-})(\tfrac{x}{N}){(\alpha-\eta_x(sN^2))} \, ds\right.\\
 &\qquad \qquad \qquad \qquad\qquad \left. -m\kappa   \int_0^t G_{s}(0)(\alpha-\overrightarrow{\eta}_0^{\ve N}(sN^2)) ds\Big| > \dfrac{\delta}{2^{6}}\right) 
 \end{split}
\end{equation}
and the sum of two terms which are very similar to the two previous ones but which are concerned with the right boundary.  Thus, to conclude we have to show that these five terms go to $0$. Applying Lemma \ref{convergence laplacian} and noting that $|\eta_x (sN^2)| \le 1$ for any $x$ and any $s \ge 0$, we conclude  that (\ref{RC15}) goes to $0$ as $N\to \infty$. Note also that by Taylor expansion, we can bound from above (\ref{RC16})  by 
\begin{equation}
\label{RC20}
\bb P_{\mu_{N}}  \left( \sup_{0\le t \le T} \Big| \int_{0}^{t} \partial_{q}G_{s}(0)  \sum_{x \in \Lambda_{N}} \Theta_{x}^{-}\left[ \eta_{x}(s N^{2})- \overrightarrow{\eta}_0^{\ve N}(sN^{2})\right]  ds \Big|>\dfrac{\delta}{2^{8}}\right).
\end{equation}
Using Lemma \ref{Rep-Neumann} we see that (\ref{RC20}) vanishes as $N\to \infty$. Now we look at (\ref{RC18}) and we prove that is vanishes as  $N\to\infty$.  Performing a Taylor expansion on $G_s$ at $0$ and using (\ref{mean}) the probability in (\ref{RC18}) is bounded from above by 
\begin{eqnarray}\label{RC21}\nonumber
 &&\bb P_{\mu_{N}}\left(\sup_{0\le t \le T} \Big|  \int_{0}^{t}G_{s}(0)\sum_{x\in\Lambda_{N}}r_{N}^{-}(\tfrac{x}{N})\Big[\overrightarrow{\eta}_{0}^{\ve N}(sN^{2}) -\eta_{x}(sN^{2})\Big] ds \Big|  > \dfrac{\delta}{2^{8}} \right), 
\end{eqnarray}
plus lower-order terms (with respect to $N$). From Lemma \ref{Rep-Neumann} and Remark \ref{sec_rep_robin} last display vanishes as $N \to \infty$. Similarly the two terms which are similar to (\ref{RC16}) and (\ref{RC18}) but which are concerned with the right boundary vanish as $N \to \infty$. Thus the proof is finished.

\vspace{0.5cm}
\noindent
$\bullet$ Now we treat the case $\theta<1$. We have to prove that
\begin{eqnarray} \nonumber
&&\bb Q\left(\pi _{\cdot}\in \mc D([0,T], \mathcal{M^{+}}): \sup_{0\le t \le T} \left\vert F_{RD}(t,\rho,G,g) \right\vert>\delta\right)=0
\end{eqnarray}
for any $G \in C_c^{1,2} ([0,T]\times[0,1])$. We can bound from above the previous probability by 
\begin{equation}\label{RD6}
\begin{split}
&\bb Q\left(  \sup_{0\le t \le T} \Big|\int_0^1 \rho_{t}(q)  G_{t}(q) \,dq  -\int_0^1  \rho_{0}(q) G_{0}(q) \,dq 
-\int_0^t\int_0^1 \rho_{s}(q)  \Big(\tfrac{\hat \sigma^{2}}{2}\Delta+\partial_s\Big) G_{s}(q)\,dqds \right.\\
 &\left. - {\hat \kappa}\int^{t}_{0}  \int_0^1G_{s}(q)V_0(q)\,dq\,ds+ {\hat \kappa}\int_0^t\int_0^1G_s(q)\rho_s(q)V_1(q)\, dq\,ds  \Big|>\dfrac{\delta}{2}\right),
\end{split}
\end{equation}
and
\begin{equation}
\label{RD71}
\bb Q \left(  \Big| \int^{1}_{0} (\rho_{0}(q)-g(q)) G_{0}(q)\, dq \Big|>\dfrac{\delta}{2}\right),
\end{equation}
where $V_0(q)=\frac{\alpha}{q^\gamma}+ \frac{\beta}{(1-q)^\gamma}$ and $V_1(q)=\frac{1}{q^\gamma}+\frac{1}{(1-q)^\gamma}$
We note that (\ref{RD71}) is equal to zero since $\mathbb Q$ is a limit point of $\{\mathbb Q_N\}_{N\in\mathbb N}$ and $\mathbb Q_N$ is induced by $\mu_N$ which satisfies  \eqref{assoc_mea}.
We note that  from Proposition A.3 of \cite{FGN}, the set inside the probability in \eqref{RD6} is an open set in the Skorohod space (the singularities of $V_0$ and $V_1$ are not present because $G_s$ has compact support).  From Portmanteau's Theorem we bound \eqref{RD6} from above by
\begin{eqnarray}\label{RC11}\nonumber
&&\liminf_{N\to\infty}\,\bb Q_{N}\left( \sup_{0\le t \le T} \Big| \int_0^1 \rho_{t}(q)  G_{t}(q) \,dq  -\int_0^1  \rho_{0}(q)  G_{0}(q) \,dq \right.\\\nonumber
& &-\int_0^t\int_0^1 \rho_{s}(q)  \Big(\tfrac{\hat \sigma^{2}}{2}\Delta+\partial_s\Big) G_{s}(q)\,dqds   - {\hat \kappa}\int^{t}_{0}\int_0^1  G_{s}(q)V_0(q) \,dq\,ds \\\nonumber
& &\left.  +{\hat \kappa}\int^{t}_{0}\int_0^1  G_{s}(q)\rho_s(q)V_1(q) \,dq\,ds 
 \Big|>\dfrac{\delta}{2}\right) .
\end{eqnarray} 
Summing and subtracting $\displaystyle\int_{0}^{t} \Theta(N)L_{N}\langle \pi_{s}^{N},G_{s}\rangle ds$ to the term inside the previous absolute values, recalling \eqref{Dynkin'sFormula} and  the definition of $\mathbb Q_N$,  we can bound the previous probability   from above by the sum of the next two terms 
\begin{eqnarray}\label{CLP1} \nonumber
 & &\bb P_{\mu_{N}} \left(\sup_{0\le t \le T} \left\vert M_{t}^{N}(G) \right\vert>\dfrac{\delta}{4}\right),
\end{eqnarray}
and
\begin{equation}
\label{CLP2}
\begin{split}
&\bb P_{\mu_{N}}  \left( \sup_{0\le t \le T} \Big| \int_{0}^{t} \Theta(N)L_{N}\langle \pi_{s}^{N},G_{s}\rangle ds -\int_0^t\left\langle \pi_{s}^{N},\tfrac{\sigma^{2}}{2}\Delta  G_{s} \right\rangle \,ds \right.\\
&- {\hat \kappa}\int^{t}_{0}\int_0^1  G_{s}(q)V_0(q) \,dq\,ds
\left. +{\hat \kappa}\int^{t}_{0}\int_0^1  G_{s}(q)\rho_s(q)V_1(q) \,dq\,ds  \Big|>\dfrac{\delta}{4}\right). 
\end{split}
\end{equation}
The first term above can be estimated as in the case $\theta\geq 1$ and it vanishes as $N\to\infty$.  It remains to prove that (\ref{CLP2}) vanishes as $N\to\infty$. For that purpose, we recall \eqref{eq:rs-r+-} and  we use (\ref{gen_action}), (\ref{bulkaction})  to bound it from above by the sum of the following terms
\begin{equation}
\label{DC1}
\bb P_{\mu_{N}}  \left(\sup_{0\le t \le T} \Big| \int_{0}^{t}\cfrac{\Theta(N)}{N-1} \sum_{x\in \Lambda_N}K_NG_{s}(\tfrac{x}{N})\eta_x(sN^2) ds- \dfrac{\hat\sigma^{2}}{2} \int_{0}^{t}\left\langle \pi_{s}^{N},\Delta G_{s} \right\rangle   \, ds \Big|>\dfrac{\delta}{2^{4}}\right),
\end{equation}
and
\begin{equation}
\label{DC2}
\begin{split}
&\bb P_{\mu_{N}}\left(\sup_{0\le t \le T} \Big|   \int_{0}^{t} \Big\{ \tfrac{ \kappa\Theta(N)}{(N-1)N^\theta} \sum_{x \in \Lambda_N}  (G_{s} r_{N}^{-})(\tfrac{x}{N}){(\alpha-\eta_x(sN^2))}  \right.\\
& \quad \quad \quad \quad \quad \quad \quad \quad   \left.- \hat \kappa \int_{0}^{1}  (G_{s}r^{-})(q)(\alpha - \rho_{s}(q))dq \Big\}\, ds\Big| > \dfrac{\delta}{2^{4}}\right),
\end{split}
\end{equation}
and
\begin{equation}
\label{DC3}
\begin{split}
&\bb P_{\mu_{N}}\left(\sup_{0\le t \le T} \Big|   \int_{0}^{t} \Big\{ \tfrac{ \kappa\Theta(N)}{(N-1)N^\theta} \sum_{x \in \Lambda_N}  (G_{s} r_{N}^{+})(\tfrac{x}{N}){(\beta-\eta_x(sN^2))} \right.\\
& \quad \quad \quad \quad \quad \quad \quad \quad  \left.- \hat \kappa \int_{0}^{1}  (G_{s}r^{+})(q)(\beta - \rho_{s}(q))dq \Big\}\, ds\Big| > \dfrac{\delta}{2^{4}}\right),
\end{split}
\end{equation}

In the case $\theta\in[2-\gamma,1)$, since $\Theta(N)=N^2$ and $\hat\sigma=\sigma$,   from  Lemma \ref{convergence laplacian}  we have that (\ref{DC1}) goes to $0$ as $N\to \infty$.
In the case $\theta<2-\gamma$, since $\Theta(N)=N^{\theta+\gamma}$ and $\hat\sigma=0$,   from  Lemma \ref{convergence laplacian}  we also have that (\ref{DC1}) goes to $0$ as $N\to \infty$.

Now we analyze the boundary terms (\ref{DC2}) and (\ref{DC3}). 
Note that in the case $\theta\in(2-\gamma,1)$ we have $\theta(N)=N^2$ and  $\hat\kappa=0$, so that the two previous probabilities vanish, as $N\to\infty$, as a consequence of Lemma \ref{Rep-Dirichlet1}.  In the case $\theta\le 2-\gamma$,  since  $\Theta(N)=N^{\gamma+\theta}$, $\hat\kappa=\kappa c_\gamma \gamma^{-1}$,  $\vert \eta_{x}(sN^{2})\vert \leq 1$, in order to conclude it is enough to note that 
since $G_{s}$ has compact support in $(0,1)$ we know by (\ref{eq:rs-r+-}) that  $N^{\gamma}G_{s}r^{-}_{N}(q)$ and $N^{\gamma}G_{s}r^{+}_{N}(q)$ converge uniformly to $(G_{s}r^{-})(q)$ and $(G_{s}r^{+})(q)$, respectively, as $N\to\infty$. 
This ends the proof.
\end{proof}

\section*{Acknowledgements}
This work has been supported by the projects EDNHS ANR-14- CE25-0011, LSD ANR-15-CE40-0020-01 of the French National Research Agency (ANR) and of the PHC Pessoa Project 37854WM. B.J.O. thanks Universidad Nacional de Costa Rica for financial support through his Ph.D grant.

This project has received funding from the European Research Council (ERC) under  the European Union's Horizon 2020 research and innovative programme (grant agreement   No 715734). 

This work was finished during the stay of P.G. at Institut Henri Poincar\'e - Centre Emile Borel during the trimester "Stochastic Dynamics Out of Equilibrium". P.G. thanks this institution for hospitality and support.

\appendix

\section{Uniqueness of weak solutions}
\label{sec:app-unique}

The uniqueness of the weak solutions of the partial equations given in Section \ref{sec:hyd_eq}  is fundamental for the proof of the hydrodynamic limit. The uniqueness of weak solutions of  (\ref{eq:Dirichlet source Equation-g}) is standard if $\hat \kappa =0$. Since we were not able to find in the literature a proof in the case $\hat \kappa>0$ we give a complete proof below. The proof of uniqueness of weak solutions of (\ref{Robin Equation-g}) can be found in, for example, \cite{Adriana}. 

Now  we prove the uniqueness of weak solutions of  \eqref{eq:Dirichlet source Equation-g}. We assume that $\hat \sigma>0$ and $\hat \kappa>0$ first and then we consider the case $\hat \sigma=0$ and $\hat \kappa>0$.

Let $\rho^1$ and $\rho^2$ be  two weak solutions of (\ref{eq:Dirichlet source Equation-g}) with the same initial condition and let us denote $\bar \rho = \rho^1 -\rho^2$. By assumption we have that 
$$\bar \rho \in L^2 \Big(0,T ; {\mc H}^1\Big) \cap L^2 \Big( 0,T; L^2 ((0,1); V_1(q) dq) \,\Big) $$ 
where $V_1(q)=q^{-\gamma} + (1-q)^{-\gamma}$. Let us denote by $\langle \cdot, \cdot \rangle_{V_1}$ (resp. $\| \cdot\|_{V_1}$) the scalar product (resp. the norm) corresponding to the Hilbert space  $L^2 ((0,1), V_1(q) dq)$.

For almost every $t\in [0,T]$, we identify ${\bar \rho}_t$ with its continuous representation in $[0,1]$. Therefore, from {Remark \ref{use:rem_dir}},  we have that ${\bar \rho}_t (0)={\bar \rho}_t  (1)=0$ for all $t\in[0,T]$. Since ${\mc H}_0^{1}$ is equal to the set of functions in ${\mc H}^{1}$ vanishing at $0$ and $1$ we have that for a.e. time $t \in [0,T]$, ${\bar \rho}_t \in {\mc H}_0^{1}$ and in fact ${\bar \rho} \in L^2 (0,T;{\mc H}_0^{1})$. From \emph{2.} in Definition \ref{Def. Dirichlet source Condition-g},  for any $t \in [0,T]$ and any $G \in C_c^{1,2} ([0,T] \times [0,1])$ we have
\begin{equation}
\label{eq:unique007}
\begin{split}
&\int_0^1 {\bar \rho}_{t}(q)  G_{t}(q) \,dq - \int_0^t \int_0^1 {\bar \rho}_{s}(q)\Big(\partial_s + \dfrac{\hat \sigma^{2}}{2}\Delta \Big) G_{s}(q)  \,dq ds \\
&\quad +\hat \kappa \int^{t}_{0} \int_0^1 V_1(q)  G_s (q) {\bar \rho}_s (q) \,dq\, ds=0.
\end{split}
\end{equation}
We know that $C_c^{1,\infty} ([0,T] \times (0,1))$ is dense in $L^2 (0,T; {\mc H}_0^{1}) \cap L^2 \Big( 0,T; L^2 ((0,1); V(q) dq) \,\Big)$. Therefore, let $(H_n)_{n \ge 0}$ be a sequence of functions in  $C_c^{1,\infty} ([0,T] \times (0,1))$ converging to ${\bar \rho}$ with respect to the norms of $L^2 (0,T; {\mc H}_0^{1})$ and $ L^2 \Big( 0,T; L^2 ((0,1); V_1(q) dq) \,\Big)$. We define  $G_n$ in $C_{c}^{1,\infty}([0,T]\times [0,1])$ by 
\begin{equation}
\label{def_G_n}
\forall t \in [0,T], \quad \forall q \in [0,1], \quad G_n (t,q)= \int_t^T  H_n (s,q) \,ds.
\end{equation}
Plugging $G_n$ into (\ref{eq:unique007}) and letting $n \to \infty$ we conclude, by Lemma \ref{lem:unique-conv-007} below, that
\begin{equation*}
\int_0^T \int_0^1 {\bar \rho}^2_s (q) \, dq\, ds + \cfrac{\hat \sigma^2}{4} \; \Big\| \int_0^T {\bar \rho}_s ds \Big\|^2_{1} + \cfrac{\hat \kappa}{2}\;  \Big\|\int_0^T {\bar \rho}_s ds  \, \Big\|_{V_1}^2 =0.
\end{equation*}

It follows that for almost every time $s \in [0,T]$ the continuous function $\bar \rho_s$ is equal to $0$ and we conclude the uniqueness of weak solution to  (\ref{eq:Dirichlet source Equation-g}) in the case $\hat \sigma>0$.

\begin{lem}
\label{lem:unique-conv-007}
Let $(G_n)_{n\geq 0}$ be defined as in \eqref{def_G_n}. We have
\begin{enumerate}[i)]
\item $\lim_{n \to \infty}  \int_0^T \int_0^1 {\bar \rho}_s (q) \, (\partial_s G_n) (s,q) \, dq ds = - \int_0^T \int_0^1 {\bar \rho}^2_s (q) \, dq ds$. 
\item $\lim_{n \to \infty} \int_0^T \int_0^1 {\bar \rho}_{s}(q)  \Delta G_{n} (s,q)  \,dq ds =- \; \frac{1}{2}\Big\| \int_0^T {\bar \rho}_s ds \Big\|^2_{1}.$ 
\item $\lim_{n \to \infty} \int^{T}_{0} \int_{0}^1V_1(q)  G_n (s,q) {\bar \rho}_s (q) \,dq\, ds =\;\frac{1}{2}  \Big\|\int_0^T {\bar \rho}_s ds  \, \Big\|_{V_1}^2 < \infty.$
\end{enumerate}
\end{lem}

\begin{proof}
For i) we write
\begin{equation*}
\begin{split}
&-\int_0^T \int_0^1 {\bar \rho}_s (q) \, (\partial_s G_n) (s,q) \, dq ds = \int_0^T \int_0^1 {\bar \rho}_s (q) \, H_n (s,q) \, dq ds = \int_0^T  \langle {\bar \rho}_s \, , \, H_n (s, \cdot) \rangle \, ds\\
& = \int_0^T  \big\langle {\bar \rho}_s \, , \, H_n (s, \cdot) - {\bar \rho}_s \big\rangle \, ds \; + \; \int_0^T \| {\bar \rho}_s \|^2_{L^2} \, ds.
\end{split}
\end{equation*}
Observe then that by Cauchy-Schwarz inequality we have
\begin{equation}\label{use_uniq}
\begin{split}
& \left| \int_0^T  \big\langle {\bar \rho}_s \, , \, H_n (s, \cdot) - {\bar \rho}_s \big\rangle \, ds \right| \le \int_0^T \| {\bar \rho}_s \|_{L^2} \, \| H_n (s, \cdot) - {\bar \rho}_s \|_{L^2} \, ds\\
&\le \sqrt{ \int_0^T  \| {\bar \rho}_s \|_{L^2}^2 \, ds} \; \sqrt{ \int_0^T  \| H_n (s, \cdot) - {\bar \rho}_s \|_{L^2}^2 \, ds } 
\end{split} 
\end{equation}
which goes to $0$ as $n \to \infty$. Above we have used the fact that $(H_n)_{n\geq 0}$ converges to $\bar \rho$ as $N\to\infty$ with respect to the norm of $L^2 (0,T;{\mc H}_0^{1})$.

For ii) we first use the integration by parts formula  for ${\mc H}_1$ functions which permits to write
$$ \int_0^T \int_0^1 {\bar \rho}_{s}(q)  \; \Delta G_{n} (s,q)  \,dq ds = - \int_0^T \Big\langle \bar{ \rho}_s \,  , G_n (s, \cdot) \, \Big\rangle_{1}\, ds.$$ 
Then we have
\begin{equation*}
\begin{split}
&\int_0^T \Big\langle\bar{ \rho}_s \,  , G_n (s, \cdot) \, \Big\rangle_{1}\, ds = \int_0^T \Big\langle \bar{\rho}_s \,  , \int_s^T {\bar \rho}_u du  \, \Big\rangle_{1}\, ds + \int_0^T \Big\langle\bar{ \rho}_s \,  , G_n (s, \cdot) - \int_s^T {\bar \rho}_u du \, \Big\rangle_{1}\, ds\\
&= \iint_{0 \le s < u \le T} \langle {\bar \rho}_s \, , \, {\bar \rho}_u \rangle_{1} \, du\, ds \; + \;  \int_0^T \Big\langle \bar{\rho}_s \,  ,  \int_s^T \{ H_n (u, \cdot) -{\bar \rho}_u\}  du \, \Big\rangle_{1}\, ds\\
&= \cfrac{1}{2} \, \iint_{[0,T]^2} \langle {\bar \rho}_s \, , \, {\bar \rho}_u \rangle_{1} \, du ds \; + \;  \int_0^T \Big\langle \bar{\rho}_s \,  , \int_s^T \{ H_n (u, \cdot) -{\bar \rho}_u\}  du \, \Big\rangle_{1}\, ds\\
&=\cfrac{1}{2} \; \Big\| \int_0^T {\bar \rho}_s ds \Big\|^2_{1}+ \;  \int_0^T \Big\langle \bar{\rho}_s \,  , \int_s^T \{ H_n (u, \cdot) -{\bar \rho}_u\}  du \, \Big\rangle_{1}\, ds.
\end{split}
\end{equation*}
To conclude the proof of ii) it is sufficient to prove that
$$\lim_{n \to \infty} \int_0^T \Big\langle \bar{\rho}_s \,  , \int_s^T \{ H_n (u, \cdot) -{\bar \rho}_u\}  du \, \Big\rangle_{1}\, ds =0.$$
This is a consequence of a successive use of Cauchy-Schwarz inequalities:
\begin{equation*}
\begin{split}
&\left| \int_0^T \Big\langle \bar{\rho}_s \,  , \int_s^T \{ H_n (u, \cdot) -{\bar \rho}_u\}  du \, \Big\rangle_{1}\, ds \right| \le \int_0^T \Big\| \bar{\rho}_s \Big\|_{1} \; \Big\| \int_s^T \{ H_n (u, \cdot) -{\bar \rho}_u \} du \Big\|_{1}\, ds\\
&\le \int_0^T \Big\| \bar{\rho}_s \Big\|_{1} \; \int_s^T \Big\| H_n (u, \cdot) -{\bar \rho}_u  \Big\|_{1}\, du \,  ds \\
&\le  \left(\int_0^T \Big\| \bar{\rho}_s \Big\|_{1} ds \right) \, \left(  \int_0^T \Big\|  H_n (u, \cdot) -{\bar \rho}_u  \Big\|_{1}\, du \right)\\
& \le T \, \sqrt{ \int_0^T \Big\| \bar{\rho}_s \Big\|^2_{1} ds} \; \sqrt{  \int_0^T \Big\|  H_n (u, \cdot) -{\bar \rho}_u  \Big\|_{1}^2\, du} \; \xrightarrow[n \to \infty]{} \; 0.
\end{split}
\end{equation*}
Above we have used again the fact that $(H_n)_{n\geq 0}$ converges to $\bar \rho$ as $N\to\infty$ with respect to the norm of $L^2 (0,T;{\mc H}_0^{1})$.

The proof of iii) is similar. We have
\begin{equation*}
\begin{split}
&\int_0^T \Big\langle \bar{\rho}_s \,  , G_n (s, \cdot) \, \Big\rangle_{V_1}\, ds = \int_0^T \Big\langle \bar{\rho}_s \,  , \int_s^T {\bar \rho}_u du  \, \Big\rangle_{V_1}\, ds + \int_0^T \Big\langle \bar{\rho}_s \,  , G_n (s, \cdot) - \int_s^T {\bar \rho}_u du \, \Big\rangle_{V_1}\, ds\\
&=\iint_{0 \le s < u \le T} \langle {\bar \rho}_s \, , \, {\bar \rho}_u \rangle_{V_1} \, du\, ds \; + \;  \int_0^T \Big\langle \bar{\rho}_s \,  ,  \int_s^T \{ H_n (u, \cdot) -{\bar \rho}_u\}  du \, \Big\rangle_{V_1}\, ds\\
&= \cfrac{1}{2} \, \iint_{[0,T]^2} \langle {\bar \rho}_s \, , \, {\bar \rho}_u \rangle_{V_1} \, du\, ds \; + \;  \int_0^T \Big\langle \bar{\rho}_s \,  , \int_s^T \{ H_n (u, \cdot) -{\bar \rho}_u\}  du \, \Big\rangle_{V_1}\, ds\\
&=\cfrac{1}{2} \; \Big\| \int_0^T {\bar \rho}_s ds \Big\|^2_{V_1}+ \;  \int_0^T \Big\langle \bar{\rho}_s \,  , \int_s^T \{ H_n (u, \cdot) -{\bar \rho}_u\}  du \, \Big\rangle_{V_1}\, ds.
\end{split}
\end{equation*}
To conclude the proof of iii) it is sufficient to prove that
$$\lim_{n \to \infty} \int_0^T \Big\langle \bar{\rho}_s \,  , \int_s^T \{ H_n (u, \cdot) -{\bar \rho}_u\}  du \, \Big\rangle_{V_1}\, ds =0.$$
This is a consequence of the Cauchy-Schwarz inequality:
\begin{equation*}
\begin{split}
&\left| \int_0^T \Big\langle \bar{\rho}_s \,  , \int_s^T \{ H_n (u, \cdot) -{\bar \rho}_u\}  du \, \Big\rangle_{V_1}\, ds \right| \le \int_0^T \Big\| \bar{\rho}_s \Big\|_{V_1} \; \Big\| \int_s^T \{ H_n (u, \cdot) -{\bar \rho}_u \} du \Big\|_{V_1}\, ds\\
&\le \int_0^T \Big\| \bar{\rho}_s \Big\|_{V_1} \; \int_s^T \Big\| H_n (u, \cdot) -{\bar \rho}_u  \Big\|_{V_1}\, du \,  ds \\
&\le \left(\int_0^T \Big\| \bar{\rho}_s \Big\|_{V_1} ds \right) \, \left(  \int_0^T \Big\|  H_n (u, \cdot) -{\bar \rho}_u  \Big\|_{V_1}\, du \right)\\
& \le T \, \sqrt{ \int_0^T \Big\| \bar{\rho}_s \Big\|^2_{V_1} ds} \; \sqrt{  \int_0^T \Big\|  H_n (u, \cdot) -{\bar \rho}_u  \Big\|_{V_1}^2\, du} \; \xrightarrow[n \to \infty]{} \; 0.
\end{split}
\end{equation*} 
\end{proof}

{Note that when $\hat \sigma>0$ and $\hat \kappa=0$ the proof above also shows uniqueness of the weak solution of the heat equation with Dirichlet boundary conditions.}

Now we look at the  case $\hat \sigma=0$. In this case we do not have any regularity assumption on $\bar \rho(\cdot)$. However, it can be proved that 
\begin{equation}
\int_0^T \int_0^1 {\bar \rho}^2_s (q) \, dq ds + \cfrac{\hat \kappa}{2}\;  \Big\|\int_0^T {\bar \rho}_s ds  \, \Big\|_{V_1}^2 =0
\end{equation}
holds by showing only the first and third item of the previous lemma. This requires only the density of $C_c^{1,\infty} ([0,T] \times (0,1))$ in $L^2 \Big( 0,T; L^2 ((0,1); V_1(q) dq) \,\Big)$. We also note that in the proof of item \emph{i)} in Lemma \ref{lem:unique-conv-007}, in order to conclude the convergence in  \eqref{use_uniq}, before applying the Cauchy-Schwarz inequality, we multiply and divide the integrand function by $V_1$ and since $V_1^{-1}$ is bounded we get that $\| {\bar \rho}_sV_1^{-1} \|_{L^2}^2<\infty$ and  the result follows.

\section{Computations involving the generator} \label{sec:gen_comp}
\label{sec:compgen}

\begin{lem}
\label{lem:compA}
For any $x \ne y \in \Lambda_N$, we have 
\begin{equation}
\begin{split}
L_N^0 (\eta_x \eta_y) &= \eta_x L_N^0 \eta_y + \eta_y L_N^0 \eta_x -  p(y-x) (\eta_y -\eta_x)^2,\\
L_N^r (\eta_x \eta_y) &= \eta_x L_N^r \eta_y + \eta_y L_N^r \eta_x,\\
L_N^\ell (\eta_x \eta_y) &= \eta_x L_N^\ell \eta_y + \eta_k L_N^\ell \eta_x. 
\end{split}
\end{equation}
\begin{proof}
By definition of $L_N^0$ we have that
\begin{equation*}
\begin{split}
L_N^0 (\eta_x \eta_y) &= \cfrac{1}{2}\sum _{u,v\in \Lambda_{N}}p(v-u)\left[(\sigma^{u,v}\eta_{x})(\sigma^{u,v}\eta_{y})-\eta_{x}\eta_{y}\right]\\
&=\cfrac{1}{2}\sum _{u,v\in \Lambda_{N}}p(v-u)\left[((\sigma^{u,v}\eta_{x})\eta_{y}-\eta_{x}\eta_{y})+((\sigma^{u,v}\eta_{y})\eta_{x}-\eta_{x}\eta_{y})+\right. \\
&\;\left. +(\sigma^{u,v}\eta_{x})(\sigma^{u,v}\eta_{y})-(\sigma^{u,v}\eta_{x})\eta_{y}-(\sigma^{u,v}
\eta_{y})\eta_{x}+\eta_{x}\eta_{y}\right]\\
&=\eta_x L_N^0 \eta_y + \eta_y L_N^0 \eta_x + \cfrac{1}{2}\sum _{u,v\in \Lambda_{N}}p(v-u)\left[(\sigma^{u,v}\eta_{x})-\eta_{x}\right] \left[ (\sigma^{u,v}\eta_{y})-\eta_{y}\right] \\
&= \eta_x L_N^0 \eta_y + \eta_y L_N^0 \eta_x -  p(y-x) (\eta_y -\eta_x)^2.
\end{split}
\end{equation*}
In order to prove the second expression, note that $\left[(\sigma^{u}\eta_{x})-\eta_{x}\right] \left[ (\sigma^{u}\eta_{y})-\eta_{y}\right] = 0$, for all $u \in \bZ$, thus by definition of $L_N^r$ we have
\begin{equation*}
\begin{split}
L_N^r (\eta_x \eta_y)&=\sum_{u\in\Lambda_{N},v\ge N}p(v-u)\left[\eta_{u}(1-\beta)+(1-\eta_{u})\beta\right]\left[(\sigma^{u}(\eta_{x}\eta_{y}))-\eta_{x}\eta_{y}\right]\\
&= \eta_x L_N^r \eta_y + \eta_y L_N^r \eta_x +\\
&+\sum_{u\in\Lambda_{N},v\ge N}p(v-u)\left[\eta_{u}(1-\beta)+(1-\eta_{u})\beta\right]\left[(\sigma^{u}\eta_{x})-\eta_{x}\right] \left[(\sigma^{u} \eta_{y})-\eta_{y}\right]\\
&=\eta_x L_N^r \eta_y + \eta_y L_N^r \eta_x.
\end{split}
\end{equation*}
The proof of the third expression is analogous.
\end{proof}
\end{lem}


\end{document}